\documentclass{article}

\usepackage{PRIMEarxiv}

\usepackage[dvipsnames]{xcolor}
\usepackage[utf8]{inputenc} 
\usepackage[T1]{fontenc}    
\usepackage{hyperref}       
\usepackage{url}            
\usepackage{doi}            
\usepackage{booktabs}       
\usepackage{amsmath}        
\usepackage{amsfonts}       
\usepackage{amssymb}        
\usepackage{amsthm}         
\usepackage{thmtools}       
\usepackage[fixamsmath,disallowspaces]{mathtools} 
\mathtoolsset{showonlyrefs=true}
\usepackage{newtxtext,newtxmath} 
\usepackage{nicefrac}       
\usepackage{microtype}      
\usepackage{lipsum}
\usepackage{fancyhdr}       
\usepackage{graphicx}       
\usepackage{bm}             
\graphicspath{{media/}}     
\usepackage{enumitem}
\usepackage[normalem]{ulem} 
\usepackage{colortbl}
\usepackage{scalerel}
\usepackage{tikz}
\usetikzlibrary{shapes.geometric,positioning}
\usepackage{multirow}

\usepackage{bm}

\usepackage{caption}             
\usepackage{subcaption}          

\usepackage{xspace}
\definecolor{amaranth}{rgb}{0.9, 0.17, 0.31}
\makeatletter
\newcommand{\todo}[1]{\marginpar{\tiny\color{amaranth}#1}\@latex@warning{#1}\xspace}
\makeatother

\usepackage[normalem]{ulem}

\colorlet{RobertColor}{MidnightBlue}
\colorlet{PhilippColor}{ForestGreen!50!OliveGreen}
 
\newcommand*{\mbb}[1]{\mathbb{#1}}
\newcommand*{\mcal}[1]{\mathcal{#1}}
\newcommand*{\mfrak}[1]{\mathfrak{#1}}
\newcommand*{\dd}{\ensuremath{\mathrm{d}}}
\newcommand*{\dx}[1][x]{\ensuremath{\,\dd{#1}}}

\newcommand{\PL}{\lambda}

\DeclareMathOperator{\supp}{supp}

\let\inf\relax  
\DeclareMathOperator*{\inf}{inf\vphantom{\sup}}

\DeclarePairedDelimiter{\pars}{\ensuremath{(}}{\ensuremath{)}}
\DeclarePairedDelimiter{\bracs}{\ensuremath{[}}{\ensuremath{]}}
\DeclarePairedDelimiter{\braces}{\ensuremath{\{}}{\ensuremath{\}}}

\DeclarePairedDelimiter{\floor}{\lfloor}{\rfloor}
\DeclarePairedDelimiter{\inner}{\langle}{\rangle}
\DeclarePairedDelimiter{\norm}{\|}{\|}
\DeclarePairedDelimiter{\abs}{\lvert}{\rvert}

\makeatletter
\newcommand{\opnorm}{\@ifstar\@opnorms\@opnorm}
\newcommand{\@opnorms}[1]{%
  \left|\mkern-1.5mu\left|\mkern-1.5mu\left|
   #1
  \right|\mkern-1.5mu\right|\mkern-1.5mu\right|
}
\newcommand{\@opnorm}[2][]{%
  \mathopen{#1|\mkern-1.5mu#1|\mkern-1.5mu#1|}
  #2
  \mathclose{#1|\mkern-1.5mu#1|\mkern-1.5mu#1|}
}
\makeatother

\mathchardef\mhyphen="2D

\let\oldbullet\bullet
\newlength{\raisebulletlen}
\setbox1=\hbox{$\bullet$}\setbox2=\hbox{\tiny$\bullet$}
\renewcommand\bullet{\raisebox{\raisebulletlen}{\,\tiny$\oldbullet$}\,}

\usepackage{accents}
\newlength{\dhatheight}
\newcommand{\doublehat}[1]{%
    \settoheight{\dhatheight}{\ensuremath{\hat{#1}}}%
    \addtolength{\dhatheight}{-0.25ex}%
    \hat{\vphantom{\rule{1pt}{\dhatheight}}%
    \smash{\hat{#1}}}}

\declaretheorem[numberwithin=section]{theorem}
\declaretheorem[sibling=theorem]{proposition}
\declaretheorem[sibling=theorem]{corollary}
\declaretheorem[sibling=theorem]{lemma}
\declaretheorem[sibling=theorem]{definition}
\declaretheorem[sibling=theorem]{remark}

    \definecolor{bleudefrance}{rgb}{0,0,0}
    \newcommand{\numRevisions}{1}
    \newcommand{\revision}[2][0]{%
    \begingroup%
        \newcount\colorRatio%
        \colorRatio=\numexpr(100*(#1+1))/\numRevisions\relax%
        \colorlet{revisionColor}{bleudefrance!\the\colorRatio!black}\color{revisionColor}#2%
    \endgroup}
    \newenvironment{revisione}[1][0]{\par%
        \begingroup%
        \newcount\colorRatio%
        \colorRatio=\numexpr(100* (#1+1))/\numRevisions\relax%
        \colorlet{revisionColor}{bleudefrance!\the\colorRatio!black}\color{revisionColor}
    }{\endgroup\par}
\newcommand{\ownTitle}{Optimal sampling for stochastic and natural gradient descent}

\newcommand{\ownShortTitle}{Optimal sampling for stochastic gradient descent}

\newcommand{\ownKeywords}{
risk minimisation,
optimisation on manifolds,
active learning,
optimally weighted least squares,
convergence rates
}

\newcommand{\ownAMS}{
49M15, 
93E24, 
68T07, 
90C30, 
60H30 
}

\hypersetup{
    breaklinks,
    raiselinks=true,
    colorlinks,
    citecolor=gray,
    filecolor=gray,
    linkcolor=gray,
    urlcolor=gray,
    bookmarksopenlevel=1,    
    bookmarksopen=true,      
    bookmarksnumbered=true,  
    bookmarkstype=toc,       
    pdfdisplaydoctitle=true, 
    pdfstartview=FitH,       
    pdfpagemode=UseOutlines, 
    pdfpagelayout=OneColumn, 
    pdftitle={\ownTitle},
}

\title{\ownTitle 
}

\author{
  Robert Gruhlke \\
  Fachbereich Mathematik und Informatik\\
  Freie Universit\"at Berlin\\
  Arnimallee 6, 14195 Berlin, Germany\\
  \texttt{r.gruhlke@fu-berlin.de} \\
  \And
  Anthony Nouy \\
  Centrale Nantes \& Nantes Universit\'e\\
  Laboratoire de Math\'ematiques Jean Leray,
  CNRS UMR 6629\\
  1 rue de la Noe, 44321 Nantes, France\\
  \texttt{anthony.nouy@ec-nantes.fr}\\
  \And
  Philipp Trunschke \\
  Centrale Nantes \& Nantes Universit\'e\\
  Laboratoire de Math\'ematiques Jean Leray,
  CNRS UMR 6629\\
  1 rue de la Noe, 44321 Nantes, France\\
  \texttt{philipp.trunschke@univ-nantes.fr}\\
}

\begin{document}
\maketitle

\begin{abstract}
    We consider the problem of optimising the expected value of a loss functional over a nonlinear model class of functions, assuming that we have only access to realisations of the gradient of the loss.
This is a classical task in statistics, machine learning and physics-informed machine learning.
A straightforward solution is to replace the exact objective with a Monte Carlo estimate before employing standard first-order methods like gradient descent, which yields the classical stochastic gradient descent method.
But replacing the true objective with an estimate ensues a ``generalisation error''.
Rigorous bounds for this error typically require strong compactness and Lipschitz continuity assumptions while providing a very slow decay with sample size.
\revision[0]{To alleviate these issues, we propose a version of natural gradient descent that is based on optimal sampling methods.}
Under classical assumptions on the loss and the nonlinear model class, we prove that this scheme converges almost surely monotonically to a stationary point of the true objective.
\revision[0]{Under Polyak-\L{}ojasiewicz-type conditions, this provides bounds for the generalisation error.}
As a remarkable result, we show that our stochastic optimisation scheme achieves the linear or exponential convergence rates of deterministic first order descent methods under suitable conditions.
\end{abstract}

\keywords{\ownKeywords}
\textit{\textbf{AMS subject classifications}}\ \ {\ownAMS}

\section{Introduction}

Let $\mathcal{X}$ be a set equipped with a  probability measure $\rho$ and $\mathcal{H}:=\mcal{H}(\mcal{X},\rho)$ be a Hilbert space of functions defined on $\mathcal{X}$, equipped with the norm $\|\cdot\|$ \revision[0]{which is induced by the inner product $\pars{\cdot, \cdot}$.}
Given a \emph{loss function} $\ell : \mathcal{H} \times \mathcal{X} \to \mathbb{R}$, we consider the optimisation problem
\begin{equation}
\label{eq:min_problem}
    \min_{v\in\mathcal{M}} \mathcal{L}(v),\qquad
    \mathcal{L}(v):= \int \ell(v; x) \,\mathrm{d}\rho(x) 
    = \mathbb{E}_{x\sim \rho}[\ell(v;x)]
\end{equation}
where $\mathcal{M}\subseteq\mathcal{H}$ is a possibly nonlinear \emph{model class} and $\mathcal{L}$ is the \emph{risk} or expected loss.
When computing the integral is infeasible, a common approach is to replace the exact risk with an \emph{empirical risk} (that is, a Monte Carlo estimate of the risk) before employing a standard optimisation scheme.
However, using an estimated risk instead of the true risk can result in a ``generalisation error''.
Rigorous bounds for this error usually require compactness of $\mathcal{M}$ and Lipschitz continuity of $\mathcal{L}$ and provide a very slow decay with increasing sample size~\cite{Cucker2007}.
This slow decay is unfavourable in settings where high accuracy is required or sample generation (or acquisition) is costly.
Moreover, a plain sampling from $\rho$ can lead to a bad condition number of the resulting empirical risk~\cite{cohen2017optimal}.
To address these issues, we propose a new iterative algorithm that performs successive corrections in local linearisations of $\mathcal{M}$.
We provide convergence rates in terms of the true loss $\mathcal{L}$, \revision[0]{proving almost sure convergence to stationary points} and providing bounds for the generalisation error \revision[0]{under Polyak--\L{}ojasiewicz(P\L{})-type conditions}.

Even under \revision[0]{convexity} assumptions on the objective $\mathcal{L}$, classical \emph{stochastic gradient descent} (SGD) methods converge at the relatively slow rate of at most $\mathcal{O}(t^{-1})$ with the number of steps $t$.
In contrast, under the same assumptions, deterministic gradient descent converges significantly faster, i.e.\ with exponential rate $\mathcal{O}(a^t)$ for some $a\in(0,1)$.
This, however, requires exact knowledge of the objectives gradients which may be impossible or prohibitively costly.
In this setting, the question naturally arises whether it is possible to find an optimisation scheme, that is based on empirical estimates of the gradient but comes close to the convergence rates of deterministic gradient descent.
The present manuscript provides an affirmative answer to this question, using the framework of active learning.

\subsection{Setting}
\label{sec:method}

Let $(\Omega, \Sigma, \mathbb{P}, (\mathcal{F}_t)_{t\geq 0})$ be a filtered probability space.
Given an $\mathcal{F}_0$-measurable initial guess $u_0$ we iteratively define
\begin{equation}
\label{eq:descent_scheme}
\begin{aligned}
    \bar{u}_{t+1}
    &:= u_t - s_t P_t^n g_t,
    \qquad
    g_t := \nabla \mathcal{L}(u_t), \\
    u_{t+1}
    &:= R_{t}(\bar{u}_{t+1}) .
\end{aligned}
\end{equation}
Here, we assume that the subsequent assumptions are satisfied at every step $t\in\mathbb{N}$.
\begin{enumerate}[label=(\roman*)]
    \item $\mathcal{L} : \mathcal{H}\to\mathbb{R}$ admits a Fréchet derivative $D\mathcal{L}(u_t)$ at $u_t$ and we can compute the gradient $\nabla \mathcal{L}(u_t)$ as the Riesz representative of $D\mathcal{L}(u_t)$ in $\mathcal{H}$.
    \item There exists a $\mcal{F}_t$-measurable linear space $\mathcal{T}_t$ of dimension $d_t$ ``approximating'' $\mathcal{M}$ locally around $u_t$.
    \item There exists a $\mathcal{F}_t$-measurable estimator $P_t^n$ of the $\mathcal{H}$-orthogonal projector $P_t\colon \mathcal{H}\to \mathcal{T}_t$.
    \item The step size $s_t>0$ is $\mathcal{F}_t$-measurable.
    \item There exists a $\mathcal{F}_t$-measurable retraction mapping $R_t : \mathcal{H} \to \mathcal{M}$.
\end{enumerate}
This procedure is illustrated graphically in Figure~\ref{fig:algo}. 

\begin{figure}[htb]
\begin{center}
    \begin{tikzpicture}
    
    \node (M) at (1.0,2.75) {$\mathcal{M}$};
    \draw[line width = 1pt]  (0.5,2.5) to[out=350,in=150] (2,2)  to[out=-30,in=0]  (-2,0);
    
    \draw[line width = 1pt,-] (-1.8,.-0.625) -- (6., 2.7);
    \node at (5.3, 2.75) {$\mathcal{T}_t$};
    
    \node[inner sep=1, outer sep=1] (u_t) at (1.5,0.8) {};
    \filldraw (u_t) circle (2pt);
    \node at (1.3,1.) {$u_t$};
    \node[inner sep=1, outer sep=1] (utpproj) at (3.45, 1.6) {};
    \filldraw (utpproj) circle (2pt);
    \node at (3.4, 1.9) {$\bar{u}_{t+1}$};
    \node[inner sep=1, outer sep=1] (utp) at (2.35, 1.6) {};
    \filldraw (utp) circle (2pt);
    \node at (1.9, 1.6) {$u_{t+1}$};
    \node[inner sep=1, outer sep=1] (utpexact) at (1.5 + 4.2*0.75, 0.9  -.3*0.75) {};
    \node at (5.4, 0.65) {$u_t - g_t$};
    \filldraw (utpexact) circle (2pt);
    \node[inner sep=1, outer sep=1] (utpn) at (1.12 + 4.2*0.75, 2.18  -.3*0.75) {};
    \filldraw (utpn) circle (2pt);
    \node at (5.35, 1.9) {$u_t - P^n_t g_t$};


    \draw[dashed, gray] (utp) -- (utpproj) -- (utpn) -- (utpexact) -- (u_t);





\end{tikzpicture}
\end{center}
\captionsetup{width=.8\linewidth}
\caption{Illustration of the proposed algorithm.
Starting from the iterate $u_t\in\mathcal{M}$, an approximation $P_t^n g_t \in \mathcal{T}_t$ of the true gradient $g_t$ is computed via a random projection $P_t^n$ onto the linear space $\mathcal{T}_t$.
Subsequently, an intermediate, linear update $\bar{u}_{t+1} = u_t - s_t P_t^n g_t$ is performed with step size $s_t$.
Finally, the next iterate $u_{t+1}\in\mathcal{M}$ is obtained through application of the retraction map $R_t$.}
\label{fig:algo}
\end{figure}

To give a concrete example, suppose that $\mathcal{H} := L^2(\mathcal{X})$ and $\mathcal{L}(v) := \tfrac12\norm{v - u^*}_{L^2(\rho)}^2$ is the least squares risk with respect to $\rho$.
Moreover, consider a parameterised model class $\mathcal{M} = \{F(\theta) : \theta \in \mathbb{R}^D\}$ with differentiable $F : \mathbb{R}^D \to \mathcal{H}$.
\begin{enumerate}[label=(\roman*)]
    \item In this setting $\nabla\mathcal{L}(u_t) = u_t - u^*$.
    \item The space $\mathcal{T}_t$ around $u_t = F(\theta_t) $ can be any subspace of $\operatorname{span}\braces{\partial_{k} F(\theta_t)\;:\; k=1,\hdots,D}$. 
    \item We can draw $x_t^1,\ldots,x_t^n\sim\rho$ i.i.d.\ and estimate $P_t g_t$ by a linear least squares projection using point evaluations of the gradient $g_t(x_t^1), \ldots, g_t(x_t^n)$.
    \item We can choose $s_t := t^{-3/4}$ satisfying the Robbins--Monro condition, \revision[0]{i.e.\ $(s_t)_t\in \ell^2\setminus\ell^1$.}
    \item Given $u_t = F(\theta_t)$ and $P_t^ng_t = \sum_{k=1}^D \vartheta_k \partial_kF(\theta_t)$, for some $\vartheta\in\mathbb{R}^D$, we can define $R_t(u_t + s_t P_t^ng_t) = F(\theta_t + s_t\vartheta)$ (see section~\ref{sec:shallow} for further discussion).
    Alternatively, when $\mathcal{M}$ is a manifold, we can use a retraction in the classical differential-geometric sense (cf.~\cite[Section~3.6]{Boumal2023-yt}).
\end{enumerate}
Of course, convergence can only be guaranteed when several other assumptions, like bounded variance of the estimates, or a controlled error of the retraction step $u_{t+1} := R_t(\bar{u}_{t+1})$, are satisfied.
These assumptions are discussed in section~\ref{sec:assumption}.

\revision[0]{
Throughout the document, we make use of the notation $\lesssim$, referring to ``bounded from above up to a constant'', and $\asymp$, referring to ``bounded from below and above up to constants''.
These shorthand notations correspond to the Bachmann–Landau symbols $\mathcal{O}$ and $\Theta$, respectively.
}

\revision[0]{
\begin{remark}
   The discussed setup includes the setting where $\mathcal{M}$ is a Riemannian manifold.
   However, the setup is more general and allows an application to the optimization of artificial neural networks, which in general do not form manifolds.
   For this we consider the concept of the linear space $\mathcal{T}_t$ with flexible dimension $d_t$, which in the case of a smooth Riemannian manifold would take the role of the tangent space with constant dimension $d_t\equiv d\in\mathbb{N}$ if $\mathcal{M}$ has finite dimension.
\end{remark}
}

\subsection{Contributions}

\revision[0]{This work analyses the influence of the sampling density on the convergence of multiple stochastic gradient descent-type algorithms.}
Equation~\eqref{eq:descent_scheme} provides a general framework for analysing algorithms such as \emph{stochastic gradient descent} (SGD) and \emph{natural gradient descent} (NGD), but differs from many classical approaches in two ways.
\begin{enumerate}
    \item We define the risk as a functional on a Hilbert space $\mcal{H}$ and consider the optimisation not over a parameter space $\mathbb{R}^D$ but the corresponding model class $\mathcal{M} = \{F(\theta) : \theta \in \mathbb{R}^D\} \subseteq \mathcal{H}$.
    This brings the advantage that Lipschitz smoothness and (strong) convexity parameters of the function $\mcal{L}$ are often known exactly while
    the corresponding bounds for the combined parameter-to-risk map $\theta\mapsto \mcal{L}(F(\theta))$ are prohibitively large or even inaccessible.
    Direct access to the smoothness and convexity parameters also means that hyper-parameters, like step size, are easy to choose and don't have to be tuned algorithmically.

    \item We assume that we have control over the sample generation process.
    Using a particular importance sampling strategy, known as \emph{optimal sampling}, ensures that the gradient is estimated with bounded variance.
    This results in faster and more monotone convergence, as illustrated in Figure~\ref{fig:gaussian_convergence} in section~\ref{sec:examples}. 
\end{enumerate}

We analyse the convergence of algorithm~\eqref{eq:descent_scheme} to a minimiser or stationary point of the true risk $\mcal{L}$ (not an empirical estimate) over the model class $\mcal{M}$, with the potential of yielding theoretical bounds for the generalisation error. 
In particular, we provide (to the best of our knowledge) the first proof of convergence for NGD in this general setting.
The convergence rates of our algorithms in different settings are presented in~Table~\ref{tab:rates}.
\revision[0]{In the unbiased case, our proposed \emph{stochastic gradient descent–type} algorithm achieves convergence rates comparable to deterministic first-order methods. In contrast, for biased stochastic descent methods, convergence to a stationary point is generally impossible; only convergence to a neighborhood can be expected, and we provide the corresponding analysis for our descent scheme.}

We do not consider line search methods for adaptive step size selection since this would require estimating the risk $\mcal{L}$ with certified error bounds.
This would command an a priori variance bound, which seems an implausible assumption in our setting.
An example of such a strategy, under the strong assumption that function values and gradients are estimated with almost surely uniformly bounded errors, is presented in~\cite{sun2022trust}.

\begin{table}
    \begin{tabular}{cccccc}
    \toprule
      & \phantom{a} &\multicolumn{4}{c}{\textbf{rates}} \\
        \cmidrule{3-6}
        \textbf{assumptions}
            &
            & GD
            & best-case (ours)
            & worst-case (ours)
            & SGD
        \\
        \midrule
            \eqref{eq:lower_bound} + \eqref{eq:L-smooth} + \eqref{eq:C-retraction}
            &
            & $\mathcal{O}(t^{-1})$
            & $\mathcal{O}(t^{-1+\varepsilon})$ \textsuperscript{[eq.~\eqref{eq:rate_recovery_BLSCR}]}
            & $\mathcal{O}(t^{-1/2+\varepsilon})$ \textsuperscript{[eq.~\eqref{eq:rate_nonrecovery_BLSCR}]}
            & $\mathcal{O}(t^{-1/2+\varepsilon})$
        \\
            \eqref{eq:L-smooth} + \eqref{eq:mu-PL_strong} + \eqref{eq:C-retraction}
            &
            & $\mathcal{O}(a^t)$
            & $\mathcal{O}(a^t)$ \textsuperscript{[eq.~\eqref{eq:rate_recovery_LSSPLCR}]}
            & $\mathcal{O}(t^{-1+2\varepsilon})$ \textsuperscript{[eq.~\eqref{eq:rate_nonrecovery_LSSPLCR}]}
            & $\mathcal{O}(t^{-1+2\varepsilon})$
        \\
    \bottomrule
    \end{tabular}
    \centering
    \captionsetup{width=.8\linewidth}
    \caption{Almost sure convergence rates for different algorithms with $\varepsilon\in\pars{0, \tfrac{1}{2}}$ and $a\in\pars{0,1}$ depending on the chosen step size. 
    The required assumptions can be found in the linked references and equations.
    The convergence rates for GD and SGD can be found in~\cite{Bubeck2015,Boumal2023-yt,liu2022sure,garrigos2024handbook}.}
    \label{tab:rates}
\end{table}

\subsection{Related work}

The classical analysis of SGD algorithms in the case $\mathcal{M}=\mathcal{H}=\mathbb{R}^d, d\in\mathbb{N}$, relies on the \emph{ABC condition},
$$
    \mbb{E}\bracs*{\norm{\widehat {\nabla \mcal{L}(u)}}^2}
    \le A(\mcal{L}(u) - \mcal{L}_{\mathrm{min}}) + B\norm{\nabla \mcal{L}(u)}^2 + C,
$$
where $\widehat {\nabla \mcal{L}(u)}$ is an unbiased estimate of the gradient $\nabla \mcal{L}(u)$ and $\mcal{L}_{\mathrm{min}}$ is the infimum of $\mcal{L}$ on $\mathcal{H}$.
This condition is 
presented in~\cite{khaled2020better} as ``the weakest assumption'' for analysing SGD in the non-convex setting.
The parameter $C$ in the ABC condition
captures the variance of the stochastic gradient and critically influences the convergence rate.
In the best case ($C=0$), the resulting algorithm exhibits the same convergence rate as classical \emph{gradient descent} (GD).
In the typical case ($C>0$), however, the algorithm follows the classical SGD rates.
Our work can be seen as a \revision[0]{complementary} analysis of SGD, assuming $A=0$ and focusing on the constants $B$ and $C$ of the ABC condition.

SGD can also be interpreted as a first-order optimisation method with an inexact, stochastic oracle.
For deterministic perturbations, such methods have been analysed extensively, and analysis for the special case of smooth and convex cost functions can, for example, be found in~\cite{Devolder2013}.
Stochastic perturbations have been considered, e.g.\ in~\cite{Dvinskikh2020,Sadiev2024,Khaled2023}.
\revision[0]{
While our framework covers the setup beyond the manifold setup, we note that Riemannian stochastic approximation schemes with fixed step size have already been investigated in~\cite{pmlr-v130-durmus21a}.
}

When expressed in the parameter space $\mathbb{R}^D$ of a nonlinear model class $\mathcal{M} = \braces{F(\theta) \;:\; \theta\in\mathbb{R}^D}$, the optimisation scheme~\eqref{eq:descent_scheme} can be seen as a \emph{natural gradient descent} (NGD)~\cite{martens2020ngd,nurbekyan2023efficient,mueller2023NGD_PINNs,NEURIPS2018_7f018eb7}.
\revision[0]{Under suitable conditions on $\mathcal{L}$ and $F$ and with an appropriate choice of $\mathcal{H}$, NGD is a Newton-like algorithm.}
Our work weakens several of the assumptions and puts a spotlight on the sampling.
The idea of choosing problem-adapted samples is well-known as importance sampling and has been explored for learning neural networks~\cite{Bruna2024Jan,adcock2022cas4dl,adcock2023cs4ml}.
However, in general, there is no efficient importance sampling strategy for estimating a risk functional simultaneously for all functions from a highly nonlinear set $\mathcal{M}$ (see, e.g., the negative results for low-rank tensor networks in \cite{eigel2022convergence,trunschke2023weighted}).
For this reason, we adopt an adaptive approach using successive updates in linear spaces, for which optimal sampling distributions have been rediscovered recently~\cite{cohen2017optimal} and can be computed explicitly.

In the special case of linear least squares approximation,
recent years have seen a lot of active developments in the field of optimal weighted least squares projections~\cite{cohen2017optimal,cohen2020optimal}.
Using a suitable least squares projection $P_t^n$ as an estimate for $P_t$ allows to satisfy a  quasi-optimality property in expectation~\cite{haberstich2022boosted}, i.e.\ for all $v\in \mcal{H}$,
\begin{align}
    \mbb{E}\bracs*{{\norm{v - {P}^n_t v}^2} }
    &\le \pars*{1 + c\tfrac{d_t}{n}}{\norm{v - P_t v}^2}
\end{align}
for some constant $c>0$.
An alternative to the least squares projection that ensures that $P_t^n v$ is an unbiased estimate of the orthogonal projection $P_t v$ for any $v \in \mathcal{H}$ is the \textit{quasi-projection}.
As a consequence of Lemma~\ref{lem:quasi-projection_norm}, this estimator satisfies   
\begin{align}
    \mbb{E}\bracs*{\norm{v -  P_t^n v}^2} \le (1+\tfrac{d_t} n)\norm{v - P_t v}^2 + \tfrac{d_t-1}n \norm{P_t v}^2 .
\end{align}
Since the quasi-projection does not guarantee quasi-optimality, the reader may be surprised that quasi-projections yield comparable theoretical results to least squares projections in Corollary~\ref{cor:convergence_up_to_precision}.
The flaw in this argument lies in comparing only a single step of the algorithm.
Multiple cheap quasi-projection steps can achieve the same convergence rate as a single expensive least squares projection.
We finally note, that using a least squares projection based on volume-rescaled sampling  \cite{derezinski2022unbiased,nouy2023dpp} presents both advantages, i.e.\ the corresponding empirical projection is unbiased and achieves quasi-optimality in expectation.
This, however, comes at the price of a challenging sampling problem, which is still an active field of research.

The present manuscript follows the same line of thought as the two recent works~\cite{adcock2022cas4dl,adcock2023cs4ml}, where the practical relevance of optimal sampling in machine learning is demonstrated.
Our theoretical results indicate that it is beneficial to extend the works mentioned above to cure some theoretical illnesses in learning with neural networks:
\begin{enumerate}
    \item Both works linearise the network only in the last layer, resulting in samples that are optimal only for updating this last layer.
    However, changes in earlier layers are even more significant since errors in these layers are amplified by all subsequent layers.
    \item Both works use a random sample for the orthogonalisation step.
    This was already proposed in~\cite{cohen2020optimal}, where it was also discussed that the worst-case bounds for the number of samples needed for a stable orthogonalisation may be infeasible.
    This may arguably happen for a model class as complicated as neural networks.\footnote{
    On the one hand, uniform samples can not guarantee a stable orthogonalisation without the curse of dimensionality~\cite{cohen2020optimal}.
    On the other hand, the exact orthogonalisation of the layers of an arbitrary neural network is NP-hard~\cite{9022212}.
    This means that we have to estimate the inner product, but we must be very meticulous about choosing the samples that we use.
    }
    \item The samples are not drawn according to the optimal distribution but are subsampled from a larger i.i.d.\ sample from $\rho$.
    The approximation with respect to this subsample can thus not be better than the approximation using the full sample.
\end{enumerate}

\revision[0]{
We finally note that drawing samples from a general optimal sampling distribution remains an important open problem, since these densities can be non-smooth and are typically extremely multimodal.
The works~\cite{adcock2022cas4dl,adcock2023cs4ml} demonstrate impressive performance for the model class of neural networks, yet, their sampling method has no theoretical guarantee (see~\cite{trunschke2024optimalsamplingsquaresapproximation} for a detailed discussion).
For low-dimensional linear spaces an algorithm with bounded complexity is provided in~\cite{cohen2017optimal} and improved in~\cite{cohen2020optimal,Arras2019hierarchy}.
This algorithm can be generalised to sets of sparse vectors (see~\cite{trunschke2023convergence} for a formula for the density, see also~\cite{Adcock2022book}).
If the sampling distribution can be expressed as a tensor network with bounded rank, an efficient sampling algorithm is provided in~\cite{Dolgov2019density}.
This is, e.g., the case when optimising tree tensor networks~\cite{haberstich2023active}.
For general model classes, the optimal sampling density can be approximated by covering arguments~\cite{trunschke2023convergence} or estimated recursively~\cite{trunschke2024optimalsamplingsquaresapproximation,herremans2025refinementbasedchristoffelsamplingsquares}.
}

\subsection{Outline}

The remainder of this work is organised as follows.
Section~\ref{sec:assumption} summarises and discusses the assumptions for our propositions.
Examples for different estimators of the orthogonal projection are presented in section~\ref{sec:setting}.
We state our convergence theory in section~\ref{sec:Convergence_theorey}, where section~\ref{sec:Convergence_expectation} covers convergence in expectation and section~\ref{sec:Convergence_as} covers almost sure convergence.
Finally, section~\ref{sec:examples} discusses the assumptions and illustrates the theoretical results for linear spaces and shallow neural networks.
\section{Assumptions}
\label{sec:assumption}

To analyse the convergence of the presented algorithm, we define
$$
    \mcal{L}_{\mathrm{min}}:=\inf\limits_{u\in\mathcal{H}} \mathcal{L}(u) ,
    \qquad\text{and}\qquad
    \mcal{L}_{\mathrm{min},\mcal{M}}:=\inf\limits_{u\in\mathcal{M}} \mathcal{L}(u) .
$$
\begin{revisione}[0]{
For a given iterate $u_t\in \mcal{M}$, let $\mcal{T}_t$ be the linear space defined by the algorithm at the point $u_t$, $P_t$ be the $\mcal{H}$-orthogonal projection onto $\mcal{T}_t$, $P_t^n$ an empirical estimator of $P_t$ and $R_t$ be the retraction map at $u_t$.
We assume, depending on the concrete setting, that some of the following properties are satisfied.
\begin{itemize}
    \item \emph{Boundedness from below:} It holds that $\mcal{L}_{\mathrm{min},\mcal{M}}\in\mathbb{R}$ is finite and thus for all $u\in\mcal{M}$
    \begin{equation} \label{eq:lower_bound}
        \mcal{L}_{\mathrm{min},\mcal{M}} \le \mcal{L}\pars{u} . \tag{B}
    \end{equation}
    \item \emph{(Restricted) $L$-smoothness:} There exists $L > 0$ such that for all $u_t\in\mcal{M}$ with associated $\mcal{T}_t$ and $g\in\mcal{T}_t$
    \begin{equation} \label{eq:L-smooth}
        \mcal{L}\pars{u_t + g} \le \mcal{L}\pars{u_t} + \pars{\nabla \mcal{L}\pars{u_t}, g} + \frac{L}{2}\norm{g}^2 . \tag{LS}
    \end{equation}
    \item \emph{$\lambda$-Polyak-\L{}ojasiewicz}: There exists $\PL>0$ such that for all $u\in\mcal{M}$ 
    \begin{equation} \label{eq:mu-PL}
        \norm{\nabla \mcal{L}(u)}^2 \ge 2\PL \pars{\mcal{L}(u) - \mcal{L}_{\mathrm{min}}} . \tag{PL}
    \end{equation}
    \item \emph{Strong $\lambda$-Polyak-\L{}ojasiewicz}: There exists $\PL>0$ such that for all $u_t\in\mcal{M}$ with associated $\mcal{T}_t$ it holds
    \begin{equation} \label{eq:mu-PL_strong}
        \norm{P_t \nabla \mcal{L}(u)}^2 \ge 2\PL \pars{\mcal{L}(u) - \mcal{L}_{\mathrm{min},\mathcal{M}}} . \tag{SPL}
    \end{equation}

    \item \emph{Controlled retraction error:}
    There exists a constant $C_{\mathrm{R}}>0$ such that for any $u_t\in\mcal{M}$, with associated space $\mcal{T}_t$, and all $g\in\mcal{T}_t$ and all $\beta > 0$, we have a retraction $R_t : \mcal{H} \to \mcal{M}$ satisfying
    \begin{equation} \label{eq:C-retraction}
        \mcal{L}\pars{R_t(u_t + g)} \le \mcal{L}\pars{u + g} + \frac{C_{\mathrm{R}}}{2}\norm{g}^2 + \beta .
        \tag{CR}
    \end{equation}
  
    \item \emph{Bounded bias and variance of $P_t^n$:} 
    There exist positive constants $c_{\mathrm{bias}, 1}, c_{\mathrm{bias},2}, c_{\mathrm{var},1}$ and $c_{\mathrm{var},2}$ such that for all $u_t \in \mathcal{M}$, with associated  space $\mcal{T}_t$, 
    and any $g \in \mcal{H}$, it holds that
    \begin{equation} \label{eq:bbv}
    \begin{aligned}
        \mbb{E}\bracs*{\pars{g, P_t^n g} \mid \mathcal{F}_t} &\ge c_{\mathrm{bias}, 1}\norm{P_t g}^2 - c_{\mathrm{bias},2} \norm{P_t g}\norm{(I-P_t)g} \\ 
        \mbb{E}\bracs*{\norm{P_t^n g}^2 \mid \mathcal{F}_t} &\le c_{\mathrm{var},1} \norm{P_tg}^2 + c_{\mathrm{var},2} \norm{(I-P_t)g}^2.
    \end{aligned}
    \tag{BBV}
    \end{equation}
\end{itemize}
}
\end{revisione}

Assumptions~\eqref{eq:lower_bound},~\eqref{eq:L-smooth} and~\eqref{eq:mu-PL} are standard assumptions from optimisation theory, with assumption~\eqref{eq:lower_bound} being a necessary condition for the existence of a minimum.
The $L$-smoothness~\eqref{eq:L-smooth} and $\lambda$-Polyak-\L{}ojasiewicz~\eqref{eq:mu-PL} conditions (cf.~\cite{karimi2016linear}) are standard assumptions to show linear convergence in the first-order optimisation setting.
When $\mathcal{M} = \mathcal{H}$, condition~\eqref{eq:mu-PL} is implied by strong convexity of $\mathcal{L}$ but is strictly weaker, allowing for multiple global minimisers as long as any stationary point is a global minimiser.
Strong convexity, in turn, is often considered because every function is strongly convex locally around its non-degenerate minima, if those exist.
This means that related theorems give at least convergence results in the neighbourhood of such minima.

Condition~\eqref{eq:mu-PL_strong} is a stronger version of~\eqref{eq:mu-PL} and simplifies the analysis in this manuscript. \revision[0]{We note the different dependence on minimal loss values $\mathcal{L}_{\mathrm{min}}$ and $\mathcal{L}_{\mathrm{min},\mathcal{M}}$ for~\eqref{eq:mu-PL} and~\eqref{eq:mu-PL_strong}, respectively.}

\revision[0]{
Our generalized framework includes the manifold setup including various  applications of geodesically convex optimisation as provided in~\cite[Chapter~11]{Boumal2023-yt}.
In particular, we derive a sufficient condition for \eqref{eq:mu-PL_strong} in the special case of $\mathcal{M}$ being a Riemannian submanifold of $\mathcal{H}$ in Lemma~\ref{A3:Lemma_SPL} of Appendix~\ref{sec:A3}. This result generalises the relation between~\eqref{eq:mu-PL} and strong convexity to the case of manifolds. }

\revision[0]{A discussion of assumption~\eqref{eq:C-retraction} related to the retraction can be found in appendix~\ref{app:retraction}.}

The condition~\eqref{eq:bbv} is natural in the discussion of stochastic gradient descent type algorithms and characterises the bias and variance of the empirical (quasi-)projection operators discussed in section~\ref{sec:setting}.

We discuss the validity of these assumptions on three examples in section~\ref{sec:examples}.
\section{Estimators of orthogonal projections}
\label{sec:setting}
This section is devoted to the explicit construction of estimators $ P_t^n$ of the $\mcal{H}$-orthogonal projection $P_t\colon\mcal{H}\to \mathcal{T}_t$ onto a given space $\mathcal{T}_t$ at step $t$ of our algorithm. 
\revision[0]{We assume that the inner product of the Hilbert space $\mathcal{H} = \mcal{H}(\mcal{X},\rho)$ takes the form}
\begin{equation}
\label{eq:norm}
    (u,v) = \int (L_x u)^\intercal (L_x v) \dx[\rho\pars{x}] ,
\end{equation}
for a suitable family of linear operators $\braces{L_x : \mcal{H} \to \mbb{R}^l}_{x\in \mcal{X}}$, $l\in\mbb{N}$.
Two prominent examples of this setting are the following.
\begin{enumerate}
    \item[(i)] The Lebesgue space $L^2(\mcal{X},\rho)$ corresponds to the Hilbert space $\mcal{H}$ with the choice  $L_x v := v(x)$.
    \item[(ii)] The Sobolev space $H^1(\mcal{X},\rho)$ corresponds to the Hilbert space $\mcal{H}$ with the choice $L_x v := \begin{pmatrix} v(x) \\ \nabla v(x) \end{pmatrix}$.
\end{enumerate}
Given an orthonormal basis $\braces{b_k}_{k=1}^{d_t}$ of $\mcal{T}_t$, the projection $P_tg_t$ of $g_t\in\mcal{H}$ onto the subspace $\mcal{T}_t$  can be written as
\begin{equation}
\label{eq:H_projection}
    P_t g_t = \sum_{k=1}^{d_t} \eta_k b_k
    \qquad\text{with}\qquad
    \eta_k := (g_t, b_k) .
\end{equation}

To estimate the integral in~\eqref{eq:norm}, we introduce suitable, $\mathcal{F}_t$-dependent sampling measures $\mu_t$ and weight functions $w_t$ such that
\begin{equation}
\label{eq:measure_equality}
    w_t\dx[\mu_t] = \dx[\rho]
    \qquad\text{on}\qquad
    \supp(\mcal{T}_t) := \bigcup_{v\in\mcal{T}_t }\supp( \Vert L_{\bullet}v \Vert) .
\end{equation}
An unbiased estimate of the inner product $\pars{u, v}$ can then be defined as 
$$
    \pars{u, v}_n
    := \frac1n\sum_{i=1}^n w_t(x_t^i) \pars{L_{x_t^i} u}^\intercal \pars{L_{x_t^i} v},
$$
where the $x_t^1, \ldots, x_t^n$ are i.i.d.\ samples from $\mu_t$.
\revision[0]{
Finally, we define for an $\mathcal{H}$-orthonormal basis $\braces{b_k}_{k=1}^{d_t}$ of $\mcal{T}_t$ 
\begin{equation}
\label{eq:def:christoffel}
    \mfrak{K}_t(x) := \lambda^*\pars*{\sum_{k=1}^{d} \pars{L_{x}b_k}\pars{L_xb_k}^\intercal}
    \qquad\text{and}\qquad
    k_t := \norm{w \mfrak{K}_t}_{L^\infty(\rho)} .
\end{equation}
}

We now present three commonly used estimators $P_t^n$ of $P_t$, as well as their corresponding bias and variance bounds from assumption~\eqref{eq:bbv}.
Two additional estimators are discussed in appendices~\ref{sec:dpp-projection} and~\ref{app:debiased_projection}, respectively.

\begin{table}[hbt]
\caption{\revision[0]{Overview of bias and variance constants for different estimators $P_t^n$.
LSP stands for Least squares projection, VS for volume sampling. The appearing entities giving for each bias and variance are defined in the related section.
}}
\label{table:constant_overview}
\centering
\renewcommand{\arraystretch}{1.5}
\begin{tabular}{llccccc}
    \toprule
        Estimator & Def. & $c_{\mathrm{bias},1}$ & $c_{\mathrm{bias},2}$ & $c_{\mathrm{var},1}$ & $c_{\mathrm{var},2}$ & Proof \\
    \midrule
        Non-proj. & \ref{sec:non-projection} &  $\lambda_{\ast}(G)$ & $0$ & $\frac{\lambda^*(G)^2(n-1) + \lambda^*(G) k_t}{n}$ & $\frac{\lambda^*(G) k_t}n$ &
        \ref{lem:sgd_bias_variance}\\
        Quasi-proj. & 
        \ref{sec:quasi-projection} &
        $1$ & $0$ & $\frac{n-1+k_t}n$ & $\tfrac{k_t}{n}$ & \ref{lem:sgd_bias_variance} \\
        LSP & \ref{sec:least-squares_projection} & $1$ & $\tfrac{\sqrt{k_t}}{(1-\delta)\sqrt{p_{\mcal{S}_\delta}n}} $ & 
        $\tfrac{1}{(1-\delta)^2p_{\mcal{S}_\delta}}\tfrac{n-1+k_t}{n}$ & $\tfrac{1}{(1-\delta)^2p_{\mcal{S}_\delta}}\tfrac{k_t}{n} $ &
        \ref{lemma:QBP_bias}\\
        LSP with VS & \ref{sec:dpp-projection} & $1$ & $0$ & $c_{\mathrm{var},2} + \frac{4}{(1-\delta)^2}  $ & $\frac{d_t}{n}\left(1+ \frac{4}{(1-\delta)^2}\right)$ &
        \ref{lemma:DPP_bias} \\
        Debiased proj. & 
        \ref{app:debiased_projection} & $1$ & $0$ & $1$ & 
        \parbox[t]{8em}{$\begin{multlined}
            (c_{\mathrm{var},1}(Q_t)-1)B^2 \\
            + c_{\mathrm{var},2}(Q_t)
        \end{multlined}$} & \ref{lem:debiased_P_bounded} \\
    \bottomrule
\end{tabular}
\end{table}

\subsection{Non-projection}
\label{sec:non-projection}

Assume that $\mathcal{B}_t := \braces{\varphi_k}_{k=1}^{d}$ with $d\ge d_t$ forms a generating system of $\mcal{T}_t$.
Then an approximation of $g\in\mcal{H}$ with respect to the system $\mathcal{B}_t$ is given by
\begin{equation}
\label{eq:non-projection}
    P_t^n g := \sum_{k=1}^d \hat\zeta_k \varphi_k
    \qquad\text{with}\qquad
    \hat\zeta_k := \pars{g, \varphi_k}_n .
\end{equation}
It is easy to see that \textbf{$\boldsymbol{P_t^n}$ is not a projection}. 
Moreover, when $\mcal{B}_t$ is not an orthonormal basis, \textbf{$\boldsymbol{P_t^n}$ is not an unbiased estimate of $\boldsymbol{P_t}$}.
It is nevertheless important to study this operator because its use in~\eqref{eq:descent_scheme} yields the standard SGD method when $\mcal{H} = L^2(\mathcal{X},\rho)$ and $\ell(v;x) = \tilde \ell(v(x); x)$ with $\tilde \ell(\cdot ; x) : \mathbb{R} \to \mathbb{R}$ satisfies certain regularity assumptions.
\revision[0]{In this case the functional gradient of the loss coincides with the pointwise derivative of the scalar loss, and composing this gradient with a parametrization map yields the usual parameter gradient. Approximating the expectation by samples then leads exactly to a SGD update.
The formal argument, based on Leibniz’ rule and the chain rule, is provided in Appendix~\ref{app:leibniz}.}

\revision[0]{
We report the associated bias and variance constants from~\eqref{eq:bbv} in Table~\ref{table:constant_overview}.
There we see their dependence on the Gramian matrix $G_{k\ell} := (\varphi_k,\varphi_{\ell})$, implying that the system $\mcal{B}_t$ has to be chosen meticulously to ensure that $c_{\mathrm{bias},1}$ remains bounded away from zero and that $c_{\mathrm{var},1}$ and $c_{\mathrm{var},2}$ remain small.
}

\subsection{Quasi-projection}
\label{sec:quasi-projection}

To remove the dependence on \revision[0]{the gramian $G$ defined in the previous section} entirely, we assume that $\mathcal{B}_t = \braces{b_k}_{k=1}^{d_t}$ forms an $\mcal{H}$-orthonormal basis of $\mcal{T}_t$, i.e.\ $G_{kl} := (b_k, b_l) = \delta_{kl}$. \revision[0]{Then, $\lambda_{\ast}(G)=\lambda^{\ast}(G)=1$.}
The \emph{quasi-projection} of $g\in\mcal{H}$ with respect to the basis $\mathcal{B}_t$ is given by
\begin{equation}
\label{eq:quasi-projection}
    P_t^n g := \sum_{k=1}^{d_t} \hat\eta_k b_k
    \qquad\text{with}\qquad
    \hat\eta_k := \pars{g, b_k}_n .
\end{equation}
Although \textbf{$\boldsymbol{P_t^n}$ is not a projection},  it is an unbiased estimator of the real projection $P_t$ from ~\eqref{eq:H_projection}, \revision[0]{i.e. $c_{\mathrm{bias},2}=0$
as shown in Table~\ref{table:constant_overview}}.

\revision[0]{
\begin{remark}[Quasi-projection yields natural gradient descent]
Using the quasi-projection in~\eqref{eq:descent_scheme} leads to a natural gradient descent update.
Intuitively, the functional gradient is first expressed in the tangent space of the model manifold and then reweighted by the inverse Gram matrix induced by the parametrization, which corresponds to preconditioning the parameter gradient by the local geometry.
A brief derivation is given in Appendix~\ref{app:quasi_projection_relevance}.
\end{remark}
}

Even though $P_t^n g$ depends on the choice of $\mcal{B}_t$, the \revision[0]{associated bias and variance constants in Table~\ref{table:constant_overview}} hold for every orthonormal basis $\mcal{B}_t$.
\revision[0]{Although the constants in Table~\ref{table:constant_overview} no longer depend on $G$, the bounds still depend on the constant $k_t$ defined in~\eqref{eq:def:christoffel}.}
This constant depends on the space $\mcal{T}_t$ and may be unbounded.
Consider, for example, a space $\mcal{T}_t$ of univariate polynomials with degree $d_t-1$ in $ \mcal{H} = L^2\pars{\rho}$.
Then, we have the following two bounds:
\begin{enumerate}
    \item[(a)] $k_t = d_t^2$, if $w_t\equiv 1$ and $\rho$ is the uniform measure on $\mcal{X} := [-1, 1]$,
    \item[(b)] $k_t = \infty$, if $w_t\equiv 1$ and $\rho$ is the standard Gaussian measure on $\mcal{X} := \mbb{R}$.
\end{enumerate}

This demonstrates that it is crucial to control $k_t$ during the optimisation, which requires choosing a suitable sampling measure $\mu_t$ or corresponding weight function $w_t$.
Theorem~3.1 in~\cite{trunschke2023convergence} shows that the weight function $w_t$ minimising $k_t := \norm{w_t\mfrak{K}_t}_{L^\infty(\rho)}$ is
$$
    w_t = \norm{\mfrak{K}_t}_{L^1(\rho)} \mfrak{K}_t^{-1} .
$$
A discussion of $\mfrak{K}_t$ and an easier-to-compute surrogate $\tilde{\mfrak{K}}_t$ is provided in appendix~\ref{app:optimal_sampling}.
Both choices, $w_t = \norm{\mfrak{K}_t}_{L^1(\rho)} \mfrak{K}_t^{-1}$ and $w_t = \norm{\tilde{\mfrak{K}}_t}_{L^1(\rho)} \tilde{\mfrak{K}}_t^{-1}$, guarantee $k_t \le d_t$ and coincide when $l=1$.
The corresponding sampling is referred to as \emph{generalised Christoffel sampling}.

\begin{remark}
\label{rmk:complexity}
    Computing a quasi-projection with optimal sampling requires ${O}(d_t^3)$ operations and is significantly more expensive than a simple SGD step, requiring ${O}(d_t)$ operations.
    However, we will see that the convergence rate of SGD depends on the condition number of the Gramian matrix $\kappa(G)$ and on $k_t$, which may both grow uncontrollably during optimisation.
    Controlling these two constants is thus necessary to derive convergence rates with respect to the true risk $\mcal{L}$.
\end{remark}

\begin{remark}(Gradient-dependent importance sampling) 
\label{rmk:optimal_sampling}
    The proof of Lemma~\ref{lem:sgd_bias_variance} relies on the bound
    \begin{equation}
    \label{eq:tigther_variance_bound}
        \frac{1}{n}\int w_t(x)\mfrak{K}_t(x)\norm{L_xg}_2^2\dx[\rho(x)]
        \le \tfrac{1}{n}\norm{w_t\mfrak{K}_t}_{L^\infty(\rho)} \norm{g}^2
        = \tfrac{k_t}{n} \norm{g}^2 .
    \end{equation}
    Choosing $w_t$ to minimise the upper bound $k_t = \Vert w_t\mfrak{K}_t\Vert_{L^\infty(\rho)}$ instead of the left-hand side in~\eqref{eq:tigther_variance_bound} is thus not optimal and ignores the relation of $w_t$ to the current gradient $g$.
    A notable example follows from the bound
    $$
        \frac{1}{n}\int w_t(x)\mfrak{K}_t(x)\norm{L_xg}_2^2\dx[\rho(x)]
        \le \frac{1}{n}\int \mfrak{K}_t(x) \dx[\rho(x)] 
        \sup_{x\in\mcal{X}} w(x)\norm{L_xg}_2^2
        \le \frac{d_t}{n} \sup_{x\in\mcal{X}} w_t(x)\norm{L_xg}_2^2 .
    $$
    In particular, if the model class is compact and $\norm{L_x g}_2 \le B\norm{g}$, the bias terms remain bounded even for $w_t\equiv 1$.
    This is essentially the ABC condition with $A=C=0$ and shows that stochastic gradient descent can converge without optimal sampling under strong (implausible) assumptions on the gradients.
\end{remark}

\subsection{Least squares projection}
\label{sec:least-squares_projection}

Although the quasi-projection exhibits uniformly bounded \eqref{eq:bbv} constants, it has two disadvantages: (i) it requires knowledge of an orthonormal basis $\mathcal{B}_t$, and (ii) it is not a projection and therefore does not exhibit quasi-optimality properties.
To tackle these problems, we define the empirical Gramian $\hat{G}_{kl} := (\varphi_k, \varphi_l)_n$ as well as $\hat\eta_k := (g, \varphi_k)_n$.
Then the least squares projection of $g\in\mcal{H}$ onto $\mcal{T}_t$ is given by
\begin{equation}
\label{eq:least-squares-projection}
    P_t^n g := \sum_{k=1}^d \doublehat\eta_k \varphi_k
    \qquad\text{with}\qquad
    \doublehat\eta := \hat{G}^+ \hat\eta .
\end{equation}
If $\varphi_k = b_k$ is an $\mathcal{H}$-orthonormal basis of $\mathcal{T}_t$, this equation can be seen as a generalisation of equation~\eqref{eq:quasi-projection}, where the Gramian matrix $G=I$ is estimated as well.
However, since the resulting operator $P_t^n$ is a projection, its computation no longer requires knowledge of an orthonormal basis $\mathcal{B}_t$.
Moreover, it satisfies the quasi-optimality property
\begin{equation}
\label{eq:Ptn-quasi-opt}
    \norm{g - P_t^ng}^2
    \le \norm{g - P_tg}^2 + \tfrac{1}{1-\delta}\norm{g - P_tg}_{n}^2 ,
\end{equation}
where $\delta := \norm{I - \hat{G}}$ for an estimated Gramian matrix $\hat{G}$ with respect to an arbitrary but fixed orthonormal basis $\mcal{B}_t$.
This comes at the cost that \textbf{$\boldsymbol{P_t^n}$ is not an unbiased estimate of $\boldsymbol{P_t}$}. 

From equation~\eqref{eq:Ptn-quasi-opt}, it is clear that the quantity $\norm{I-\hat{G}}$ plays a crucial role in bounding the error of $P_t^n$, and it is reasonable to draw the sample points not from $\mu_t$ directly but conditioned on the event
\begin{equation}
\label{eq:stability}
    \mcal{S}_\delta := \{\norm{I - \hat{G}} \le \delta \} .
\end{equation}
This makes the least squares projection well-defined and numerically stable and also allows the use of subsampling methods to further reduce the sample size~\cite{haberstich2022boosted,dolbeault2022optimalL2,bartel2023}.
It can be
shown~\cite{cohen2017optimal} that, if the samples are drawn optimally,
$$
    \mbb{P}\pars{\mathcal{S}_\delta}
    \ge 1 - 2d_t \exp\pars{-\tfrac{\delta^2 n}{2 d_t}} .
$$
\revision[0]{The bias and variance constants for this projection are reported in Table~\ref{table:constant_overview}.}

\begin{remark}
    Note that although an orthonormal basis is no longer required to compute the least squares projection, it is still required in all currently known methods to compute an exact sample from $\mu_t$.
    This is a known problem and the focus of current research~\cite{trunschke2024optimalsamplingsquaresapproximation}.
\end{remark}

Despite all the advantages that the least squares projection brings, it also suffers from a severe disadvantage, namely that $c_{\mathrm{bias},2} > 0$, which makes the estimate not well-suited for iterative procedures, as will be discussed at the beginning of section~\ref{sec:Convergence_theorey}.
When $\mcal{H} = L^2(\rho)$, this problem is solved by a weighted least squares projection $P_t^n$ based on \emph{volume-rescaled sampling}~\cite{derezinski2022unbiased}.
The resulting projector is discussed in appendix~\ref{sec:dpp-projection}.
Another option, which works for all $\mathcal{H}$, is to ``debias'' the projection and is discussed in appendix~\ref{app:debiased_projection}.

\section{Convergence theory}
\label{sec:Convergence_theorey}

\begin{theorem}
\label{thm:descent_quasi_projection}
    Assume that the loss function satisfies assumptions~\eqref{eq:L-smooth} and~\eqref{eq:C-retraction} with a $t$-dependent perturbation $\beta_t\ge 0$ and that the projection estimate satisfies assumption~\eqref{eq:bbv}.
    Then at step $t\ge0$ it holds that
    \begin{alignat}{2}
        \mbb{E}\bracs{\mcal{L}\pars{u_{t+1}}\,|\, \mcal{F}_t}
        &\le \mcal{L}\pars{u_t}
        &&- \pars{c_{\mathrm{bias},1} s_t - s_t^2 \tfrac{L+C_{\mathrm{R}}}{2}c_{\mathrm{var},1}} \norm{P_tg_t}^2 \\
        &&&+ s_t c_{\mathrm{bias},2} \norm{P_tg_t}\norm{(I-P_t)g_t} \\
        &&&+ s_t^2 \tfrac{L+C_{\mathrm{R}}}{2} c_{\mathrm{var},2}\norm{(I-P_t)g_t}^2 \\
        &&&+ \beta_t .
        \label{eq:descent_bound}
    \end{alignat}
\end{theorem}
\begin{proof}
    Recall that $u_{t+1} = R_{u_t}(\bar{u}_{t+1})$ and $\bar{u}_{t+1} = u_t - s_tP_{t}^n g_t$ with $g_t = \nabla\mcal{L}(u_t)$.
    By assumption~\eqref{eq:C-retraction}, it holds that
    \begin{align}
        \mbb{E}\bracs*{\mcal{L}(u_{t+1}) \mid \mcal{F}_t}
        &= 
        \mbb{E}\bracs*{\mcal{L}(R_{t}(\bar{u}_{t+1})) \mid \mcal{F}_t}
        \le \mbb{E}\bracs*{\mcal{L}(\bar{u}_{t+1}) \mid \mcal{F}_t} + \revision[0]{\tfrac{C_R}{2}s_t^2\mathbb{E}[\|P_t^n g_t\|^2\mid\mathcal{F}_t]}+ \beta_t .
        \label{eq:retracted_empirical_step}
    \end{align}
    To bound the first term, we conclude from assumption~\eqref{eq:L-smooth} that
    \begin{align}
    \label{eq:Lsmooth_empirical_step}
        \mbb{E}\bracs*{\mcal{L}\pars{\bar{u}_{t+1}}\mid\mcal{F}_t}
        &= \mbb{E}\bracs*{\mcal{L}\pars{u_t - s_t P^n_{t} g_t}\mid\mcal{F}_t} \\
        &\le \mcal{L}\pars{u_t} - s_t \mbb{E}\bracs*{\pars{g_t, P^n_{t} g_t}\mid \mcal{F}_t} + \revision[0]{s_t^2\tfrac{L}{2}\mbb{E}\bracs*{\norm{P^n_{t} g_t}^2\mid\mcal{F}_t} .}
        \label{eq:L-smooth_expectation}
    \end{align}
    Combining equations~\eqref{eq:retracted_empirical_step} and~\eqref{eq:L-smooth_expectation} and employing assumption~\eqref{eq:bbv} yields the desired estimate.
\end{proof}

    Theorem~\ref{thm:descent_quasi_projection} illustrates the advantage of an unbiased projection over biased projections.
    To guarantee descent, the expression
    \begin{align}
        &- \pars{c_{\mathrm{bias},1}s_t - s_t^2 \tfrac{L+C_{\mathrm{R}}}{2}c_{\mathrm{var},1}} \norm{P_tg_t}^2
        + s_t c_{\mathrm{bias},2} \norm{P_tg_t}\norm{(I-P_t)g_t}
        + s_t^2 \tfrac{L+C_{\mathrm{R}}}{2} c_{\mathrm{var},2}\norm{(I-P_t)g_t}^2 + \beta_t \\
        =
        &- s_t \pars*{c_{\mathrm{bias},1}\norm{P_tg_t}^2
        - c_{\mathrm{bias},2} \norm{P_tg_t}\norm{(I-P_t)g_t}}
        + s_t^2 \tfrac{L+C_{\mathrm{R}}}{2} \pars*{c_{\mathrm{var},1} \norm{P_tg_t}^2 + c_{\mathrm{var},2}\norm{(I-P_t)g_t}^2}
        + \beta_t
    \end{align}
    needs to be negative.
    If $c_{\mathrm{bias},2}=0$ and $\beta_t = \beta_t(s_t)\in o(s_t)$, we can always find $s_t > 0$ such that this is the case.
    However, if $c_{\mathrm{bias},2} > 0$, then there may exist a step $t$ such that $\frac{c_{\mathrm{bias},2}}{c_{\mathrm{bias},1}} \ge \frac{\norm{P_tg_t}}{\norm{(I-P_t)g_t}}$.
    In particular, this is the case when $u_t$ is close to a stationary point in $\mcal{M}$ which is not stationary in $\mathcal{H}$.
    In this situation, there exists no step size $s_t > 0$ for which we can guarantee descent.

    A special situation occurs for the least squares risk $\mathcal{L}(v) := \tfrac12\norm{u - v}^2$ and when $u_t\in\mathcal{T}_t$.
    In this case, we can rewrite the fraction
    $$
        \frac{\norm{P_tg_t}^2}{\norm{(I-P_t)g_t}^2}
        = \frac{\norm{u_t - P_tu}^2}{\norm{u - P_t u}^2} .
    $$
    Hence, $\frac{c_{\mathrm{bias},2}}{c_{\mathrm{bias},1}} \ge \frac{\norm{P_tg_t}}{\norm{(I-P_t)g_t}}$ \revision[0]{implies that}
    \begin{equation}
    \label{eq:least_squares_convergence_to_bias}
        \norm{u - u_t}^2
        = \norm{u - P_tu}^2 + \norm{u_t - P_t u}^2
        \le
        (1 + \tfrac{c_{\mathrm{bias},2}}{c_{\mathrm{bias},1}})\norm{u - P_tu}^2 ,
    \end{equation}
    i.e., that the current iterate $u_t$ is already a quasi-best approximation of $u$ in $\mathcal{T}_t$ with constant $(1+\frac{c_{\mathrm{bias},2}}{c_{\mathrm{bias},1}})$.
    If~\eqref{eq:least_squares_convergence_to_bias} is satisfied, we can not find a step size that reduces the error further.
    Moreover, since convergence to a stationary point necessitates $P_t g_t$ vanishing, the condition $\frac{c_{\mathrm{bias},2}}{c_{\mathrm{bias},1}} \ge \frac{\norm{P_tg_t}}{\norm{(I-P_t)g_t}}$ must be satisfied eventually in the non-recovery setting $u\not\in\mcal{M}$.
    Convergence to a stationary point is thus impossible.

    Note, however, that this quasi-optimality is already a sufficient error control in most practical applications, and convergence to a stationary point may not be required. 
    In this spirit, one could consider a relaxed definition of a stationary point, where the norm of the projected gradient is not zero but bounded by the norm of the orthogonal complement.

 In the following, we first provide an analysis of convergence to stationary points in the unbiased case ($c_{\mathrm{bias},2} = 0$), and then consider the biased case, where we prove almost sure convergence to a notion of quasi-critical point.

\subsection{Convergence in expectation in the unbiased case}
\label{sec:Convergence_expectation}
{\color{red} }
We here assume $c_{\mathrm{bias},2} = 0$, 
which is satisfied for the operators from Sections~\ref{sec:non-projection},~\ref{sec:quasi-projection} and~\ref{sec:dpp-projection}.

\begin{theorem}
\label{thm:convergence_in_expectation}
    Assume that the loss function satisfies assumptions~\eqref{eq:L-smooth} and~\eqref{eq:mu-PL_strong} with constants $L,\PL > 0$ and~\eqref{eq:C-retraction} holds with $C_{\mathrm{R}}>0$ and a sequence $\beta\in\ell^1$.
    Moreover, assume that the projection estimate satisfies assumption~\eqref{eq:bbv} with $c_{\mathrm{bias},2}=0$.
    Define the descent factor $\sigma_t$ and contraction factor $a_t$ by
    $$
        \sigma_t := c_{\mathrm{bias},1} s_t - s_t^2 \tfrac{L+C_{\mathrm{R}}}{2} c_{\mathrm{var},1}
        \qquad\text{and}\qquad
        a_t := 1 - 2\PL\sigma_t
    $$
    and assume that $s_t$ is $\mcal{F}_1$-measurable and chosen such that $a_t\in\pars{0,1}$ and
    $$
        s_t^2 \tfrac{L+C_{\mathrm{R}}}{2}c_{\mathrm{var},2}\norm{(I-P_t)g_t}^2 + \beta_t
        \le \xi_t
    $$
    for some $\mcal{F}_1$-measurable sequence $\xi_t$.
    Then, for $T\in\mathbb{N}$ it holds
    $$
        \mbb{E}\bracs*{\mcal{L}\pars{u_{T+1}} - \mcal{L}_{\mathrm{min},\mcal{M}} \mid \mcal{F}_1} \le \pars*{\prod_{t=1}^T a_t} \mbb{E}\bracs*{\mcal{L}\pars{u_{1}} - \mcal{L}_{\mathrm{min},\mcal{M}} \mid \mcal{F}_1} + \bar\xi_T,
    $$
    where $\bar\xi_0 := 0$ and $\bar\xi_{t+1} := \xi_{t+1} + a_{t+1} \bar\xi_t$.
\end{theorem}
\begin{proof}
    Define $X_t := \mcal{L}\pars{u_t} - \mcal{L}_{\mathrm{min},\mcal{M}}$.
    Then, Theorem~\ref{thm:descent_quasi_projection} and~\eqref{eq:mu-PL_strong} provide the relation
    \begin{align*}
        \mbb{E}\bracs*{X_{t+1} \mid \mcal{F}_t}
        &\le a_t X_t + \xi_t .
    \end{align*}
    To solve this recurrence relation, we define $A_{t+1} := \prod\limits_{k=1}^{t}a_k$ with $A_1:=1$.
    Since $s_t$ is $\mathcal{F}_1$ measurable, $A_t^{-1}$ is $\mathcal{F}_t$ measurable for all $t\ge 1$ and it follows that
    $$
        A_{t+1}^{-1} \pars*{\mathbb{E}\bracs*{X_{t+1}\mid\mathcal{F}_t} - a_tX_t}
        = \mathbb{E}\bracs*{A_{t+1}^{-1} X_{t+1} - A_t^{-1}X_t\mid\mathcal{F}_t}
        \le A_{t+1}^{-1}\xi_t .
    $$
    Defining $Y_t := A_t^{-1}X_t$, we can thus write $\mbb{E}\bracs*{Y_{t+1} - Y_t\mid \mcal{F}_t} \le A_{t+1}^{-1} \xi_t$ and summing over all $t = 1, \ldots, T$ yields
    \begin{align}
        \mbb{E}\bracs*{Y_{T+1} - Y_1 \mid \mathcal{F}_1}
        &= \sum_{t=1}^{T} \mbb{E}\bracs*{Y_{t+1} - Y_t \mid \mathcal{F}_1}
        = \sum_{t=1}^{T} \mbb{E}\bracs*{\mbb{E}\bracs*{Y_{t+1} - Y_t \mid \mcal{F}_{t}}\mid \mathcal{F}_1} \\
        &\le \sum_{t=1}^{T} \mbb{E}[A_{t+1}^{-1}\xi_t \mid \mathcal{F}_1]
        = \sum_{t=1}^{T} A_{t+1}^{-1}\xi_t .
    \end{align}
    Resubstituting $Y_t = A_t^{-1}X_t$ and using $A_1 = 1$, we finally obtain the claim
    \begin{align}
        \mbb{E}\bracs*{X_{T+1} \mid \mathcal{F}_1}
        &= A_{T+1}\mbb{E}\bracs*{A_{T+1}^{-1}X_{T+1} \mid \mathcal{F}_1} \\
        &\le A_{T+1} \mbb{E}\bracs{X_1 \mid \mathcal{F}_1} + \sum_{t=1}^{T} A_{T+1} A_{t+1}^{-1} \xi_t \\
        &= A_{T+1} \mbb{E}\bracs{X_1 \mid \mathcal{F}_1}  + \bar\xi_T.
    \end{align}
\end{proof}

\begin{corollary}
    Theorem~\ref{thm:convergence_in_expectation} implies descent with high probability via Markov's inequality.
\end{corollary}

Assuming a non-vanishing orthogonal complement $0 < c \le \norm{(I - P_t)g_t}$, Theorem~\ref{thm:convergence_in_expectation} can only guarantee convergence up to the bias term $\bar\xi_t$.
We investigate the implications of this result in the case of algebraic and exponential decay of $\xi_t$, respectively, in the subsequent corollaries.

\begin{corollary}[Sub-linear convergence in expectation]
\label{corollar:sublinear_inexpectation}
    Assume that the loss function satisfies assumptions~\eqref{eq:L-smooth} and~\eqref{eq:mu-PL_strong} with constants $L,\PL > 0$ and~\eqref{eq:C-retraction} holds with $C_{\mathrm{R}}>0$ and a sequence $\beta_t\in\mcal{O}(t^{-s})$ for some $s>1$.
    Moreover, suppose that the projection estimate satisfies assumption~\eqref{eq:bbv} with $c_{\mathrm{bias},2}=0$.
    Assume that $\norm{(I-P_t)g_t} \le C$ \revision[0]{for some constant $C > 0$} and
    $$
        s_t := c_0 t^{-1}
    $$
    with $p=\min\{2,s\}$ and $c_0 \ge \tfrac{p-1}{2\PL c_{\mathrm{bias},1}}$.
    Then, for any $0<\varepsilon < p$ we obtain a convergence rate of
    $$
        \mbb{E}\bracs*{\mcal{L}(u_{t+1}) - \mcal{L}_{\mathrm{min},\mcal{M}}}
        \in \mathcal{O}\pars{t^{\varepsilon + 1 - p}} .
    $$
\end{corollary}
\begin{proof}
    By the choice of $s_t$, it holds that
    $$
        s_t^2 \tfrac{L+C_{\mathrm{R}}}{2}c_{\mathrm{var},2}\norm{(I-P_t)g_t}^2 + \beta_t
        \le c_0^2 \tfrac{L+C_{\mathrm{R}}}{2} c_{\mathrm{var},2} C^2 t^{-2} + \beta_t =: \xi_t ,
    $$
    with $\xi_t\in\mcal{O}(t^{-p})$ and $p = \min\braces{2, s}$.
    Applying Theorem~\ref{thm:convergence_in_expectation} yields the bound
    $$
        \mbb{E}\bracs*{\mcal{L}\pars{u_{T+1}} - \mcal{L}_{\mathrm{min},\mcal{M}} \mid \mcal{F}_1} \le \pars*{\prod_{t=1}^T a_t} \mbb{E}\bracs*{\mcal{L}\pars{u_{1}} - \mcal{L}_{\mathrm{min},\mcal{M}} \mid \mcal{F}_1} + \bar\xi_T,
    $$
    with $\bar\xi_0 := 0$ and $\bar\xi_{t+1} := \xi_{t+1} + a_{t+1} \bar\xi_t$
    and a contraction factor 
    $$
        a_t
        = 1 - 2\PL c_{\mathrm{bias},1} c_0 t^{-1} + 2\PL \tfrac{L+C_{\mathrm{R}}}{2} c_{\mathrm{var},1} c_0^2 t^{-2}
        =: 1 - c_1t^{-1} + c_2t^{-2}.
    $$
    Let $\varepsilon$ be fixed and observe that the preceding equation implies that there exists $t_0 > 0$ such that for all $t > t_0$
    \begin{equation}
        \label{eq:a_t_bounds}
        1 - c_1 t^{-1}
        \le a_t
        \le 1 - (c_1-\varepsilon) t^{-1} .
    \end{equation}
    Since for any $c>0$, redefining $t_0 := \max\braces{t_0, \floor{c}+1}$ it holds that
    $$
        \lim_{T\to\infty} \left(\prod\limits_{t=t_0}^T 1 - c t^{-1}\right) T^{c}
        = \frac{\Gamma(t_0)}{\Gamma(t_0 - c)} \in (0, \infty) ,
    $$
    we conclude from \eqref{eq:a_t_bounds} that $T^{-c_1} \lesssim \prod_{t=1}^T a_t \lesssim T^{\varepsilon-c_1}$.
    This implies
    $$
        \bar\xi_T
        = \sum_{t=1}^T \pars*{\prod_{k=t}^{T-1} a_k} \xi_t
        = \sum_{t=1}^T \frac{\prod_{k=1}^{T-1} a_k}{\prod_{k=1}^{t-1} a_k} \xi_t
        \lesssim (T-1)^{\varepsilon-c_1} \sum_{t=1}^T \pars*{t-1}^{c_1} \xi_t .
    $$
    Since $\xi_t\in\mcal{O}(t^{-p})$, we conclude 
    \begin{align}
        \bar\xi_T
        \lesssim (T-1)^{\varepsilon-c_1} \sum_{t=1}^T \pars*{t-1}^{c_1} t^{-p}
        \le (T-1)^{\varepsilon-c_1} \int_{1}^T t^{c_1-p} \dx[t]
        = \frac{T^{\varepsilon - p + 1} - (T-1)^{\varepsilon - c_1}}{c_1 - p + 1} .
    \end{align}
    The choice $c_0 \ge \tfrac{p-1}{2\PL c_{\mathrm{bias},1}}$ implies $c_1 = 2\PL c_{\mathrm{bias},1} c_0 \ge p-1$ and $c_1 - p + 1 \ge 0$.
    This yields $\prod_{t=1}^T a_t \in\mcal{O}(T^{\varepsilon - p + 1})$ and $\bar\xi_T \in\mcal{O}(T^{\varepsilon - p + 1})$.
\end{proof}

\begin{corollary}[Exponential convergence in expectation]
\label{corollary:exponential_conv_expectation}
    Assume that the loss function satisfies assumptions~\eqref{eq:L-smooth} and~\eqref{eq:mu-PL_strong} with constants $L,\PL > 0$ and~\eqref{eq:C-retraction} holds with $C_{\mathrm{R}}>0$ and a sequence $\beta_t\in\mcal{O}(b^{t})$ for some $b\in\pars{0,1}$.
    Moreover, suppose that the projection estimate satisfies assumption~\eqref{eq:bbv} with $c_{\mathrm{bias},2}=0$ and that $\norm{(I-P_t)g_t} = 0$.
    If the step size is chosen optimally as $s_t :\equiv \tfrac{c_{\mathrm{bias},1}}{(L+C_{\mathrm{R}}) c_{\mathrm{var},1}}$ and if $a := \pars{1 - 2 \tfrac{\lambda c_{\mathrm{bias},1}}{(L+C_{\mathrm{R}}) c_{\mathrm{var},1}}} \in\pars{0,1}$ we obtain the exponential convergence
    $$
        \mbb{E}\bracs*{\mcal{L}\pars{u_{T+1}} - \mcal{L}_{\mathrm{min},\mcal{M}} \mid \mcal{F}_1}
        \in 
        \begin{cases}
            \mathcal{O}\pars{a^T + b^T}, &a\neq b, \\
            \mathcal{O}\pars{Ta^T}, & a= b.
        \end{cases}
    $$
\end{corollary}
\begin{proof}
    By assumption, we have
    $$
        s_t^2 \tfrac{L+C_{\mathrm{R}}}{2}c_{\mathrm{var},2}\norm{(I-P_t)g_t}^2 + \beta_t
        = \beta_t
        \lesssim b^{t}.
    $$
    Taking $\xi_t=\beta_t$ and applying Theorem~\ref{thm:convergence_in_expectation} yields the bound
    \begin{equation}
    \label{eq:proof_bound_01}
        \mbb{E}\bracs*{\mcal{L}\pars{u_{T+1}} - \mcal{L}_{\mathrm{min},\mcal{M}} \mid \mcal{F}_1} \le \pars*{\prod_{t=1}^T a_t} \mbb{E}\bracs*{\mcal{L}\pars{u_{1}} - \mcal{L}_{\mathrm{min},\mcal{M}} \mid \mcal{F}_1} + \bar\xi_T
    \end{equation}
    with $\bar\xi_0 := 0$ and $\bar\xi_{t+1} := \xi_{t+1} + a_{t+1} \bar\xi_t$.
    The factor $a_t = 1 - 2\PL\sigma_t$ is minimal when $\sigma_t>0$ is maximal and $\sigma_t = c_{\mathrm{bias},1}s_t - s_t^2 \tfrac{L+C_{\mathrm{R}}}{2} c_{\mathrm{var},1}$ is maximised for $s_t = \tfrac{c_{\mathrm{bias},1}}{(L+C_{\mathrm{R}}) c_{\mathrm{var},1}}$.
    Since $a \equiv a_t$ is constant, $\prod_{t=1}^T a_t = a^T$. 
    If $a\neq b$, then 
    $$
        \bar\xi_T
        = \sum_{t=1}^T a^{T - t} \xi_t
        \lesssim \sum_{t=1}^T a^{T - t} b^{t}
        = a^T \sum_{t=1}^T (b/a)^{t}
        = \frac{1}{a/b-1} (a^T - b^T)
        $$
    else 
    $$
     \bar\xi_T
        = \sum_{t=1}^T a^{T - t} \xi_t
        \lesssim \sum_{t=1}^T a^{T - t} b^{t}
        = a^T \sum_{t=1}^T (b/a)^{t}
        \leq T a^T.
    $$
    Using these upper bounds to estimate \eqref{eq:proof_bound_01} yields the result.
\end{proof}

Before concluding this section, we present an interesting corollary of Theorem~\ref{thm:convergence_in_expectation} in the situation of a linear model class $\mcal{M}$ and an approximative projection given by the quasi-projection from section~\ref{sec:quasi-projection}.

\begin{corollary}
\label{cor:convergence_up_to_precision}
    Consider the least squares loss $\mathcal{L}(v) := \tfrac{1}{2} \norm{u - v}^2$ for $u\in\mathcal{H}$ on a $d$-dimensional linear model class $\mcal{M}$ and $\mcal{T}_t = \mcal{M}$ with $R_t \equiv I$, the identity operator.
    Moreover, suppose that $P_t$ satisfies the bias and variance bounds of the quasi-projection operator from Lemma~\ref{lem:sgd_bias_variance}.
    Suppose that $u_1 = 0$, $n = c(d-1)\in\mathbb{N}$ for some constant $c \ge \tfrac1{d-1}$ and that the step size is chosen optimally.
    Then, for every $\varepsilon>0$, the perturbed relative error bound
    $$
        \mbb{E}\bracs*{\frac{\norm{u-u_t}^2}{\norm{u}^2}}
        \le \pars{1 + \tfrac{1}{1+c}\tfrac{d}{d-1}}\frac{\norm{u - P_tu}^2}{\norm{u}^2} + \varepsilon
    $$
    is satisfied after $t = \tfrac{\ln(2\varepsilon^{-1})}{\ln(1+c)}$ iterations.
\end{corollary}
\begin{proof}
    Since all assumptions of Theorem~\ref{thm:convergence_in_expectation} are fulfilled, the loss iterates satisfy
    \begin{equation}
    \label{eq:exp_rate_linear_discussion}
        \mbb{E}\bracs*{\mcal{L}\pars{u_{\tau+1}} - \mcal{L}_{\mathrm{min},\mcal{M}} \mid \mcal{F}_1} \le \pars*{\prod_{t=1}^\tau a_t} \mbb{E}\bracs*{\mcal{L}\pars{u_{1}} - \mcal{L}_{\mathrm{min},\mcal{M}} \mid \mcal{F}_1} + \bar\xi_\tau,
    \end{equation}
    where $\bar\xi_0 := 0$, $\bar\xi_{t+1} := \xi_{t+1} + a_{t+1} \bar\xi_t$ and $\xi_t := s_t^2 \tfrac{L+C_{\mathrm{R}}}{2}c_{\mathrm{var},2}\norm{(I-P_t)g_t}^2 + \beta_t$.
    Note that $\beta_t\equiv C_{\mathrm{R}} = 0$ since $R_t$ is the identity. 
    The factor $a_t = 1 - 2\PL\sigma_t$ is minimal when $\sigma_t>0$ is maximal and $\sigma_t = c_{\mathrm{bias},1} s_t - s_t^2 \tfrac{L}{2} \tfrac{n+d-1}n$ is maximised for $s_t = \tfrac1L \tfrac{n}{n+d-1}$, where we have already used the bounds for $c_{\mathrm{bias},1}$ and $c_{\mathrm{var},1}$ from Lemma~\ref{lem:sgd_bias_variance}.
    Since $L = \PL = 1$, choosing the optimal step size and $n=c(d-1)$ yields
    $$
        a_t
        = 1 - \tfrac{n}{n+d-1}
        = \tfrac{d-1}{n+d-1}
        = \tfrac{1}{1+c}
    \qquad\text{and}\qquad
        \xi_t
        = \tfrac{nd}{2(n+d-1)^2} \norm{(I-P) g_t}^2
        = \tfrac{a_t^2}{2}\tfrac{cd}{d-1} \norm{(I-P) g_t}^2 ,
    $$
    again, using the bounds for $c_{\mathrm{bias},2}$ and $c_{\mathrm{var},2}$ from Lemma~\ref{lem:sgd_bias_variance}.
    Before substituting these values into equation~\eqref{eq:exp_rate_linear_discussion} it is convenient to simplify the expressions further.
    For this, note that $\norm{(I-P_t)g_t} = \norm{u - P_tu}$ and $\mathcal{L}\pars{u_t} - \mathcal{L}_{\mathrm{min},\mathcal{M}} = \tfrac12\norm{P_tu-u_t}^2$.
    With the choice $u_1 = 0$, equation~\eqref{eq:exp_rate_linear_discussion} thus implies
    \begin{align}
        \mbb{E}\bracs*{\norm{P_tu-u_t}^2}
        &\le 2 \pars*{a_t^{t}\norm{P_tu}^2 + \tfrac{1 - a_t^t}{1 - a_t}\tfrac{a_t^2}{2}\tfrac{cd}{d-1}\norm{u - P_tu}^2} \\
        &\le 2 (1+c)^{-t}\norm{P_tu}^2 + \tfrac{1}{c(1+c)}\tfrac{cd}{d-1}\norm{u - P_tu}^2 .
    \end{align}
    Finally, recall that $\norm{u - u_t}^2 = \norm{u - P_tu}^2 + \norm{P_tu - u_t}^2$ and thus
    $$
        \mbb{E}\bracs*{\frac{\norm{u-u_t}^2}{\norm{u}^2}}
        \le \pars{1 + \tfrac{1}{1+c}\tfrac{d}{d-1}}\frac{\norm{u - P_tu}^2}{\norm{u}^2} + 2 (1+c)^{-t} .
    $$
    The perturbation $2(1+c)^{-t}$ \revision[0]{is below} the bound $\varepsilon$ after $t = \tfrac{\ln(2\varepsilon^{-1})}{\ln(1+c)}$ steps.
\end{proof}

The preceding corollary provides a perturbed relative error bound.
For $c=1$ we reach machine precision $\varepsilon = 2^{-53}$ (for the IEEE $64$ bit floating point standard) after $t_{\mathrm{max}} = 54$ iterations.
This means that at most $N\le 54(d-1)$ sample points are required to achieve the same relative quasi-best approximation error as the boosted least squares approximation~\cite{haberstich2022boosted}
$$
    \mbb{E}\bracs*{\frac{\norm{u - \hat{P}^n u}^2}{\norm{u}^2}}
    \lesssim \frac{\norm{u - Pu}^2}{\norm{u}^2} .
$$
Moreover, if we require the perturbation to be of the same order as the error, i.e.\ $\varepsilon = \frac{\norm{u-P_tu}^2}{\norm{u}^2}$, we obtain the well-known sample size bound $N \gtrsim d_t\ln(d_t)$.
\footnote{
We require $t = \frac{\ln(2\varepsilon^{-1})}{\ln(1+c)}$ many steps with $n=c(d-1)$ sample points each.
This means that the total number of samples is bounded by $N \lesssim c(d-1) \frac{\ln(d)}{\ln(1+c)} \in \Theta(d\ln(d))$.}
In particular, the presented algorithm remains optimal in expectation only up to a logarithmic factor.
Optimal bounds are provided in~\cite{Dolbeault_2023}, and bounds that can be achieved without optimal sampling can be found in~\cite{Krieg2023}.
Also note that another method which kills the logarithmic factor by multi-level optimal sampling in certain situations has been introduced in~\cite{Krieg2018}.

Another interesting point is that, in contrast to projections, quasi-projections do not require $c>1$.
Indeed, the results are also valid for the minimal choice $c=\tfrac1{d-1}$ (i.e.\ $n=1$).
Intuitively, one would expect this to increase convergence speed since every new sample point can benefit from the most recent gradient information.
And indeed, it can be shown that the convergence rate in Corollary~\ref{cor:convergence_up_to_precision} is maximised for $n=1$.

\subsection{Almost sure convergence in the unbiased case}
\label{sec:Convergence_as}

In practice, \revision[0]{convergence} with a probability smaller than $1$ might be unsatisfactory.
We, therefore, prove almost sure convergence of the descent method in the subsequent section.
The results of this section rely on the following fundamental proposition.

\begin{proposition}(Robbins~and~Siegmund,~1971,~\cite{Robbins1971}) 
\label{prop:RobbinsSiegmund}
Let $X_t, Y_t$ and $Z_t$ be three sequences of random variables that are adapted to a filtration $(\mathcal{F}_t)$. Let $(\gamma_t)$ be a sequence of nonnegative real numbers such that $\prod\limits_{t=1}^\infty (1+\gamma_t)<\infty$. Suppose the following conditions hold:
\begin{enumerate}
   \item $X_t,Y_t$ and $Z_t$ are non-negative for all $t\geq 1.$
   \item $\mathbb{E}[Y_{t+1}\mid \mathcal{F}_t]\leq (1+\gamma_t)Y_t - X_t + Z_t$ for all $t\geq 1$. 
   \item $ (Z_t) \in \ell^1$ almost surely.
\end{enumerate}
Then, $X_t\in\ell^1$ almost surely and $Y_t$ converges almost surely.
\end{proposition}

This requires us to assume $\beta\in\ell^1$ almost surely and $c_{\mathrm{bias},2} = 0$, because both terms enter into the $Z_t$ term in Proposition~\ref{prop:RobbinsSiegmund} during the proofs.
\revision[0]{In the following sections~\ref{section:almost_sure_LS},~
\ref{section:almost_sure_SPL} and~\ref{section:almost_sure_PL} we assume $c_{\mathrm{bias},2} = 0$ and discuss the case $c_{\mathrm{bias},2} \neq 0$ in section~\ref{section:almost_biased}.}

\subsubsection{Almost sure convergence under assumption~\eqref{eq:L-smooth}}
\label{section:almost_sure_LS}

Theorem~\ref{thm:descent_quasi_projection} provides conditions on the step size and sample size under which a single step of the descent scheme~\eqref{eq:descent_scheme} reduces the loss.
Notably, the descent depends on the part of the gradient that is not captured by the local linearisation, i.e.\ $\norm{(I-P_t)g_t}$.
This makes intuitive sense since an empirical estimate must depend on the function $g_t$ and can not just depend on the unknown projection $P_t g_t$.
This has profound implications.
If the sequence of iterates~\eqref{eq:descent_scheme} converges to a point $u_\infty$ with $P_{u_\infty}\nabla\mcal{L}\pars{u_\infty} = 0$ but $\nabla\mcal{L}\pars{u_\infty} \neq 0$, then there exists a constant such that $\norm{(I-P_t)g_t} \ge c > 0$.
Inserting this into the equation~\eqref{eq:descent_bound} and taking the limit $\norm{P_t g_t}\to 0$ yields the asymptotic descent bound
$$
    \mbb{E}\bracs{\mcal{L}\pars{u_{t+1}}\,|\, \mcal{F}_t}
    \le \mcal{L}\pars{u_t} + s_t^2 \tfrac{L+C_{\mathrm{R}}}{2} c_{\mathrm{var},2} c^2 +\beta_t .
$$
This means that the step size sequence $s_t$ must converge to zero to ensure convergence to a stationary point.
The subsequent theorem uses this intuition to prove that a sequence of iterates almost surely converges to a stationary point when the step size sequence satisfies certain summability conditions.

\begin{theorem}
\label{thm:stationary_point}
    Assume that the loss function satisfies assumptions~\eqref{eq:lower_bound},~\eqref{eq:L-smooth} and~\eqref{eq:C-retraction} and that for every current step $t\ge0$ the  $x_1, \ldots, x_n$ are i.i.d.\ samples drawn from $\mu_t$ given $\mcal{F}_t$.
    Moreover, let $\xi_t\in\ell^1$ and assume that the steps size sequence satisfies $$
        s_t^2 \tfrac{L+C_{\mathrm{R}}}{2}c_{\mathrm{var},2}\norm{\pars{I-P_t}g_t}^2 + \beta_t \le \xi_t
        \qquad\text{and}\qquad
        s_t \le \bar{s} := \tfrac{1}{L+C_{\mathrm{R}}}\tfrac{c_{\mathrm{bias},1}}{c_{\mathrm{var},1}} .
    $$
    Then it holds almost surely that
    $$
        \min_{t=1,\ldots,\tau} \norm{P_tg_t} \lesssim \pars*{\sum_{t=1}^\tau \sigma_t}^{-1/2}
        \qquad\text{with}\qquad
        \sigma_t := c_{\mathrm{bias},1} s_t - s_t^2 \tfrac{L+C_{\mathrm{R}}}{2} c_{\mathrm{var},1} .
    $$
    Moreover, if $(s_t)\not\in \ell^1$, then zero is an accumulation point of the sequence $(\norm{P_tg_t})$ and if $s_t \ge s^* >0$, the sequence $(\norm{P_tg_t})$ converges to zero.
\end{theorem}
\begin{proof}
    The proof of this theorem is inspired by the proof of Theorem~1 in~\cite{liu2022sure}.
    Define $\eta_t := \norm{P_tg_t}^2$ and recall that Theorem~\ref{thm:descent_quasi_projection} implies
    \begin{equation}
        \mbb{E}\bracs{\mcal{L}\pars{u_{t+1}} - \mcal{L}_{\mathrm{min}}\,|\, \mcal{F}_t}
        \le \pars{\mcal{L}\pars{u_t} - \mcal{L}_{\mathrm{min}}}
        - \sigma_t\eta_t + \xi_t .
    \end{equation}
    Since $Y_t := \mcal{L}\pars{u_t} - \mcal{L}_{\mathrm{min}} \ge 0$ by assumption~\eqref{eq:lower_bound}, $X_t := \sigma_t\eta_t \ge 0$ by the choice of $s_t$ and $0\le \xi\in\ell^1$ by definition, Proposition~\ref{prop:RobbinsSiegmund} implies that $\eta\in\ell^1_\sigma$ almost surely.
    The weighted version of Stechkin's lemma~\cite{trunschke2023weighted} now implies
    $$
        \min_{t=1,\ldots,\tau} \norm{P_tg_t}
        = \pars*{\min_{t=1,\ldots,\tau} \eta_t}^{1/2}
        \le \pars*{\sum_{t=1}^{\tau} \sigma_t}^{-1/2} \norm{\eta}_{\ell^1_\sigma}^{1/2} ,
    $$
    \revision[0]{where $\|\eta\|_{\ell^1_\sigma} := \sum_{t\in\mathbb{N}} |\sigma_t\eta_t|$.}
    To show the last two assertions, recall that $\sigma_t = s_t - \tfrac{s_t^2}{2\bar{s}}$.
    Since the step sizes have to satisfy $s_t \le \bar{s}$, it holds that $\sigma_t \ge \tfrac{s_t}{2}$.
    So if $s\not\in\ell^1$, then also $\sigma\not\in\ell^1$ and we can conclude that
    $$
        \lim_{\tau\to\infty} \min_{t=1,\ldots,\tau} \norm{P_tg_t} = 0 .
    $$
    This means, that zero is an accumulation point of the sequence $\norm{P_tg_t}$.
    If moreover $s_t\ge s^* >0$, then $\sigma_t \ge \tfrac{s_t}2 \ge \tfrac{s^*}{2} =: \sigma^*$.
    In this case $\eta\in\ell^1_\sigma \subseteq \ell^1_{\sigma^*} = \ell^1$ and $\eta$ must converge to zero.
\end{proof}

In the preceding theorem, we have seen convergence if the step size sequence satisfies certain bounds that depend on the value of $\norm{(I-P_t)g_t}$ and $\beta_t$.
In the best case of $\norm{(I-P_t)g_t} = \beta_t = 0$, the step size can be chosen constant, and the rate of convergence is of order
\begin{equation}
\label{eq:rate_recovery_BLSCR}
    \min_{t=1,\ldots,\tau} \norm{P_tg_t}^2 \lesssim \frac1{\tau^{1 - \varepsilon}},
\end{equation}
which is, up to $\varepsilon$, the rate of convergence of the deterministic gradient descent algorithm.
In the worst-case setting $\norm{(I-P_t)g_t} + \beta_t\ge c > 0$, Theorem~\ref{thm:stationary_point} requires the classical Robbins--Monro step size condition $s\in\ell^2$ and $s\not\in\ell^1$.
If we choose $s_t := \tfrac{1}{t^{1/2+\varepsilon}}$ for some $\varepsilon\in\pars{0,\tfrac12}$, we obtain a rate of convergence of
\begin{equation}
\label{eq:rate_nonrecovery_BLSCR}
        \min_{t=1,\ldots,\tau} \norm{P_tg_t}^2 \lesssim \frac1{\tau^{1/2 - \varepsilon}}.
\end{equation}

This is the same rate of convergence as is known for classical stochastic gradient descent~\cite{liu2022sure}.
Besides these two extreme cases, Theorem~\ref{thm:stationary_point} also provides conditions for the intermediate case $\norm{(I-P_t)g_t} \to 0$.

\subsubsection{Almost sure convergence under assumption~\eqref{eq:mu-PL_strong}}
\label{section:almost_sure_SPL}
We have seen in equations~\eqref{eq:rate_recovery_BLSCR} and~\eqref{eq:rate_nonrecovery_BLSCR} that the proposed algorithm can already yield a faster rate of convergence than SGD in the setting of $L$-smooth loss functions.
This motivates our investigation of the rate of convergence under the additional assumption~\eqref{eq:mu-PL_strong}.
The remainder of this section is devoted to investigating the algorithm in this setting.

\begin{theorem}
    \label{thm:rates_abstrac_SPL}
    Assume that the loss function satisfies assumptions ~\eqref{eq:lower_bound},~\eqref{eq:L-smooth},~\eqref{eq:mu-PL_strong} and~\eqref{eq:C-retraction} with constants $L,\PL$ and $C_{\mathrm{R}}$, respectively.
    Assume further that for the current step $t\ge0$, the points $x_1, \ldots, x_n$ are i.i.d.\ samples from  $ \mu_t$ given $\mcal{F}_t$.
    Define the constants
    $$
        \sigma_t :=  c_{\mathrm{bias},1} s_t - s_t^2 \tfrac{L+C_{\mathrm{R}}}{2} c_{\mathrm{var},1}
        \qquad\text{and}\qquad
        a_t := \pars*{1 - 2\PL \sigma_t} .
    $$
    If $\sigma_t > 0$, then
    $$
        \mbb{E}\bracs*{\mcal{L}(u_{t+1}) - \mcal{L}_{\mathrm{min},\mathcal{M}}\,|\,\mcal{F}_t} \le a_t \pars*{\mcal{L}(u_t) - \mcal{L}_{\mathrm{min},\mathcal{M}}}
        + s_t^2\tfrac{L+C_{\mathrm{R}}}{2}c_{\mathrm{var},2}\norm{(I-P_t)g_t}^2 +\beta_t .
    $$
    Moreover, let $\xi_t\in\ell^1$ and assume that the step size sequence satisfies
    $$
        s_t^2 \tfrac{L+C_{\mathrm{R}}}{2}c_{\mathrm{var},2}\norm{\pars{I-P_t}g_t}^2 + \beta_t \le \xi_t .
    $$
    Then for any sequence $\chi_t\geq 0$ that satisfies $a_t\chi_{t+1} \leq \chi_t$ for all $t\geq T$, $T\in\mathbb{N}$ and $\chi_{t+1}\xi_t\in\ell^1$, it holds almost surely that
    $$
        \mcal{L}\pars{u_t} - \mcal{L}_{\mathrm{min},\mathcal{M}}
        \lesssim \chi_t^{-1} .
    $$
\end{theorem}
\begin{proof}
      Using Theorem~\ref{thm:descent_quasi_projection} and subtracting $\mathcal{L}_{\mathrm{min},\mathcal{M}}$ we obtain
    \begin{align}
        \mbb{E}\bracs{\mcal{L}\pars{u_{t+1}} - \mcal{L}_{\mathrm{min},\mathcal{M}}\,|\, \mcal{F}_t} 
        &\le
        \begin{multlined}[t]
            \pars{\mcal{L}\pars{u_t} - \mcal{L}_{\mathrm{min},\mathcal{M}}}
            - \pars*{c_{\mathrm{bias},1}s_t - s_t^2 \tfrac{L+C_{\mathrm{R}}}{2} c_{\mathrm{var},1}}\norm{P_{t} g_t}^2 \\
            + s_t^2 \tfrac{L+C_{\mathrm{R}}}{2} c_{\mathrm{var},2} \norm{(I-P_t)g_t}^2
        \end{multlined} \\
        &= \pars{\mcal{L}\pars{u_t} - \mcal{L}_{\mathrm{min},\mathcal{M}}}
        - \sigma_t\norm{P_{t} g_t}^2 + s_t^2 \tfrac{L+C_{\mathrm{R}}}{2} c_{\mathrm{var},2} \norm{(I-P_t)g_t}^2
    \end{align}
    and employing assumption~\eqref{eq:mu-PL_strong} yields the first claim
    \begin{align*}
        \mbb{E}\bracs{\mcal{L}\pars{u_{t+1}} - \mcal{L}_{\mathrm{min},\mathcal{M}}\,|\, \mcal{F}_t} 
        &\le \begin{multlined}[t]
            \pars{\mcal{L}\pars{u_t} - \mcal{L}_{\mathrm{min},\mathcal{M}}}
            - 2\PL\sigma_t\pars{\mcal{L}\pars{u_t} - \mcal{L}_{\mathrm{min},\mathcal{M}}} \\
            + s_t^2 \tfrac{L+C_{\mathrm{R}}}{2} c_{\mathrm{var},2} \norm{(I-P_t)g_t}^2
        \end{multlined} \\
        &\le a_t \pars{\mcal{L}\pars{u_t} - \mcal{L}_{\mathrm{min},\mathcal{M}}}
        + \xi_t .
    \end{align*}
    Now define $\tilde{Y}_t := \mcal{L}\pars{u_{t+1}} - \mcal{L}_{\mathrm{min},\mathcal{M}}$.
    Since $\chi_t$ is non-negative and deterministic, it holds
    \begin{align*}
        \mbb{E}\bracs*{\tilde Y_{t+1} \mid \mcal{F}_t}
        \le a_t \tilde Y_t + \xi_t\quad&\Rightarrow\quad
         \mbb{E}\bracs*{\chi_{t+1}\tilde Y_{t+1} \mid \mcal{F}_t}
        \le \chi_{t+1}a_t \tilde Y_t + \chi_{t+1}\xi_t \le \chi_{t} \tilde Y_t + \chi_{t+1}\xi_t 
    \end{align*}
    We now define $Y_{t-T} := \chi_t\tilde{Y}_t$ and $Z_{t-T} := \chi_{t+1}\xi_t$. Since $(Z_t)_{t=0}^\infty\in\ell^1$ by assumption and $(Y_t)_{t=0}^\infty$ and $(Z_t)_{t=0}^\infty$ are non-negative by definition, application of Proposition \ref{prop:RobbinsSiegmund} with the choice $X_t\equiv 0$ and $\gamma_t\equiv 0$ yields convergence of $Y_{t-T} = \chi_t \tilde{Y}_t$  almost surely. Consequently, $\tilde{Y}_t\in \mathcal{O}(\chi_t^{-1})$ almost surely.
\end{proof}

As in Theorem~\ref{thm:stationary_point}, the rate of convergence in Theorem~\ref{thm:rates_abstrac_SPL} depends on the value of $\norm{(I-P_t)g_t}$.
It thus makes sense to consider the best case scenario $\lim_{t\to\infty} \norm{(I-P_t)g_t} = 0$ as well as the worst case scenario $\norm{(I-P_t)g_t} \ge c > 0$.
For the best case scenario we can show the following result.
\begin{corollary}[Exponential convergence]
\label{cor:exponential-convergence}
    Assume that the loss function satisfies assumptions~\eqref{eq:L-smooth},~\eqref{eq:mu-PL_strong} and~\eqref{eq:C-retraction} with constants $L,\PL$ and $C_{\mathrm{R}}$ and $\beta_t\equiv 0$, respectively.
    Assume further, that for the current step $t\ge0$ the $x_1, \ldots, x_n$ are i.i.d.\ samples from  $\mu_t$ given $\mcal{F}_t$. 
    Assume that $\|(I-P_t)g_t\|\le c\norm{P_tg_t}$ for some $c\geq 0$ and let $s_t\equiv \bar s$ such that 
    $ \sigma = c_{\mathrm{bias},1}\bar s - \bar s^2 \tfrac{L+C_{\mathrm{R}}}{2} (c_{\mathrm{var},1} + cc_{\mathrm{var},2})>0$ and $a = 1-2\PL \sigma < 1$.
    Then almost surely
    \begin{equation}
    \label{eq:rate_recovery_LSSPLCR}
        \mathcal{L}(u_t) - \mathcal{L}_{\mathrm{min}, \mathcal{M}}
        \le a^t .
    \end{equation}
\end{corollary}
\begin{proof}
    Recalling the proof of Theorem \ref{thm:rates_abstrac_SPL} and the definition of $\sigma_t$ and employing assumption~\eqref{eq:mu-PL_strong} we analogously obtain
    \begin{align}
        \mbb{E}\bracs{\mcal{L}\pars{u_{t+1}} - \mcal{L}_{\mathrm{min},\mathcal{M}}\,|\, \mcal{F}_t} 
        &\leq \pars{\mcal{L}\pars{u_t} - \mcal{L}_{\mathrm{min},\mathcal{M}}}
        - \sigma_t\norm{P_{t} g_t}^2 + s_t^2 \tfrac{L+C_{\mathrm{R}}}{2} c_{\mathrm{var},2} \norm{(I-P_t)g_t}^2 \\
        &\leq \pars{\mcal{L}\pars{u_t} - \mcal{L}_{\mathrm{min},\mathcal{M}}}
        - \sigma_t\norm{P_{t} g_t}^2 + s_t^2 \tfrac{L+C_{\mathrm{R}}}{2} c c_{\mathrm{var},2} \norm{P_tg_t}^2 \\
        &= 
        \pars{\mcal{L}\pars{u_t} - \mcal{L}_{\mathrm{min},\mathcal{M}}}
        -  \sigma \norm{P_{t} g_t}^2 . \\
        &\leq 
        \pars{\mcal{L}\pars{u_t} - \mcal{L}_{\mathrm{min},\mathcal{M}}}
        - 2\PL\sigma\pars{\mcal{L}\pars{u_t} - \mcal{L}_{\mathrm{min},\mathcal{M}}} \\
        &= a \pars{\mcal{L}\pars{u_t} - \mcal{L}_{\mathrm{min},\mathcal{M}}} .
    \end{align}
    Now for $\tilde{Y}_t := \mcal{L}\pars{u_{t+1}} - \mcal{L}_{\mathrm{min},\mathcal{M}}$ and  $\chi_t = a^{-t}$ it follows
    \begin{align*}
        \mbb{E}\bracs*{\tilde Y_{t+1} \mid \mcal{F}_t}
        \le a_t \tilde Y_t  \quad&\Rightarrow\quad
          \mbb{E}\bracs*{\chi_{t+1}\tilde Y_{t+1} \mid \mcal{F}_t}
         \le \chi_{t} \tilde Y_t.
    \end{align*}
    We now define $Y_{t-T} := \chi_t\tilde{Y}_t$.
    Application of Proposition \ref{prop:RobbinsSiegmund} with the choice $X_t\equiv 0$, $\gamma_t\equiv 0$ and $Z_t\equiv 0$ yields convergence of $\tilde{Y}_t\in \mathcal{O}(\chi_t^{-1}) = \mathcal{O}(a^t)$ almost surely.
\end{proof}

\begin{remark}
    The statement of Corollary~\ref{cor:exponential-convergence} remains true, although in modified form, under the weaker assumption that $\beta_t$ decays exponentially.
    We conjecture that in general
    $$
        \mcal{L}(u_t) - \mcal{L}_{\mathrm{min},\mcal{M}} \lesssim a^t + \beta_t .
    $$
\end{remark}

The situation of Corollary~\ref{cor:exponential-convergence} applies in two particular cases where $\lim_{t\to\infty} \norm{(I-P_t)g_t} = 0$.
In the first case, the best-case scenario $\norm{(I-P_t)g_t} = 0$, we can choose $c=0$ in Corollary~\ref{cor:exponential-convergence}.
In the second case, we consider a model class $\mcal{M}$ with positive reach \revision[0]{$R_{\mathcal{M}} := \operatorname{rch}(\mathcal{M}) > 0$} \revision[0]{(see~\cite{Federer1959,trunschke2023convergence} for a definition)}.
In this setting, we can prove the subsequent Lemma~\ref{lemma:kappa_bounded_reach} in Appendix~\ref{app:kappa_bound_reach}.
\begin{lemma}
\label{lemma:kappa_bounded_reach}
    Suppose that $\mcal{M}$ is a differentiable manifold and that $\mcal{T}_v$ is the tangent space of $\mcal{M}$ at $v\in\mcal{M}$.
    Assume moreover, that its \revision[0]{local} reach $R_v := \operatorname{rch}\pars{\mcal{M}, v} > 0$ is positive.
    Let $u\in\mcal{M}$ and consider the cost functional $\mcal{L}(v) := \tfrac12\norm{u - v}^2$.
    Then it holds for all $v\in\mcal{M}$ with $\norm{u - v} \le r \le R_v$ that
    $$
        \frac{\norm{(I-P_v)g}}{\norm{P_v g}}
        \le \frac{r}{R_v} ,
    $$
    where $g := \nabla \mcal{L}\pars{v} = v - u$.
\end{lemma}
Under the preconditions of the preceding lemma, \revision[0]{and recalling that $\operatorname{rch}(\mathcal{M}) := \inf_{v\in\mathcal{M}} \operatorname{rch}(\mathcal{M}, v) > 0$}, we can choose $c = \frac{r}{R_{\mathcal{M}}}$ in Corollary~\ref{cor:exponential-convergence}.

We now investigate the worst case scenario $\norm{(I-P_t)g_t} \ge c > 0$ and obtain the following result.

\begin{corollary}[Sublinear convergence]
\label{cor:quasi-linear-convergence}
    Assume that the loss function satisfies assumptions~\eqref{eq:L-smooth},~\eqref{eq:mu-PL_strong} and~\eqref{eq:C-retraction} with constants $L,\PL$ and $C_{\mathrm{R}}$, respectively.
    Assume further that for the current step $t\ge0$,  the   $x_1, \ldots, x_n$ are i.i.d.\ samples drawn from $\mu_t$ given   $\mcal{F}_t$.
    Choose $\delta\in(0,\tfrac12)$ and $s_t \in \mathcal{O}(t^{\delta-1})$ as well as $\beta_t\in\mcal{O}(t^{2\delta-2})$.
    Assume $\|(I-P_t)g_t\|\in\ell^\infty$ almost surely.
    For any $\epsilon\in (2\delta,1)$ it then holds almost surely that
    \begin{equation}
    \label{eq:rate_nonrecovery_LSSPLCR}
        \mathcal{L}(u_t) - \mathcal{L}_{\mathrm{min}, \mathcal{M}} \in \mcal{O}\pars{t^{-1+\epsilon}} .
    \end{equation}
\end{corollary}
\begin{proof}
    We want to apply Theorem~\ref{thm:rates_abstrac_SPL} with $\chi_t = t^{1-\epsilon}$.
    For this, observe that large enough $t$
    $$
        \sigma_t = c_{\mathrm{bias},1} s_t -  s_t^2 \tfrac{L+C_{\mathrm{R}}}{2} c_{\mathrm{var},1}>0
        \quad \text{and}\quad
        a_t = 1-2\PL \sigma_t < 1 .
    $$
    To show that $\chi_{t+1}\xi_t \in \ell^1$, observe that $\xi_t\asymp s_t^2 + \beta_t \in \mcal{O}(t^{2\delta - 2})$ and thus $\chi_{t+1}\xi_t\in \mcal{O}\pars{(t+1)^{1-\varepsilon}t^{2\delta-2}}$.
    Since
    $$
        (t+1)^{1-\epsilon} t^{2\delta-2}
        \leq 2^{1-\epsilon} t^{1-\epsilon+2\delta-2}
        \in \mathcal{O}( t^{2\delta-\epsilon-1}),
    $$
    and $\varepsilon > 2\delta$ by definition, it holds that $\chi_{t+1}\xi_t\in\ell^1$. 

    Finally, we need to show the condition $a_t\chi_{t+1}\leq \chi_t$ for some $t\geq T$. 
    For this we follow the proof of Lemma~1 from~\cite{liu2022sure}, which relies on the basic inequality
    $$
        (t+1)^{1-\epsilon}
        \leq t^{1-\epsilon} + (1-\epsilon)t^{-\epsilon}.
    $$
    To see this define $g(x)= x^{1-\epsilon}$ and observe that the derivative $g'(x) = (1-\epsilon)x^{-\epsilon}$ is decreasing.
    Hence, by the mean value theorem, there exists $y\in(t,t+1)$ such that
    $$
        (t+1)^{1-\epsilon} - t^{1-\epsilon}
        = g'(y)\leq g'(t)
        = (1-\epsilon)t^{-\epsilon}.
    $$
    Now observe that $a_t = (1 + \beta t^{2\delta-2} - \alpha t^{\delta-1})$ with $\alpha = 2\PL$ and $\beta = 2\PL\tfrac{L+C_{\mathrm{R}}}{2}(1-\tfrac{1}{n})$.
    We thus know that asymptotically
    \begin{align}
        a_t\chi_{t+1}
        &= (1 + \beta t^{2\delta-2}  -\alpha t^{\delta-1}) (t+1)^{1-\epsilon} \\
        &\leq (1 + \beta t^{2\delta-2} - \alpha t^{\delta-1}) (t^{1-\epsilon} + (1 - \epsilon)t^{-\epsilon} ) \\
        &\asymp (1 - \alpha t^{\delta-1}) t^{1-\epsilon} \\
        &\le t^{1-\epsilon} \\
        &= \chi_t .
    \end{align}
    This means, that there exists $T>0$ such that $a_t\chi_{t+1}\le \chi_t$ for all $t\ge T$.    
    Hence, we obtain $\mathcal{L}(u_t)-\mathcal{L}_{\text{min},\mathcal{M}} \in \mathcal{O}(\chi_t^{-1})= \mathcal{O}(t^{\epsilon-1})$.
\end{proof}

In the best case setting, Corollary~\ref{cor:exponential-convergence} shows that the proposed algorithm exhibits the same exponential rate of convergence as deterministic gradient descent setting for strongly convex functions.
In the worst-case setting, Corollary~\ref{cor:quasi-linear-convergence} guarantees the same algebraic rate of convergence as classical SGD for strongly convex functions.
Notably, the step size sequence still has to satisfy the Robbins--Monro conditions $s\in\ell^2$ and $s\not\in\ell^1$.
However, in contrast to Theorem~\ref{thm:stationary_point}, where the optimal convergence rate is obtained when the step sizes decrease with the slowest rate $s_t \in\mcal{O}(t^{-1/2})$, Corollary~\ref{cor:quasi-linear-convergence} suggests that the rate of convergence is maximised when the step size approaches $\mcal{O}\pars{t^{-1}}$.
This is due to the fact that under assumption~\eqref{eq:C-retraction} Theorem~\ref{thm:rates_abstrac_SPL} guarantees an exponential convergence that is perturbed by the empirical projection and retraction errors.
Decreasing the step size as quickly as possible ensures that the empirical projection error decreases as quickly as possible.
The retraction error $\beta_t$ has to be bounded accordingly.

\begin{remark}
    Combining bounds~\eqref{eq:rate_nonrecovery_BLSCR} and~\eqref{eq:rate_nonrecovery_LSSPLCR} yields
    $$
        \min_{t=1,\ldots,\tau} \{\mcal{L}(u_t) - \mcal{L}_{\mathrm{min},\mcal{M}} \}
        \le \begin{cases}
            \tau^{-1+2\varepsilon}, & \varepsilon\in(0, \tfrac{1}{3}), \\
            \tau^{-\varepsilon}, & \varepsilon\le(\tfrac{1}{3}, \tfrac12) .
        \end{cases}
    $$
    for any step size sequence $s_t := t^{-1+\varepsilon}$ with $\varepsilon\in(0,\tfrac12)$.
\end{remark}

\begin{remark}
    In the almost sure setting, the presented algorithm can achieve near-optimal minimax rates for standard smoothness classes when the quasi-projection with optimal sampling (cf.~Section~\ref{sec:quasi-projection}) is used.
    To see this, consider a Hilbert space $\mcal{V} \subseteq \mcal{H} := L^2(\mcal{X})$ and suppose that there exists a sequence of subspaces $\mcal{V}_d \subseteq \mcal{V}$, indexed by their dimension, such that
    $$
        \norm{u - P_{\mcal{V}_d} u} \lesssim d^{-\alpha} .
    $$
    Moreover, consider the least squares loss $\mcal{L}(v) := \tfrac12 \norm{u-v}^2$ for some $u\in \mcal{V}$ on the model class $\mcal{M} := \mcal{V}$.
    For the linear space $\mcal{T}_t$, choose the $d$-dimensional space $\mcal{V}_{d}$.
    Drawing the sample points optimally and choosing $n_t = d$ constant, means that $c_{\mathrm{var},1} = \frac{2d-1}d \le 2$ and $c_{\mathrm{var},2} = \frac{d}{d} = 1$ and therefore $\sigma_t$ and $a_t$ in  Corollary~\ref{cor:quasi-linear-convergence} only depend on $s_t$.
    This means that, under the assumption $\|(I-P_t)g_t\|\in\ell^\infty$, the corollary can be applied with $s_t \in \mathcal{O}(t^{\delta-1})$ yielding the bound
    \begin{equation}
    \label{eq:sampling_number_crude}
        \norm{u - u_t}^2 \lesssim d^{-2\alpha} + t^{-1+\varepsilon} ,
    \end{equation}
    where $\delta\in(0,\tfrac12)$ and $\varepsilon\in(2\delta, 1)$ can be chosen arbitrarily small.
    For the sake of simplicity, we ignore the $\varepsilon$ and $\delta$ exponents in the following.
    To equilibrate both terms in~\eqref{eq:sampling_number_crude}, we need a number of steps $t = d^{2\alpha}$. 
    The total number of samples after $t$ steps is given by $N_t = tn_t$ with $n_t = d = t^{1/2\alpha}$.
    This means that $N_t = t^{1+1/2\alpha}$ and, consequently, $t^{-1} = N^{-2\alpha/(2\alpha + 1)}$.
    The resulting algorithm thus exhibits the rate
    $$
        \norm{u - u_t}^2 \lesssim N_t^{-2\alpha/(2\alpha+1)} .
    $$
    Concerning the condition, $\|(I-P_t)g_t\|\in\ell^\infty$ we note that $g_t = u_t - u$ and thus $\norm{(I-P_t)g_t} = \norm{u - P_t u} \le \norm{u}$ is uniformly bounded.
    
    For the Sobolev spaces $H^{k}([0,1]^D)$ and the mixed Sobolev spaces $H^{k,\mathrm{mix}}([0,1]^D)$, the resulting rates in $L^2$-norm (ignoring log terms) are
    $$
        N_t^{-2k/(2k + D)}
        \qquad\text{and}\qquad
        N_t^{-2k/(2k+1)},
    $$
    respectively.
    These rates are minimax-optimal \cite{Donoho1998Jun,gine2016mathematical,suzuki2018adaptivity} (ignoring log terms), and it is therefore not necessary to adapt the sample to the current gradient $g$ as discussed in Remark~\ref{rmk:optimal_sampling}.
\end{remark}

\subsubsection{Almost sure convergence under assumption~\eqref{eq:mu-PL}}
\label{section:almost_sure_PL}

Although Theorem~\ref{thm:rates_abstrac_SPL} ensures almost sure convergence, the strong assumption~\eqref{eq:mu-PL_strong} may not be easy to verify in practice since the condition depends on the model class $\mathcal{M}$.
Hence, we use this section to briefly discuss almost sure convergence under the weaker~\eqref{eq:mu-PL} assumption.

\begin{theorem}
\label{thm:descent_mu-PL}
    Assume that the loss function satisfies assumptions~\eqref{eq:lower_bound},~\eqref{eq:L-smooth},~\eqref{eq:mu-PL} and~\eqref{eq:C-retraction} with $\beta_t\equiv 0$.
    Assume further that for the current step $t\ge0$, the $x_1, \ldots, x_n$ are i.i.d.\ samples from $\mu_t$ given $\mcal{F}_t$. 
    Define the constants
    $$
        \kappa_t :=  \frac{\norm{g_t}^2}{\norm{P_tg_t}^2},
        \qquad
        \tilde\sigma_t := c_{\mathrm{bias},1} s_t - s_t^2 \tfrac{L+C_{\mathrm{R}}}{2} \pars{c_{\mathrm{var},1} + c_{\mathrm{var},2}(\kappa_t - 1)}
        \qquad\text{and}\qquad
        a_t := 1 - 2\PL \tfrac{\tilde\sigma_t}{\kappa_t}
        .
    $$
    If $\kappa_t\in[1,\infty)$ is finite and $\tilde{\sigma}_t > 0$, then
    $$
        \mbb{E}\bracs*{\mcal{L}(u_{t+1}) - \mcal{L}_{\mathrm{min}}\,|\,\mcal{F}_t} \le a_t \pars*{\mcal{L}(u_t) - \mcal{L}_{\mathrm{min}}}.
    $$
    Moreover, if $a_t \le \bar{a} < b < 1$ for all $t$, 
    then the descent scheme~\eqref{eq:descent_scheme} converges almost surely with
    $$
        \mcal{L}\pars{u_t} - \mcal{L}_{\mathrm{min}}
        \lesssim b^{t} .
    $$
\end{theorem}
\begin{proof}
    Using Theorem~\ref{thm:descent_quasi_projection} with $c_{\mathrm{bias},2}=0$ and $\beta_t \equiv 0$ and subtracting $\mathcal{L}_{\mathrm{min}}$ we obtain
    \begin{alignat}{2}
        \mbb{E}\bracs{\mcal{L}\pars{u_{t+1}} - \mcal{L}_{\mathrm{min}}\,|\, \mcal{F}_t}
        &\le (\mcal{L}\pars{u_t} - \mcal{L}_{\mathrm{min}})
        &&- \pars{c_{\mathrm{bias},1} s_t - s_t^2 \tfrac{L+C_{\mathrm{R}}}{2}c_{\mathrm{var},1}} \norm{P_tg_t}^2 \\
        &&&+ s_t^2 \tfrac{L+C_{\mathrm{R}}}{2} c_{\mathrm{var},2}\norm{(I-P_t)g_t}^2 .
    \end{alignat}
    Substituting $\norm{(I - P_t) g_t}^2 = \norm{g_t}^2 - \norm{P_t g_t}^2 = (\kappa_t - 1)\norm{P_t g_t}^2$ and $\norm{P_t g_t}^2 = \frac{1}{\kappa_t}\norm{g_t}^2$, we obtain
    \begin{alignat}{2}
        \mbb{E}\bracs{\mcal{L}\pars{u_{t+1}} - \mcal{L}_{\mathrm{min}}\,|\, \mcal{F}_t}
        &\le (\mcal{L}\pars{u_t} - \mcal{L}_{\mathrm{min}})
        &&- \pars{c_{\mathrm{bias},1} s_t - s_t^2 \tfrac{L+C_{\mathrm{R}}}{2}c_{\mathrm{var},1}} \tfrac{1}{\kappa_t} \norm{g_t}^2 \\
        &&&+ s_t^2 \tfrac{L+C_{\mathrm{R}}}{2} c_{\mathrm{var},2} \tfrac{\kappa_t - 1}{\kappa_t} \norm{g_t}^2 \\
        &\le (\mcal{L}\pars{u_t} - \mcal{L}_{\mathrm{min}})
        &&- \pars*{c_{\mathrm{bias},1} s_t - s_t^2 \tfrac{L+C_{\mathrm{R}}}{2} \pars{c_{\mathrm{var},1} + c_{\mathrm{var},2}(\kappa_t - 1)}} \tfrac1{\kappa_t} \norm{g_t}^2 \\ 
        &\le (\mcal{L}\pars{u_t} - \mcal{L}_{\mathrm{min}}) &&- \tfrac{\tilde{\sigma}_t}{\kappa_t} \norm{g_t}^2 . 
    \end{alignat}
    Finally, employing assumption~\eqref{eq:mu-PL} yields the result
    \begin{align*}
        \mbb{E}\bracs{\mcal{L}\pars{u_{t+1}} - \mcal{L}_{\mathrm{min}}\,|\, \mcal{F}_t} 
        &\le \pars{\mcal{L}\pars{u_t} - \mcal{L}_{\mathrm{min}}}
        - \frac{\tilde\sigma_t}{\kappa_t}\norm{g_t}^2 \\
        &\le \pars{\mcal{L}\pars{u_t} - \mcal{L}_{\mathrm{min}}}
        - 2\PL\frac{\tilde\sigma_t}{\kappa_t}\pars{\mcal{L}\pars{u_t} - \mcal{L}_{\mathrm{min}}} \\
        &= a_t (\mcal{L}\pars{u_{t+1}} - \mcal{L}_{\mathrm{min}})
        .
    \end{align*}%
    To prove the assertion about almost sure exponential convergence, we lean heavily on the proof of Lemma~1 from~\cite{liu2022sure}.
    We define $Y_t := \mcal{L}\pars{u_t} - \mcal{L}_{\mathrm{min}}$ and recall that
    $$
        \mbb{E}\bracs*{Y_{t+1} \,\big|\, \mcal{F}_t} \le a Y_t .
    $$
    Now define $c := \frac{a}{b} \in\pars{0,1}$ and $\hat{Y}_t := b^{-t}Y_t$ and observe that
    $$
        \mbb{E}\bracs*{b^{-(t+1)} Y_{t+1} \,\big|\, \mcal{F}_t}
        \le b^{-(t+1)} a Y_t 
        = c b^{-t} Y_t
        \qquad\Leftrightarrow\qquad
        \mbb{E}\bracs*{\hat{Y}_{t+1} \,\big|\, \mcal{F}_t}
        \le c \hat Y_t
        = \hat Y_t - (1 - c)\hat Y_t .
    $$
    By Proposition~\ref{prop:RobbinsSiegmund}, it holds that $\pars{1-c}\sum_{t=1}^\infty \hat{Y}_t < \infty$ a.s.\ and hence $\pars{Y_t}_{t\in\mbb{N}}\in \ell^1_{\gamma}$ with $\gamma_t := b^{-t}$.
    The weighted Stechkin inequality~\cite[Lemma~15]{trunschke2023weighted} hence implies $Y_t \lesssim \gamma_t^{-1} = b^{t}$.
\end{proof}

\begin{remark}
   The preceding theorem could be generalised to $\beta_t\ne 0$ in which case $\mcal{L}(u_t) - \mcal{L}_{\mathrm{min}} \lesssim \max\braces{b^t, \beta_t t^\varepsilon}$ for arbitrarily small $\varepsilon$.
\end{remark}

The exponential convergence in Theorem~\ref{thm:descent_mu-PL} relies on the uniform bound $a_t \le \bar{a} < 1$, which induces constraints on the step size $s_t$ and the value of $\kappa_t$.
An easy algebraic manipulation reveals that these constraints can only be satisfied when
$$
    \kappa_t \le \bar\kappa < \infty
$$
is uniformly bounded.
For a fixed point $u_t$, the value of $\kappa_t$ is a deterministic quantity measuring how close $u_t$ lies to a suboptimal stationary point.
The assumption $\kappa_t \le \bar\kappa < \infty$ is hence a deterministic condition on the model class $\mcal{M}$ and the loss $\mcal{L}$ and implies that we can not get stuck in a stationary point in $\mcal{M}$ that is not also stationary in $\mcal{H}$.
Conversely, the fact that the condition $\kappa_t\le\bar\kappa$ prohibits convergence to suboptimal local minima also implies that it can only be satisfied for a short sequence of steps $t_0\le t\le t_1$ or in the recovery setting, i.e.\ when $\norm{(I-P_t)g_t} \to 0$.
In fact, Lemma~\ref{lemma:kappa_bounded_reach} provides sufficient conditions for $\kappa_t$ to remain bounded in the least squares recovery setting.

\begin{remark}
\label{rmk:adaptivity}
    Since the rate of convergence depends on the ratio $\kappa_t = \frac{\norm{g_t}^2}{\norm{P_t g_t}^2}$, it may be favourable to enlarge the model class when $\kappa_t$ becomes too large.
    \revision[0]{This idea was recently explored in~\cite{dahmen_expansive}.}
    However, a deeper discussion lies outside the scope of this manuscript.
\end{remark}

\subsection{Almost sure convergence analysis in the biased case}
\label{section:almost_biased}

The behaviour of the algorithm in the biased case can be made precise with the notion of an accumulation set, which is defined analogously to an accumulation point.
\begin{definition}
    A subset $A\subseteq M$ of a metric space $(M, d)$ is called accumulation set of the sequence $x\in M^{\mbb{N}}$, if for all $\varepsilon>0$ and $T\in\mbb{N}$ there exists a $t\ge T$ such that $d(x_t, A) \le \varepsilon$.
\end{definition}
\begin{definition}
\label{def:c_critical_point}
\revision[0]{
    A point $u_t\in\mcal{M}$ is called $c$-quasi-critical point of $\mcal{L}$ in $\mcal{M}$ with associated $\mathcal{T}_t$, if $\norm{P_t\nabla\mcal{L}(u_t)} \le c \norm{(I-P_t)\nabla\mcal{L}(u_t)}$.
    }
\end{definition}

Note that if \revision[0]{$P_t$} is the projection onto the tangent space at \revision[0]{$u_t$}, then a $0$-quasi-critical point is a critical point in the classical sense.

\begin{theorem}
    Assume that the loss function satisfies assumptions~\eqref{eq:lower_bound},~\eqref{eq:L-smooth} and~\eqref{eq:C-retraction} and that the projection estimate satisfies assumption~\eqref{eq:bbv}.
    Moreover, suppose that the model class $\mcal{M}$ is sufficiently regular, such that \revision[0]{$u_t\mapsto P_t$} is continuous, that the cost of the initial guess $\mcal{L}(u_0) \le B$ is almost surely finite, that the sequence of perturbations satisfies $(\beta_t)_{t\in\mbb{N}}\in \ell^1$ and that the step size sequence satisfies the Robbins--Monro condition $(s_t)_{t\in\mbb{N}}\in \ell^2\setminus\ell^1$ and that the gradients $\norm{g_t} \le G$ remain uniformly bounded.
    Then, for every $\varepsilon\in(0, 1)$, the set of $\tfrac{c_{\mathrm{bias},2}}{(1-\varepsilon)c_{\mathrm{bias},1}}$-quasi-critical points is almost surely an accumulation set of the iterates.
\end{theorem}

\begin{proof}
For the sake of simplicity in notation, we define
$
    c:= \tfrac{c_{\mathrm{bias},2}}{c_{\mathrm{bias},1}} .$
Recall that a point $u\in\mcal{M}$ is called a $c$-quasi-critical if \revision[0]{$\norm{P_t\nabla\mcal{L}(u)} \le c \norm{(I-P_t)\nabla\mcal{L}(u)}$}.
Under the given conditions, Theorem~\ref{thm:descent_quasi_projection} guarantees that
\begin{alignat}{2}
    \mbb{E}\bracs{\mcal{L}\pars{u_{t+1}}\,|\, \mcal{F}_0}
    &\le \mcal{L}\pars{u_t}
    &&- s_t \pars*{c_{\mathrm{bias},1} \norm{P_tg_t} -  c_{\mathrm{bias},2} \norm{(I-P_t)g_t}} \norm{P_tg_t} \\
    &&&+ s_t^2 \tfrac{L+C_{\mathrm{R}}}{2} \pars*{c_{\mathrm{var},1} \norm{P_tg_t}^2 + c_{\mathrm{var},2}\norm{(I-P_t)g_t}^2} \\
    &&&+ \beta_t .
\end{alignat}
Now suppose that $c_{\mathrm{bias},1} \norm{P_tg_t} -  c_{\mathrm{bias},2} \norm{(I-P_t)g_t} \le 0$ for infinitely many $t\in\mbb{N}$.
This means that $\norm{P_tg_t} \le c\norm{(I-P_t)g_t}$ for infinitely many $t\in\mbb{N}$ and thus that the set of $c$-quasi-critical points is an accumulation set of the iterates.
If this is not the case, there exists a $T\in\mbb{N}$ such that $c_{\mathrm{bias},1} \norm{P_tg_t} -  c_{\mathrm{bias},2} \norm{(I-P_t)g_t} > 0$ for all $t\ge T$.
We can thus define the non-negative stochastic processes
\begin{align}
    Y_t
    &:= \mcal{L}(u_t), \\
    X_t
    &:= s_t \pars*{c_{\mathrm{bias},1} \norm{P_tg_t} -  c_{\mathrm{bias},2} \norm{(I-P_t)g_t}} \norm{P_tg_t} \\
    Z_t
    &:= s_t^2 \tfrac{L+C_{\mathrm{R}}}{2} \pars*{c_{\mathrm{var},1} \norm{P_tg_t}^2 + c_{\mathrm{var},2}\norm{(I-P_t)g_t}^2} + \beta_t
\end{align}
for all $t\ge T$.
Since $\norm{g_t} \le G$ and $(s_t)\in\ell^2$, we can conclude that $(Z_t)\in\ell^1$ almost surely and Proposition~\ref{prop:RobbinsSiegmund} implies $(X_t)\in\ell^1$ almost surely.

We now distinguish two cases.
In the first case, where $0$ is an accumulation point of $P_t g_t$,
a subsequence of iterates converges to a classical critical point.
We can write this as $h(u_t)\to 0$ with \revision[0]{$h(u_t) := \norm{P_t g_t}$.}
Since \revision[0]{$u_t\mapsto P_t$} is continuous by assumption, the map $h$ is a composition of continuous maps and, therefore, continuous.
Moreover, since $h(u_t) \le G$ by assumption we can conclude that all $u_t$ are elements of the compact set $\bar{\mcal{M}} := \mcal{M}\cap h^{-1}([0, G])$.
Since $h : \bar{\mcal{M}} \to \mbb{R}$ is a continuous function from a compact set to a Haussdorf space, the closed map lemma implies that $h$ is closed.
This means that $\overline{h(V)} \subseteq h(\overline{V})$ for every $V\subseteq\bar{\mcal{M}}$.
For the particular choice $V = \braces{u_t : t\in\mbb{N}}$ we conclude that $0 \in h(\overline{V})$ and therefore $0\in\overline{V}$.
This proves that $0$ is an accumulation point of the sequence of iterates $u_t$.

In the second case, there exists a lower bound $\gamma>0$ such that $\norm{P_tg_t} \ge \gamma$ for all but finitely many $t\in\mbb{N}$.
Since also
$$
    c_{\mathrm{bias},1}\norm{P_tg_t} - c_{\mathrm{bias},2} \norm{(I-P_t)g_t} > 0
$$
for all but finitely many $t\in\mbb{N}$, the $\ell^1$-summability of $(X_t)$ implies that 
$$
    \sum_{t\in\mbb{N}} s_t \pars*{c_{\mathrm{bias},1}\norm{P_tg_t} - c_{\mathrm{bias},2} \norm{(I-P_t)g_t}}
    < \infty .
$$
Because $(s_t)\not\in\ell^1$, this means that $\norm{P_tg_t} - c \norm{(I-P_t)g_t} \to 0$.
To conclude that the set of $c$-quasi-critical points is an accumulation set of the iterates $u_t$, we proceed similarly to the first case and define the function \revision[0]{$f(u_t) := \tfrac{\norm{(I-P_t)g_t}}{\norm{P_t g_t}}$.}
Since $\gamma \le h(u_t) \le G$ by assumption, we can define the compact set $\bar{\mcal{M}} := \mcal{M} \cap h^{-1}([\gamma, G])$ and note that $f : \bar{\mcal{M}} \to \mbb{R}$ is a continuous function.
The close map lemma implies that $f:\bar{\mcal{M}} \to \mbb{R}$ is closed.
This means that $\overline{f(V)} \subseteq f(\overline{V})$ for $V = \braces{u_t : t\in\mbb{N}}$.
Since $f(u_t) \to c^{-1}$ by construction, we conclude that $c^{-1} \in f(\overline{V})$.
To conclude the proof, note that a point $v\in\bar{\mcal{M}}$ is $c$-quasi-critical if $f(v) \ge c^{-1}$.
Since $c^{-1} \in f(\overline{V})$ there must exist a $c$-quasi-critical point $v\in\overline{V}$.
This proves that the set of $c$-quasi-critical points is an accumulation set of the sequence of iterates.

\end{proof}

\section{Examples}
\label{sec:examples}

Consider the Hilbert space $\mcal{H} := L^2(\rho)$, where $\rho$ is either the uniform measure $\rho=\mcal{U}(\mcal{X})$ on the interval $\mcal{X} = [-1, 1]$ or the standard Gaussian measure $\mcal{N}(0, 1)$ on $\mbb{R}$.
We want to minimise the least squares loss
$$
    \mathcal{L}\pars{v} := \tfrac12\norm{u-v}^2
$$
for some $u\in\mcal{H}$, \revision[0]{defined differently in each subsection}.
For this loss,  $\mcal{L}_{\mathrm{min}} = 0$, the Lipschitz smoothness constant is $L=1$ and the (weak) Polyak-\L{}ojasiewicz constant is $\PL = 1$.
\revision[0]{We compare the proposed method for different projectors from section~\ref{sec:setting} evaluating the impact of the optimal sampling methods from appendix~\ref{app:optimal_sampling}.}

\vspace{1em}
\noindent
Related code to reproduce the numerics can be found at \url{https://github.com/ptrunschke/ossgd}.

\subsection{Linear approximation on compact domains}
\label{sec:numerics_lin_bounded}

This subsection considers the uniform distribution on $\mcal{X} =[-1,1]$.
We aim to approximate the target function $u(x) = \exp(x)$ on the linear space $\mcal{V}$ of polynomials of degree less than $3$ and choose for $\mcal{B}_t$   the first $3$ Legendre polynomials.
We choose $\mcal{M} = \mcal{V}$ and $\mcal{T}_t \equiv \mcal{V}$ as well as $R_t(v) := v$.
In this setting, $\mcal{L}_{\mathrm{min},\mcal{M}} \approx 3.6\cdot 10^{-4}$, the strong Polyak-\L{}ojasiewicz property holds with $\PL = 1$ and the retraction error is controlled with $C_{\mathrm{R}}=0$ and $\beta=0$.

Although it does not make much sense to consider iterative algorithms in a linear setting because we could solve the least-squares problem directly,
we use this simple setting to 
showcase the theoretical predictions and
compare the behaviour of different sampling strategies and projection operators without the distraction of local critical points of non-linear model classes.
The plots in this section show $\mcal{L}(u_t) - \mcal{L}_{\mathrm{min}, \mcal{M}}$.

\subsubsection{Sublinear convergence for unbiased projections}

For the empirical operator $P_t^n$, we choose the unbiased quasi-projection defined in~\eqref{eq:quasi-projection}.
The sublinear rates of convergence that are predicted in Corollary~\ref{cor:quasi-linear-convergence} are displayed for $n=1$ and $s_t := \tfrac29 t^{\varepsilon-1}$ with $\varepsilon=0.1$ in Figure~\ref{fig:sublinear_convergence}.
We observe a faster pre-asymptotic rate in the case of optimal sampling (Figure~\ref{fig:sublinear_convergence_optimal}), transitioning to the asymptotic sublinear rate at $t = 10^3$.
In the light of Corollary~\ref{corollary:exponential_conv_expectation} this is to be expected for the first few steps, where $a_t \le a_*$ for some $a_* < 1$.
\begin{figure}[htb]
    \centering
    \begin{subfigure}[b]{0.49\textwidth}
        \centering
        \includegraphics[width=\textwidth]{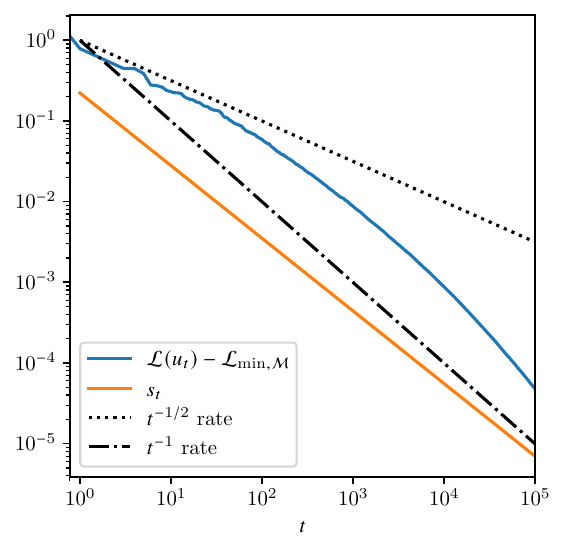}
        \caption{Uniform sampling}
        \label{fig:sublinear_convergence_uniform}
    \end{subfigure}
    \hfill
    \begin{subfigure}[b]{0.49\textwidth}
        \centering
        \includegraphics[width=\textwidth]{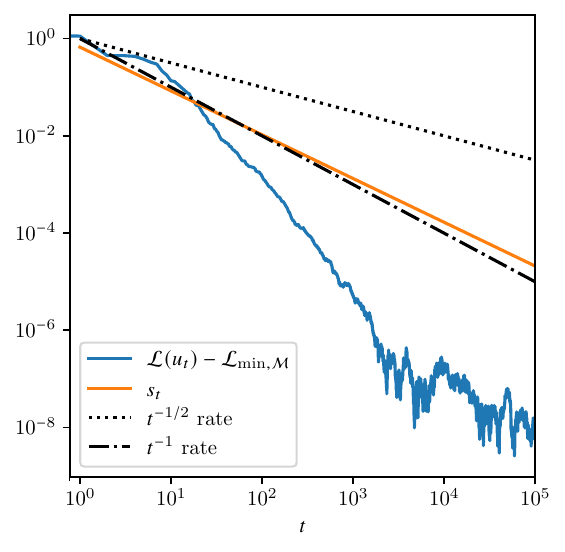}
        \caption{Optimal sampling}
        \label{fig:sublinear_convergence_optimal}
    \end{subfigure}
    \caption{Loss error $\mcal{L}(u_t) - \mcal{L}_{\mathrm{min}, \mcal{M}}$, plotted against the number of steps }
    \label{fig:sublinear_convergence}
\end{figure}

\subsubsection{Exponential convergence for unbiased projections}

We consider again the unbiased quasi-projection defined in~\eqref{eq:quasi-projection}.
Corollary~\ref{cor:convergence_up_to_precision} guarantees that choosing the step size optimally yields exponential convergence up to a fixed precision (in expectation).
Figure~\ref{fig:convergence_up_to_precision} demonstrates this for the choice $n=1$ and the resulting step size $s_t \equiv \tfrac19$.
We can see a stagnation of the loss in the regime of $\mathcal{L}_{\mathrm{min},\mathcal{M}}$ originating from the use of a constant step size.
Since the orthogonal complement $\norm{(I-P_t)g_t}$ does not vanish, neither Corollary~\ref{cor:exponential-convergence} nor Corollary~\ref{cor:quasi-linear-convergence} apply, and we can not expect convergence.

\begin{figure}[htb]
    \centering
    \begin{subfigure}[b]{0.49\textwidth}
        \centering
        \includegraphics[width=\textwidth]{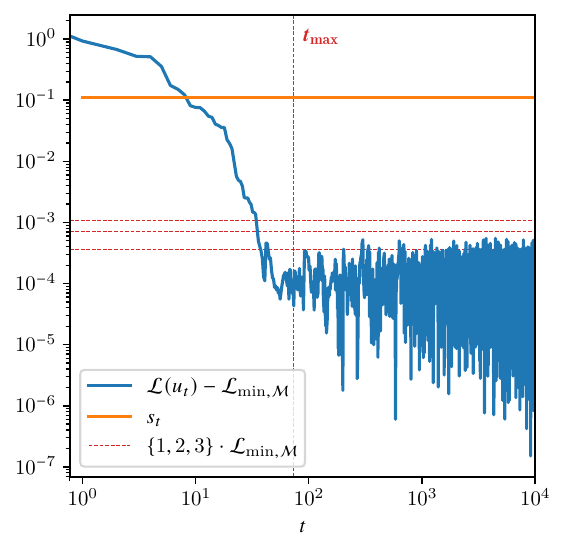}
        \caption{Uniform sampling}
        \label{fig:convergence_up_to_precision_uniform}
    \end{subfigure}
    \hfill
    \begin{subfigure}[b]{0.49\textwidth}
        \centering
        \includegraphics[width=\textwidth]{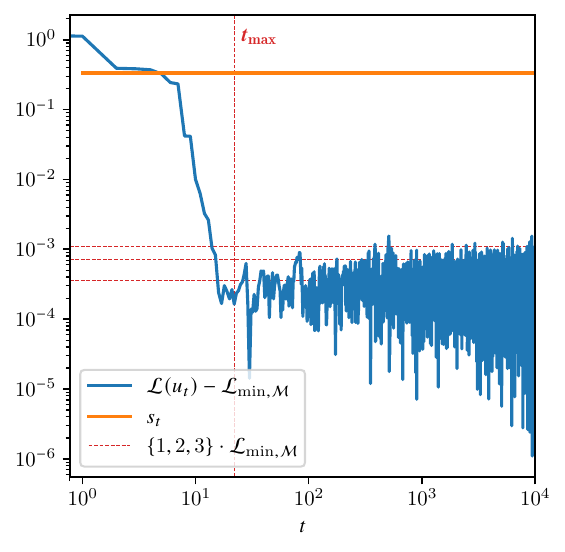}
        \caption{Optimal sampling}
        \label{fig:convergence_up_to_precision_optimal}
    \end{subfigure}
    \caption{Loss error $\mcal{L}(u_t) - \mcal{L}_{\mathrm{min}, \mcal{M}}$, plotted against the number of steps}
    \label{fig:convergence_up_to_precision}
\end{figure}

Since we know that the best approximation error will be attained in expectation after $t_{\mathrm{max}} = 22$ steps \revision[0]{(by using Corollary~\ref{cor:convergence_up_to_precision} with $\varepsilon = \mathcal{L}_{\mathrm{min},\mathcal{M}}$)}, we could then switch from the constant step $s_t = \tfrac19$ to the step size sequence $s_t = \tfrac29 t^{\varepsilon-1}$ with $\varepsilon = 0.1$, guaranteeing sublinear convergence.
Figure~\ref{fig:convergence_up_to_precision_mixed} demonstrates this.

\begin{figure}[htb]
    \centering
    \begin{subfigure}[b]{0.49\textwidth}
        \centering
        \includegraphics[width=\textwidth]{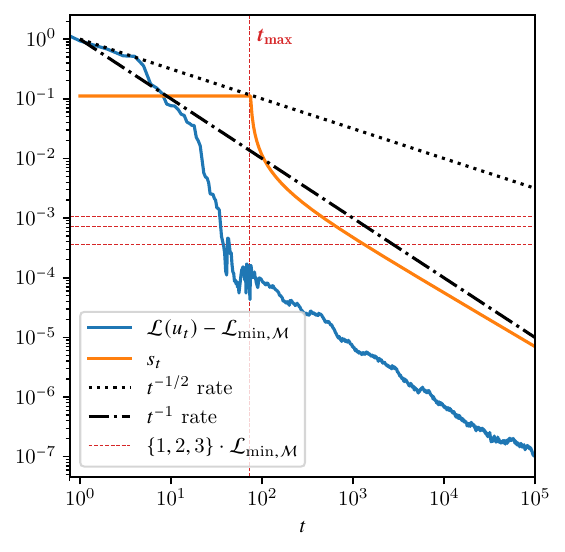}
        \caption{Uniform sampling}
        \label{fig:convergence_up_to_precision_mixed_uniform}
    \end{subfigure}
    \hfill
    \begin{subfigure}[b]{0.49\textwidth}
        \centering
        \includegraphics[width=\textwidth]{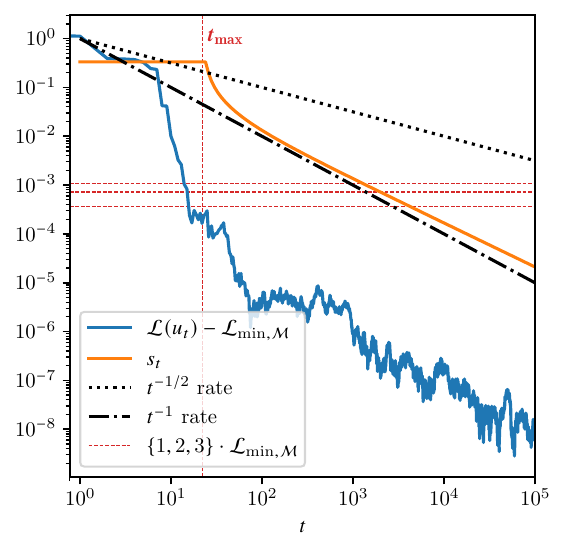}
        \caption{Optimal sampling}
        \label{fig:convergence_up_to_precision_mixed_optimal}
    \end{subfigure}
    \caption{Loss error $\mcal{L}(u_t) - \mcal{L}_{\mathrm{min}, \mcal{M}}$, plotted against the number of steps}
    \label{fig:convergence_up_to_precision_mixed}
\end{figure}

\subsubsection{Convergence for biased projectors}

In this section, we consider two different types of projectors
\begin{enumerate}[label=(\roman*)]
    \item quasi-projections~\eqref{eq:quasi-projection} and
    \item least squares projections~\eqref{eq:least-squares-projection}.
\end{enumerate}
In both cases, the samples are conditioned on the stability event~\eqref{eq:stability} with $\delta=0.5$, which induces a bias even in the quasi-projection.
To ensure that~\eqref{eq:stability} is satisfied with non-vanishing probability, we draw 
$n = 2d_t = 6 > d_t \log(d_t)$ samples via rejection sampling~\cite{Casella2004}.
As discussed in the beginning of section~\ref{sec:Convergence_theorey}, this will produce an algorithm that converges up to the bias~\eqref{eq:least_squares_convergence_to_bias}.
This can be rewritten as
\begin{equation}
\label{ref:bias_oscillation}
    \mcal{L}(u_t) - \mcal{L}_{\mathrm{min},\mcal{M}} \le c_{\mathrm{bias},2}^2 \mcal{L}_{\mathrm{min},\mcal{M}},
\end{equation}
which is demonstrated in Figure~\ref{fig:convergence_up_to_bias}.
Note that while the bias introduced in the conditional stable quasi-projection and the least square projection leads to stagnation, as illustrated in Figure~\ref{fig:convergence_up_to_bias}, the developed quasi-projection strategy (without conditioning) does not stagnate (cf.~Figures~\ref{fig:sublinear_convergence} and~\ref{fig:convergence_up_to_precision_mixed}).

\begin{figure}[htb]
    \centering
    \begin{subfigure}[b]{0.49\textwidth}
        \centering
        \includegraphics[width=\textwidth]{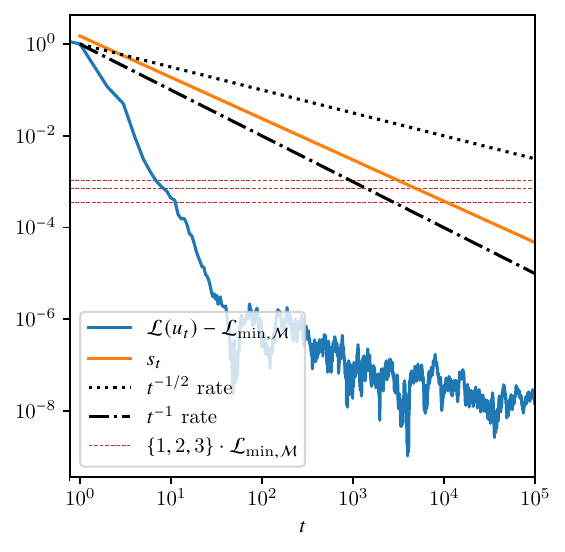}
        \caption{Conditionally stable quasi-projection}
        \label{fig:convergence_up_to_bias_quasi}
    \end{subfigure}
    \hfill
    \begin{subfigure}[b]{0.49\textwidth}
        \centering
        \includegraphics[width=\textwidth]{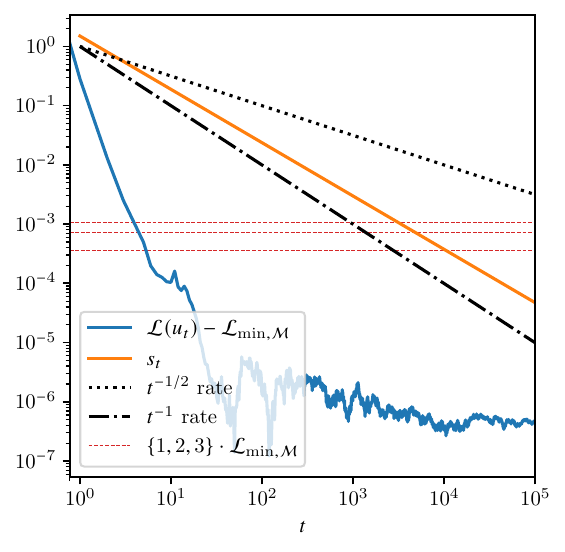}
        \caption{Least squares projection}
        \label{fig:convergence_up_to_bias_least-squares}
    \end{subfigure}
    \caption{Loss error $\mcal{L}(u_t) - \mcal{L}_{\mathrm{min}, \mcal{M}}$, plotted against the number of steps}
    \label{fig:convergence_up_to_bias}
\end{figure}

Since the empirical Gramian is close to the identity, we would expect both the projection and the quasi-projection to be close to the orthogonal projection.
In this case, we could use the deterministically optimal step size $s_t\equiv 1$ and achieve convergence in a single iteration.
This is demonstrated for four levels of stability in Figure~\ref{fig:deterministic_step_sizes}.

\begin{figure}[htb]
    \centering
    \begin{subfigure}[b]{0.24\textwidth}
        \centering
        \includegraphics[width=\textwidth]{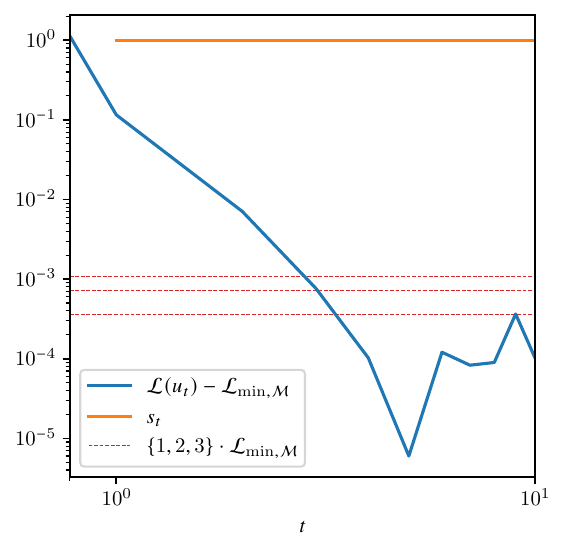}
        \caption{Quasi-projection\\\centering($\delta=\tfrac12$)}
        \label{fig:deterministic_step_sizes_quasi_12}
    \end{subfigure}
    \hfill
    \begin{subfigure}[b]{0.24\textwidth}
        \centering
        \includegraphics[width=\textwidth]{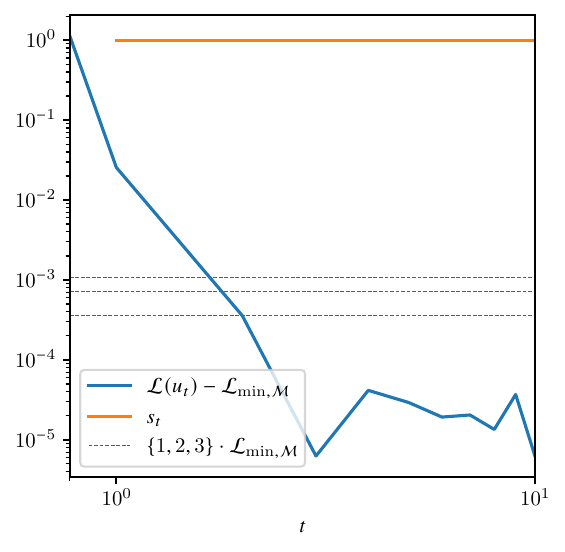}
        \caption{Quasi-projection\\\centering($\delta=\tfrac14$)}
        \label{fig:deterministic_step_sizes_quasi_14}
    \end{subfigure}
    \hfill
    \begin{subfigure}[b]{0.24\textwidth}
        \centering
        \includegraphics[width=\textwidth]{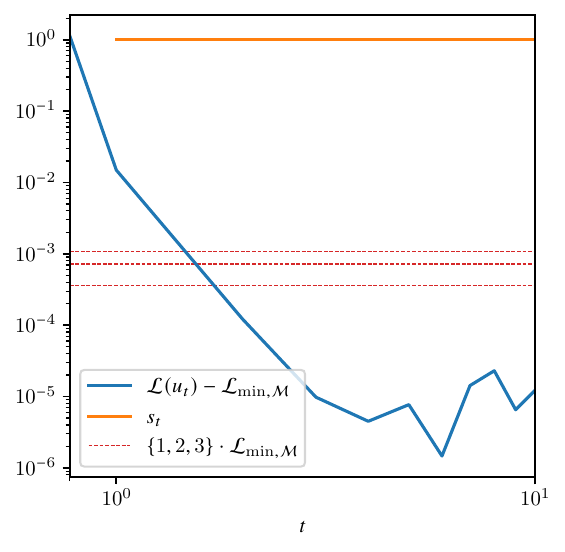}
        \caption{Quasi-projection\\\centering($\delta=\tfrac18$)}
        \label{fig:deterministic_step_sizes_quasi_18}
    \end{subfigure}
    \hfill
    \begin{subfigure}[b]{0.24\textwidth}
        \centering
        \includegraphics[width=\textwidth]{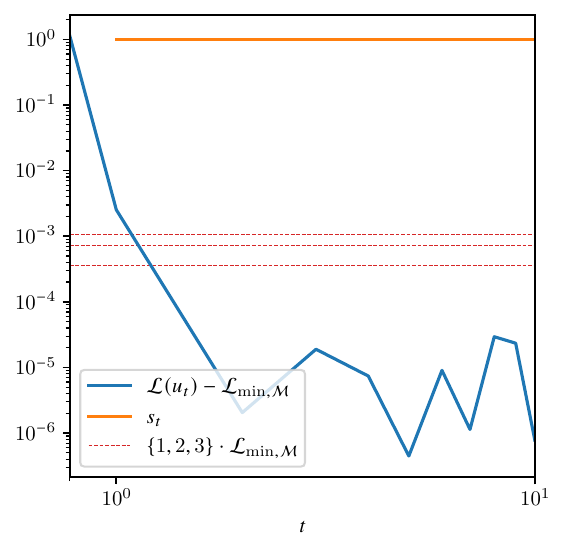}
        \caption{Quasi-projection\\\centering($\delta=\tfrac1{16}$)}
        \label{fig:deterministic_step_sizes_quasi_116}
    \end{subfigure}

    \begin{subfigure}[b]{0.24\textwidth}
        \centering
        \includegraphics[width=\textwidth]{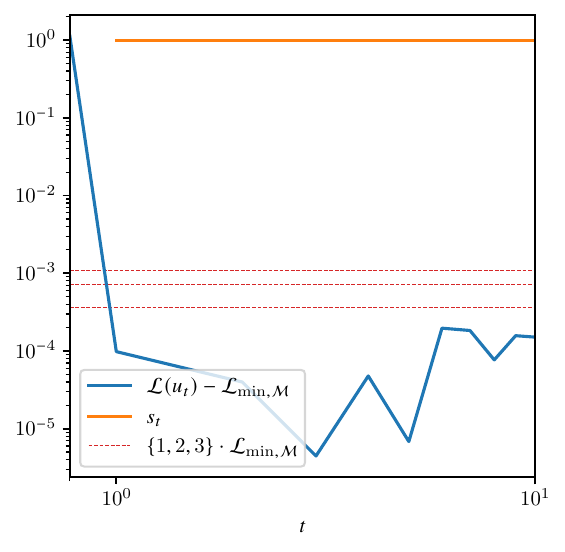}
        \caption{LS projection\\\centering($\delta=\tfrac1{2}$)}
        \label{fig:deterministic_step_sizes_least-squares_12}
    \end{subfigure}
    \hfill
    \begin{subfigure}[b]{0.24\textwidth}
        \centering
        \includegraphics[width=\textwidth]{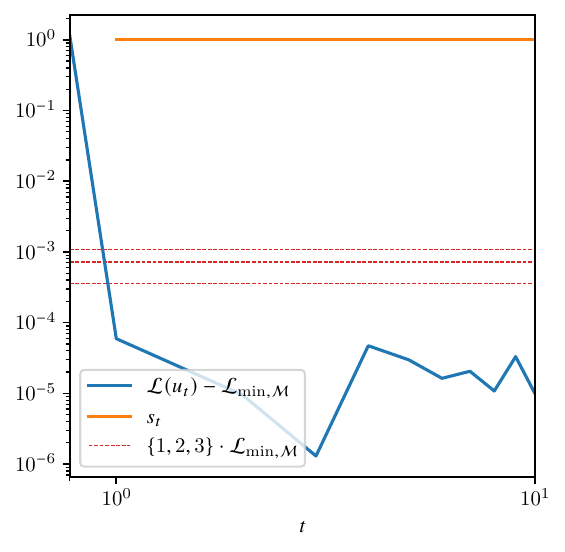}
        \caption{LS projection\\\centering($\delta=\tfrac1{4}$)}
        \label{fig:deterministic_step_sizes_least-squares_14}
    \end{subfigure}
    \hfill
    \begin{subfigure}[b]{0.24\textwidth}
        \centering
        \includegraphics[width=\textwidth]{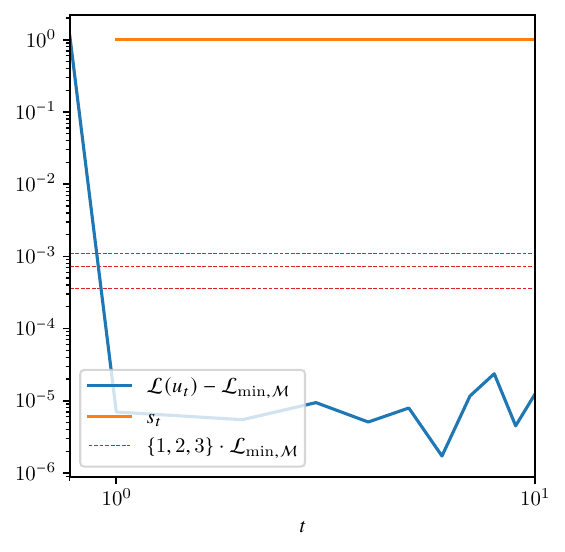}
        \caption{LS projection\\\centering($\delta=\tfrac1{8}$)}
        \label{fig:deterministic_step_sizes_least-squares_18}
    \end{subfigure}
    \hfill
    \begin{subfigure}[b]{0.24\textwidth}
        \centering
        \includegraphics[width=\textwidth]{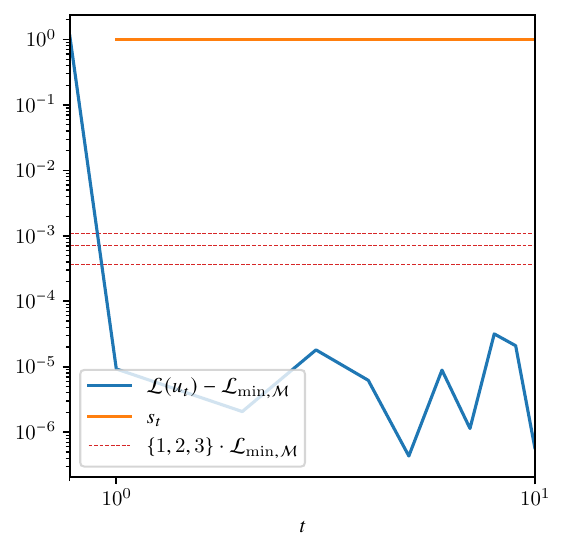}
        \caption{LS projection\\\centering($\delta=\tfrac1{16}$)}
        \label{fig:deterministic_step_sizes_least-squares_116}
    \end{subfigure}

    \caption{Loss error $\mcal{L}(u_t) - \mcal{L}_{\mathrm{min}, \mcal{M}}$, plotted against the number of steps}
    \label{fig:deterministic_step_sizes}
\end{figure}

\subsection{Linear approximation on unbounded domains}
\label{sec:numerics_lin_unbounded}

This subsection considers for $\rho$ the Gaussian distribution on $\mbb{R}$.
We again aim to approximate the target function $u(x) = \exp(x)$.
The linear space $\mcal{V}$ is the space of polynomials of degree less than $7$, and we choose for $\mcal{B}_t$ the first $7$ Hermite polynomials.
We choose $\mcal{M} = \mcal{V}$ and $\mcal{T}_t \equiv \mcal{V}$ as well as $R_t(v) := v$.
In this setting, $\mcal{L}_{\mathrm{min},\mcal{M}} \approx 3.1\cdot 10^{-4}$, the strong Polyak-\L{}ojasiewicz property holds with $\PL = 1$ and the retraction error is controlled with $C_{\mathrm{R}}=0$ and $\beta=0$.
For the empirical operator $P_t^n$, we choose the unbiased quasi-projection defined in~\eqref{eq:quasi-projection}.
However, as discussed in section~\ref{sec:quasi-projection}, the variance constant $k_t$ is unbounded when the weight function is chosen constant ($w_t\equiv 1$).
Figure~\ref{fig:gaussian_convergence} contrasts the uncontrolled behaviour of unweighted SGD with the sublinear convergence rate guaranteed for the optimal weight function in Corollary~\ref{cor:quasi-linear-convergence}.

This section also highlights another very important aspect of SGD.
People might think that the Gramian $G$ may be well estimated by Monte Carlo just because the sought function $u$ is smooth.
As illustrated in Figure~\ref{fig:gaussian_convergence}, this is \textbf{by no means the case}!
We have to use optimal sampling in order to promote stability of the stochastic gradient.

\begin{figure}[htb]
    \centering
    \begin{subfigure}[b]{0.325\textwidth}
        \centering
        \includegraphics[width=\textwidth]{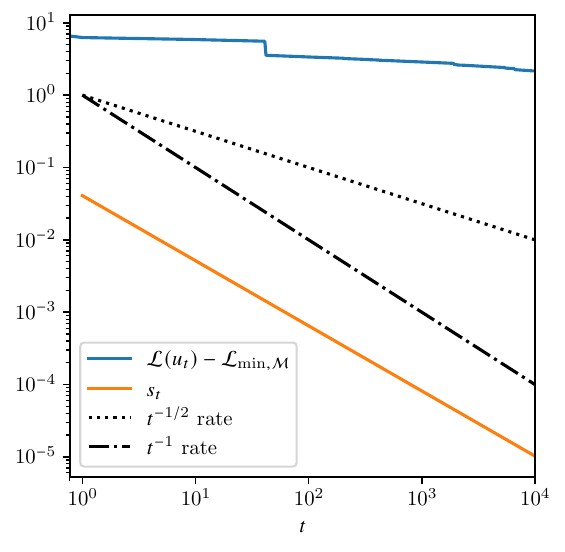}
        \caption{Gaussian sampling, $n=1$}
        \label{fig:gaussian_convergence_gaussian_sampling_n1}
    \end{subfigure}
    \hfill
    \begin{subfigure}[b]{0.325\textwidth}
        \centering
        \includegraphics[width=\textwidth]{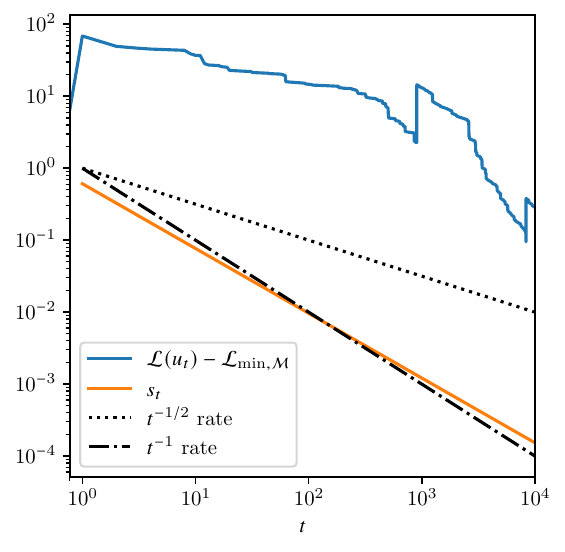}
        \caption{Gaussian sampling, $n=21$}
        \label{fig:gaussian_convergence_gaussian_sampling_n21}
    \end{subfigure}
    \hfill
    \begin{subfigure}[b]{0.325\textwidth}
        \centering
        \includegraphics[width=\textwidth]{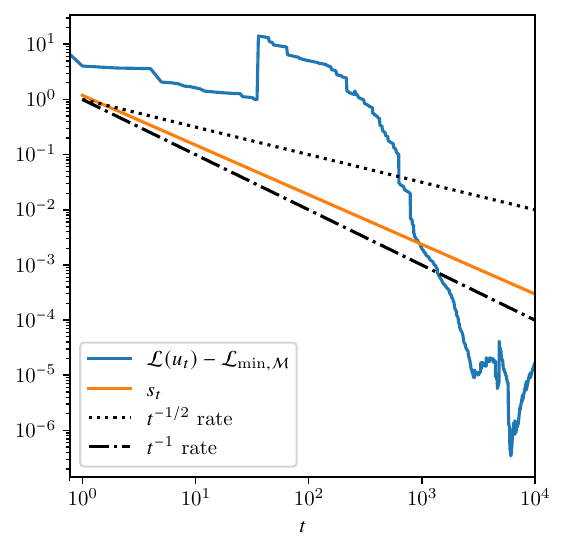}
        \caption{Gaussian sampling, $n=70$}
        \label{fig:gaussian_convergence_gaussian_sampling_n70}
    \end{subfigure}

    \begin{subfigure}[b]{0.325\textwidth}
        \centering
        \includegraphics[width=\textwidth]{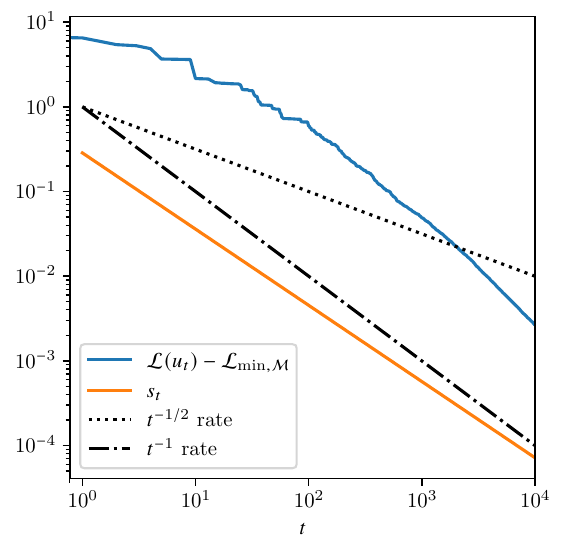}
        \caption{Optimal sampling, $n=1$}
        \label{fig:gaussian_convergence_optimal_sampling_n1}
    \end{subfigure}
    \hspace{0.0125\textwidth}
    \begin{subfigure}[b]{0.325\textwidth}
        \centering
        \includegraphics[width=\textwidth]{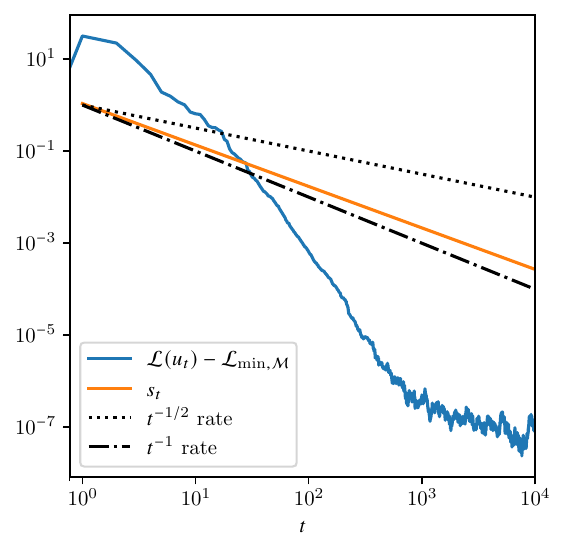}
        \caption{Optimal sampling, $n=7$}
        \label{fig:gaussian_convergence_optimal_sampling_n7}
    \end{subfigure}
    \hfill

    \caption{Loss error $\mcal{L}(u_t) - \mcal{L}_{\mathrm{min}, \mcal{M}}$, plotted against the number of steps}
    \label{fig:gaussian_convergence}
\end{figure}

\subsection{Shallow neural networks}
\label{sec:shallow}

As a final example, we consider model classes of shallow neural networks.
For a continuously differentiable activation function $\sigma \colon \mathbb{R}\to\mathbb{R}$, we define the shallow network $\Phi_\theta\colon \mcal{X} = \mathbb{R}^D\to\mathbb{R}$ of width $m\in\mbb{N}$ and with parameters $\theta := (A_1, b_1, A_0, b_0) \in \Theta:=\mathbb{R}^{1\times m} \times \mathbb{R} \times \mathbb{R}^{m\times D} \times \mathbb{R}^m \simeq \mathbb{R}^{m(D+2)+1}$ by
$$
    \Phi_{(A_1, b_1, A_0, b_0)}(x)
    = A_1 \sigma(A_0x + b_0) + b_1 .
$$
As a model class, we thus define
$$
    \mcal{M} := \braces{\Phi_\theta \,:\, \theta \in \Theta
    } .
$$
Some basic properties of this function class are presented in appendix~\ref{app:shallow_nns}.
We present these properties for general $D\in\mbb{N}$ to showcase the applicability of our theory in the broader setting.
However, for the sake of simplicity and ease of implementation, we restrict the experiments to the case $D=1$.

For the linearisations, we choose $\mcal{T}_t = \mcal{T}_{\Phi_{\theta_t}}$, as defined in appendix~\ref{app:shallow_nns}.

To define a retraction, note that $u_t = \Phi_{\theta_t}$ and $P_t^n g_t = \sum_{k=1}^d (\vartheta_t)_k (\partial_{(\theta_t)_k} \Phi_{\theta_t})$ and thus, by Taylor's theorem,
$$
    \Phi_{\theta_t} + s_t P_t^ng_t = \Phi_{\theta_t + s_t\vartheta_t} + \mathcal{O}(s_t^2)
$$
for sufficiently small $s_t$.
We can thus define the retraction operator as the parameter space $R_t(u_t + s_tP_t^ng_t) := \Phi_{\theta_t + s_t\vartheta_t}$,
as usually done when learning neural networks. 

\revision[0]{We show in appendix~\ref{app:shallow_nns} that this retraction satisfies assumption~\eqref{eq:C-retraction} with $C_{\mathrm{R}} = 0$ and  we can estimate the bias term by
$$
    \beta_t(s_t) := \mathrm{Lip}_t(\sqrt{2\lambda_t} + \varepsilon(s_t) + s_t\norm{P_t^ng_t}) \varepsilon(s_t) .
$$
Here $\lambda_t$ is an iterative estimate of $\mathcal{L}(u_t)$ given by
$$
    \lambda_{t} := \tfrac12\lambda_{t-1} + \tfrac14 \norm{u - u_t}_n^2 , \qquad\text{with}\qquad \lambda_{0} = \tfrac12 \norm{u - u_t}_n^2
$$
and $\varepsilon(s_t) := \norm{u_{t+1} - \bar{u}_{t+1}}$ is the retraction error for step size $s_t$.
Since $u_{t+1}$, $\bar{u}_{t+1}$ and $P_t^ng_t$ are shallow neural networks, the norms $\varepsilon(s_t)$ and $\norm{P_t^ng_t}$ can be computed analytically.

With the estimate $\beta_t(s_t)$, we can choose the step size $s_t$ as to maximise the descent factor $\sigma_t$ from Theorem~\ref{thm:convergence_in_expectation} while satisfying a given bound on this bias term $\beta_t(s_t) \le \bar\beta_t$ for an a priori chosen sequence $\bar\beta_t$.
\footnote{Note, however, the resulting step size $s_t$ is no longer adapted to the filtration $\mcal{F}_t$ since it depends on the update direction $P_t^ng_t$.}
Note that, since the condition $\beta_t(s_t)\le\bar\beta_t$ restricts the step size, the sequence $\bar\beta_t$ must not decrease too quickly as we need $s\not\in\ell^1$.
In the present setting, this gives the condition $\bar\beta_t\not\in\ell^2$.
}
As a final adaptation to increase the convergence rate, we argue that the retraction error $\beta_t(s_t)$ should always be smaller than the current loss estimate $\lambda_t$ \revision[0]{and select the step size to ensure}
\begin{equation}
\label{eq:chosen_retraction_error_threshold}
    \beta_t(s_t) \le \bar\beta_t := \min\braces{\lambda_t, s_{\mathrm{opt}} t^{-1/2}},
\end{equation}
where $s_{\mathrm{opt}}$ is the optimal (linear) step size, maximising $\sigma_t$.

\paragraph{Experiments}
In the following, we list some plots for the approximation of $u(x) := \sin(2\pi x)$ on the interval $[-1, 1]$ by a shallow neural network of width $m = 20$ and the RePU activation $\sigma(x) := \max\{x, 0\}^2$.
The following experiments were performed and all use $n = 10m = 200$ sample points per step.
\begin{enumerate}[label=(\alph*)]
    \item $P_t^n$ is the least squares projection with optimal sampling and stability constant $\delta = \tfrac12$.
    The step size is chosen such that $\beta_t(s_t) \le \beta_t$ with $\beta_t$ defined as in~\eqref{eq:chosen_retraction_error_threshold}.
    The results are presented in Figure~\ref{fig:NGD_ls_threshold}.
    \item $P_t^n$ is the least squares projection with optimal sampling and stability constant $\delta = \tfrac12$.
    The step size is chosen as $s_t := s_{\mathrm{opt}} t^{-1/2}$.
    The results are presented in Figure~\ref{fig:NGD_ls_decreasing}.
    \item $P_t^n$ is the quasi-projection with optimal sampling and \textbf{without conditioning}.
    The step size is chosen such that $\beta_t(s_t) \le \beta_t$ with $\beta_t$ defined as in~\eqref{eq:chosen_retraction_error_threshold}.
    The results are presented in Figure~\ref{fig:NGD_quasi_threshold}.
    \item We compare these to a standard SGD.
    $P_t^n$ is the non-projection with uniform sampling.
    The step size is chosen as $s_t := s_{\mathrm{opt}} t^{-1/2}$.
    The results are presented in Figure~\ref{fig:SGD}.
\end{enumerate}

\revision[0]{From Figures~\ref{fig:NGD_ls_threshold} and~\ref{fig:SGD}, it is clear that NGD with optimal sampling outperforms standard SGD.
(NGD attains a loss of $10^{-8}$ while SGD stagnates at $10^{-1}$.)
Moreover, comparing figures~\ref{fig:NGD_ls_threshold} and~\ref{fig:NGD_ls_decreasing}, it can be seen that the adaptive sample size strategy proposed in equation~\eqref{eq:chosen_retraction_error_threshold} clearly outperforms a standard Robbins--Monro step size.
Finally, and quite noteworthy, a NGD using quasi-projections (Figure~\ref{fig:NGD_quasi_threshold}) seems to stagnate earlier than a NGD using least squares projections (Figure~\ref{fig:NGD_ls_threshold}).
While it is not clear why this is the case, we suspect this to be a consequence of projections being more accurate than quasi-projections.
}

\begin{figure}[htb]
    \centering
    \begin{subfigure}[b]{0.475\textwidth}
        \centering
        \includegraphics[width=\textwidth]{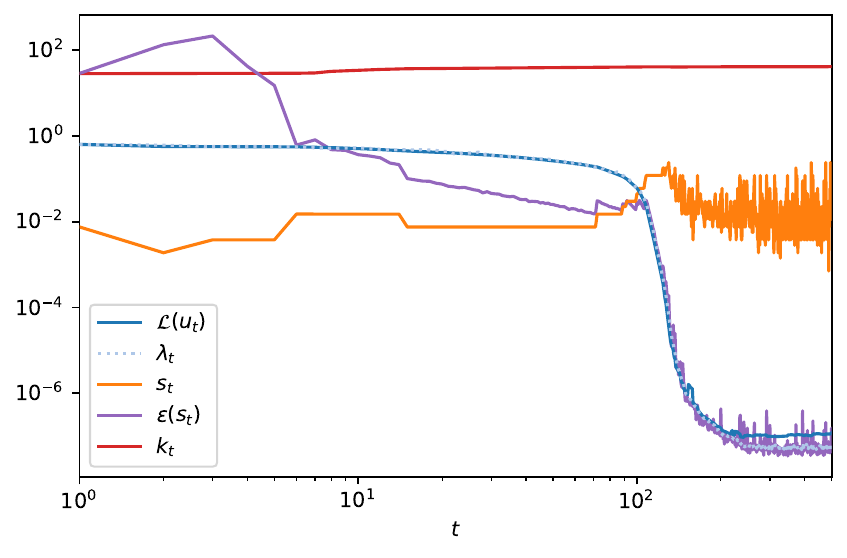}
        \caption{NGD with optimal sampling, least squares projection and adaptive step sizes}
        \label{fig:NGD_ls_threshold}
    \end{subfigure}
    \hfill
    \begin{subfigure}[b]{0.475\textwidth}
        \centering
        \includegraphics[width=\textwidth]{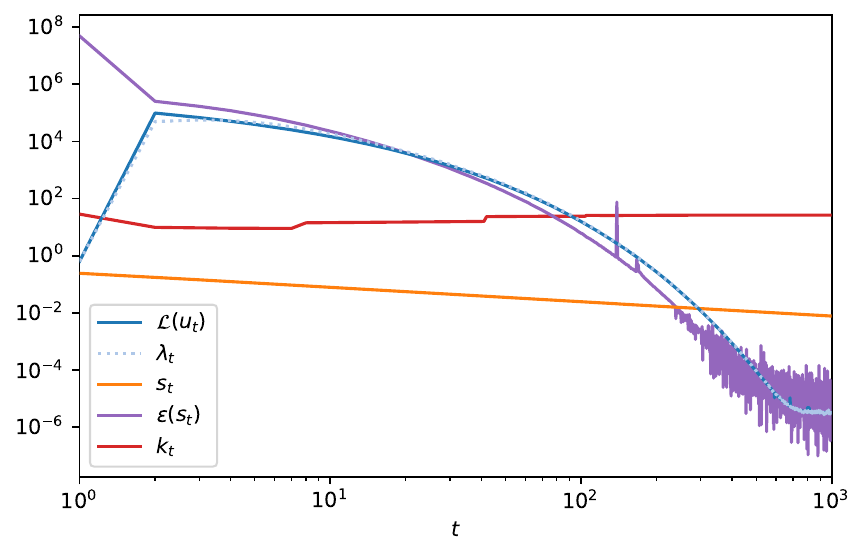}
        \caption{NGD with optimal sampling, least squares projection and decreasing step sizes}
        \label{fig:NGD_ls_decreasing}
    \end{subfigure}

    \begin{subfigure}[b]{0.475\textwidth}
        \centering
        \includegraphics[width=\textwidth]{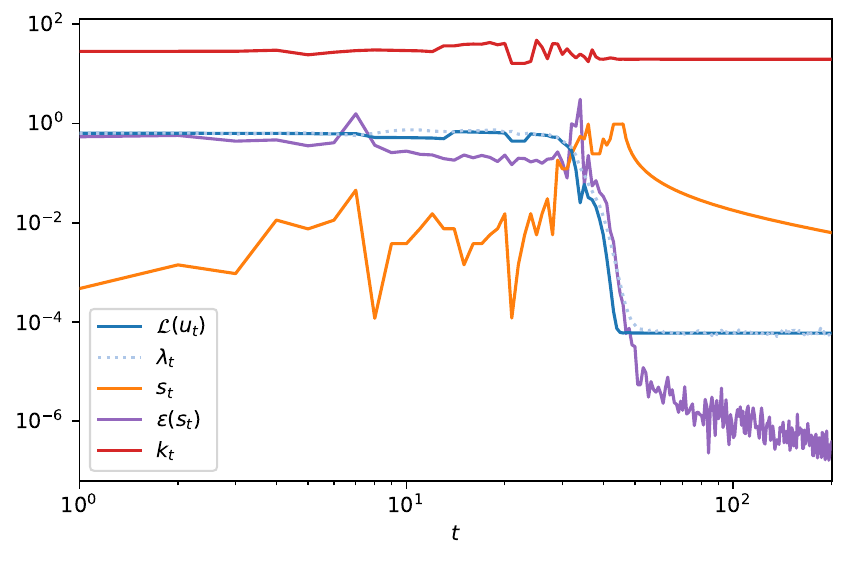}
        \caption{NGD with optimal sampling, quasi-projection and adaptive step sizes}
        \label{fig:NGD_quasi_threshold}
    \end{subfigure}
    \hfill
    \begin{subfigure}[b]{0.475\textwidth}
        \centering
        \includegraphics[width=\textwidth]{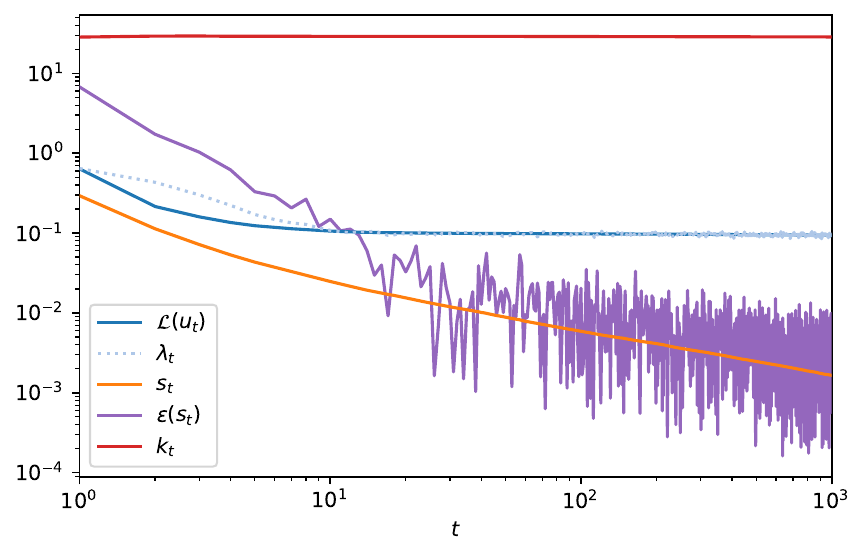}
        \caption{SGD with uniform sampling and decreasing step sizes}
        \label{fig:SGD}
    \end{subfigure}

    \caption{Loss $\mcal{L}(u_t)$ plotted against the number of steps $t$.}
\end{figure}

\section{Conclusion and Outlook}
\revision[0]{
This work analyses a natural gradient descent-type algorithm where the descent directions are given by empirical estimates of projected gradients.
We discuss several projection estimators based on optimal sampling strategies like quasi-projections, least squares projections and least squares projections based on determinantal point processes (or volume sampling).
We obtain convergence rates in (conditional) expectation and almost surely.
Numerical experiments confirm the theoretical findings for least squares problems with linear model classes on bounded and unbounded domains, and with shallow neural networks.
Overall, the proposed scheme is a first order (stochastic) decent method and consequently we do not claim convergence to a global minimizer in the general non-convex setup. In such setting, the analysis may be extended to local versions of the assumptions including weak P\L{} conditions. 

\paragraph{Optimal sampling.}
Key to our results is the use of optimal sampling to bound the variance of the estimators.
Since sampling from these measures is a non-trivial open problem (as discussed in the introduction), the presented strategy can only be compared to others in specific settings.
For linear cases, this is done in Remark~\ref{rmk:complexity}. 
We note that using optimal sampling is, in general, not a trade-off between fast convergence and cost per iteration. 
Convergence of the iteration requires the variance of the estimates to remain bounded.
This will, in general, require some version of optimal sampling.

\paragraph{Step size selection.}
Since, intuitively, a properly defined empirical projection should be close to the true projection, we would expect similar step size requirements as for deterministic descent methods.
Our results confirm this intuition in many cases.
This immediately raises the question of whether we can adapt the presented theory to other step size selection strategies.
This bares the problem that, when exact evaluations of the risk functional are inaccessible, line search strategies can not be applied directly.
Moreover, to allow using a priori bias and variance bounds known for many projectors, the step sizes should be stochastically independent of the empirical projectors $P_t^n$.
}

\section*{Data Availability}
All data used in this study are synthetic and generated as part of the numerical experiments.

\section*{Acknowledgements}
This project is funded by the ANR-DFG project COFNET (ANR-21-CE46-0015).
This work was partially conducted within the France 2030 framework program, Centre Henri Lebesgue ANR-11-LABX-0020-01.
RG acknowledges support by the DFG MATH+ project AA5-5 (was EF1-25) -
\textit{Wasserstein Gradient Flows for Generalised Transport in Bayesian Inversion}.

Our code makes extensive use of the Python packages \texttt{numpy}~\cite{numpy}, \texttt{scipy}~\cite{scipy}, and \texttt{matplotlib}~\cite{matplotlib}. 

\bibliographystyle{abbrv}
\bibliography{references}

\newpage
\appendix
\section{Using the non-projection yields SGD}
\label{app:leibniz}

To see that using the non-projection in~\eqref{eq:descent_scheme} yields the standard SGD method when $\mcal{H} = L^2(\mathcal{X},\rho)$ and $\ell(v;x) = \tilde \ell(v(x); x)$ with $\tilde \ell(\cdot ; x) : \mathbb{R} \to \mathbb{R}$ satisfies certain regularity assumptions, recall that SGD approximates $\mcal{L}(v) $ by 
$$
    \mcal{L}_n(v)
    := \frac{1}{n} \sum_{i=1}^n \ell(v; x_i)
    = \frac{1}{n} \sum_{i=1}^n \tilde \ell(v(x_i); x_i) ,
$$
where the $x_i$ are samples from $\rho$, and $w_t = 1$.
Now assume that $\mathcal{M} = \{ F(\theta) : \theta \in \mathbb{R}^d\}$ with a differentiable parametrisation map $F : \mathbb{R}^d \to \mcal{H}$.
Denoting by $\theta_t\in\mbb{R}^d$ the parameter's value of $u_t = F(\theta_t)$, we can define $\varphi_k = \partial_{k} F(\theta_t)$ and $\mcal{T}_t := \operatorname{span}\braces{\varphi_k : 1 \le k \le d} 
$.
SGD then uses the chain rule $D_{\theta} \mcal{L}_n(F(\theta)) = D_v\mcal{L}_n(F(\theta)) \circ D_{\theta} F(\theta)$ to compute the gradient
$$
    \pars{\nabla_{\theta} \mcal{L}_n(F(\theta)), e_k}
    = \inner{D_v\mcal{L}_n(F(\theta)), \inner{D_{\theta} v(\theta), e_k}}
    = \inner{D_v\mcal{L}_n(F(\theta)), \varphi_k} .
$$
Recall that the Fr\'echet derivative of $\mcal{L}_n$ in $\mcal{H} = L^2\pars{\mcal{X},\rho}$ is given by
$$
    \inner{D_v \mcal{L}_n(v), \varphi_k} 
    = \frac{1}{n} \sum_{i=1}^n  \tilde\ell'(v(x_i); x_i) \varphi_k(x_i) = \pars{\tilde\ell'(v(\bullet); \bullet), \varphi_k}_n .
$$
On the other hand, by Leibniz's integral rule (Proposition~\ref{prop:leibniz_integral_rule} and Corollary~\ref{cor:leibniz_application}), it holds for sufficiently regular loss functions that $\nabla_v \mcal{L}(v) = \tilde \ell'(v(\bullet); \bullet)$.
Hence, if the loss is sufficiently regular and $\mcal{H} = L^2(\mcal{X},\rho)$, we can write the gradient of SGD as
\begin{equation}
\label{eq:etahat_is_parameter_grad}
    \pars{\nabla_{\theta} \mcal{L}_n(F(\theta)), e_k}  
    = \pars{\nabla_v\mcal{L}(F(\theta)), \varphi_k}_n   
    = \hat\zeta_k ,
    \qquad\text{i.e.}\qquad
    \hat\zeta = \nabla_{\theta} \mcal{L}_n(F(\theta)) .
\end{equation}

\begin{proposition}[Leibniz integral rule] \label{prop:leibniz_integral_rule}
    Let $\mcal{V}$ be a Banach space and $(\mcal{X}, \Sigma, \rho)$ be a measure space.
    Suppose that $\ell:\mcal{V}\times\mcal{X}\to\mbb{R}$ satisfies the following conditions:
    \begin{enumerate}
        \item The function $\ell\pars{v;\bullet} : \mcal{X}\to\mbb{R}$ is Lebesgue integrable for each $v\in\mcal{V}$.
        \item The function $\ell(\bullet; x) : \mcal{V}\to\mbb{R}$ is Fr\'echet differentiable for almost all $x\in\mcal{X}$.
        \item There exists a Lebesgue integrable function $\alpha:\mcal{X}\to\mbb{R}$ such that $\norm{D_v\ell\pars{v;x}}_{\mcal{V}^*} \le \alpha\pars{x}$ for all $v\in\mcal{V}$ and almost all $x\in\mcal{X}$.
    \end{enumerate}
    Then, for all $v\in\mcal{V}$ and $h\in\mcal{V}$.
    \begin{equation}
        \inner*{
        D_v \pars*{\int\ell\pars{v;x} \dx[\rho(x)]}, h}
        = \int \inner{D_v \ell\pars{v;x} h} \dx[\rho(x)] .
    \end{equation}
\end{proposition}
\begin{proof}
    See~\cite{kammar2016leibniz}.
\end{proof}

\begin{corollary}
\label{cor:leibniz_application}
    Consider a loss function
    $$
        \mcal{L}(v) := \int \ell(v; x) \dx[\rho(x)]
    $$
    with $\ell$ satisfying the conditions of Proposition~\ref{prop:leibniz_integral_rule} and $\ell(v; x) = \tilde \ell(v(x) ; x)$ with $\tilde \ell : \mathbb{R} \times \mcal{X} \to \mathbb{R}$.
    Let $\tilde \ell'$ denote the derivative of $\ell'$ with respect to its first argument.
    Then the gradient of $\mcal{L}$ in $L^2(\rho)$ is given by
    $$
        \nabla_v\mcal{L}(v)
        = \tilde\ell'(v(\bullet); \bullet)
        .
    $$
\end{corollary}
\begin{proof}
    By Proposition~\ref{prop:leibniz_integral_rule}, it holds for any $v,h\in L^2(\rho)$, that
    \begin{align}
        \pars{\nabla_v \mcal{L}(v), h}
        &= \inner{D_v \mcal{L}(v), h}
        = \int \inner{D_v \ell(v; x), h} \dx[\rho(x)] \\
        &= \int \tilde\ell'(v(x); x) h(x) \dx[\rho(x)]
        = \pars{\tilde\ell'(v(\bullet); \bullet), h} . 
    \end{align}
\end{proof}

\section{Non- and quasi-projection}
\label{app:quasi-projection}

In this section, we let $\mcal{B}_t = \{b_1,...,b_d\}$ be a generating system of the $\mcal{F}_t$-measurable space $\mcal{T}_t$ of dimension $d_t\le d$. We let $G \in \mathbb{R}^{d\times d}$ be the corresponding Gramian matrix, with $G_{kl} := (b_k, b_l)$. When $\mcal{B}_t$  is an orthonormal basis, which is the case for quasi-projection, we have $d=d_t$ and $G$ is equal to the identity matrix.  We let  $\lambda_*(G)$ denote the smallest strictly positive  eigenvalue of $G$, and $\lambda^*(G)$ denote the largest eigenvalue of $G$.
The subsequent lemmas are generalisations of the basic results from~\cite{cohen2017optimal} and~\cite{trunschke2023convergence}.

\begin{lemma}
\label{lem:quasi-projection_unbiased}
    Assume that  $x_1, \ldots, x_n$ are i.i.d.\ samples from $\mu_t$ given $\mcal{F}_t$ and let $P_t^n$ the non-projection or quasi-projection  defined in~\eqref{eq:non-projection} or~\eqref{eq:quasi-projection} respectively.
    Then, it holds for any $g\in\mcal{H}$ that
    \begin{align}
        \mbb{E}\bracs*{\pars{g, P_t^n g} \mid \mcal{F}_t} &\ge \lambda_*\pars{G} \norm{P_tg}^2 .
    \end{align}
    This bound is tight, and equality holds when $G$ is the identity matrix.
\end{lemma}
\begin{proof}
    Using linearity, we can compute
    \begin{equation}
        \mbb{E}\bracs*{P^n_t g \,\big|\, \mcal{F}_t}
        = \mbb{E}\bracs*{\sum_{k=1}^{d} \pars{g, b_k}_n b_k \,\Bigg|\, \mcal{F}_t}
        = \sum_{k=1}^{d} \mbb{E}\bracs*{\pars{g, b_k}_n \,\big|\, \mcal{F}_t} b_k
        = \sum_{k=1}^{d} \pars{g, b_k} b_k .
    \end{equation}
    Now let $\eta,\theta\in\mbb{R}^{d}$ be defined by $\eta_k := (g, b_k)$ and $\theta := G^+\eta$, with $G^+$ the Moore-Penrose pseudo-inverse of $G$, and observe that
    $$
        \mbb{E}\bracs*{P_t^n g \mid \mcal{F}_t} = \sum_{k=1}^{d} \eta_k b_k
        \qquad\text{while}\qquad
        P_t g = \sum_{k=1}^d \theta_k b_k .
    $$
    This means that
    $$
        \mbb{E}\bracs*{\pars{g, P_t^n g} \mid \mcal{F}_t}
        = \pars{P_t g, \mbb{E}\bracs*{P_t^n g \mid \mcal{F}_t}}
        = \sum_{k,l=1}^{d} \eta_k\theta_k\pars{b_k, b_l}
        = \theta^\intercal G \eta
        = \eta^\intercal G^+G \eta .
    $$
    Let $G = U\Lambda U^\intercal$ be the compact spectral decomposition with $U\in\mbb{R}^{d\times r}$ and $\Lambda\in\mbb{R}^{r\times r}$ and let $\lambda_* := \min\operatorname{diag}(\Lambda) = \lambda_*(G) > 0$.
    Then $G^+ = U\Lambda^{-1} U^\intercal$ and
    \begin{align}
        \eta^\intercal G^+ G \eta
        &= \eta^\intercal U U^\intercal \eta
        \ge \lambda_* \eta^\intercal U \Lambda^{-1} U^\intercal \eta
        = \lambda_* \eta^\intercal G^+ \eta
        = \lambda_* \eta^\intercal G^+ G G^+ \eta
        = \lambda_* \theta^\intercal G \theta
        = \lambda_* \norm{P_t g}^2 .
    \end{align}
\end{proof}

\begin{lemma}
\label{lem:quasi-projection_norm}
    Assume that  $x_1, \ldots, x_n$ are i.i.d.\ samples from $\mu_t$ given $\mcal{F}_t$ and let $P_t^n$ be the non-projection or quasi-projection    defined in~\eqref{eq:non-projection} or~\eqref{eq:quasi-projection} respectively.
    Define $G_{kl} := (b_k, b_l)$ and let $\lambda^*(G)$ denote the largest eigenvalue of $G$.
    Moreover, let $\mfrak{K}_t(x) := \lambda^*\pars*{\sum_{k=1}^{d_t} \pars{L_{x}b_k}\pars{L_xb_k}^\intercal}$ and $k_t := \norm{w\mfrak{K}_t}_{L^\infty(\rho)}$.
    Then, it holds for any $g\in\mcal{H}$ that
    \begin{align}
        \mbb{E}\bracs*{\norm{P_t^n g}^2\,\big|\,\mcal{F}_t}
        &\le \Big(1 - \frac1n\Big)\lambda^*(G)^2 \norm{P_t g}^2 + \frac{\lambda^*(G)}{n}\int w(x)\mfrak{K}(x) \norm{L_x g}_2^2 \dx[\rho(x)] \\
        &\le \Big(1 - \frac1n\Big)\lambda^*(G)^2 \norm{P_t g}^2 + \frac{\lambda^*(G)k_t}{n}\norm{g}^2 .
    \end{align}
    This bound is tight, and equality holds when $G$ is the identity matrix.
\end{lemma}
\begin{proof}
    Let $\eta^n, \eta, \theta\in\mbb{R}^{d}$ be defined by $\eta_k^n := (g, b_k)_n$, $\eta_k := (g, b_k)$ and $\theta := G^+\eta$ and recall that 
    $$
        \norm{P_t^n g}^2
        = (\eta^n)^\intercal G \eta^n
        \le \lambda^*(G) \norm{\eta^n}_2^2
        \qquad\text{and}\qquad
        \norm{\eta}_2^2
        \le \lambda^*(G) \pars{\eta^\intercal G^+ \eta}
        = \lambda^*(G) \pars{\theta^\intercal G \theta}
        = \lambda^*(G) \norm{P_t g}^2 .
    $$
    We start by considering a single term of the sum $\norm{\eta^n}_2^2 = \sum_{k=1}^{d} (g, b_k)_n^2$.
    It holds that
    \begin{align}
        \mbb{E}\bracs*{(g, b_k)_n^2 \,\big|\, \mcal{F}_t}
        &= \frac{1}{n^2} \sum_{i,j=1}^n \mbb{E}\bracs*{w\pars{x_i}(L_{x_i}g)^\intercal (L_{x_i}b_k)w\pars{x_j}(L_{x_j}g)^\intercal (L_{x_j}b_k) \,\big|\, \mcal{F}_t} \\
        &= \begin{multlined}[t]
            \frac{1}{n^2} \sum_{i\ne j=1}^n \mbb{E}\bracs*{w\pars{x_i}(L_{x_i}g)^\intercal (L_{x_i}b_k) \,\big|\, \mcal{F}_t} \mbb{E}\bracs*{w\pars{x_j}(L_{x_j}g)^\intercal (L_{x_j}b_k) \,\big|\, \mcal{F}_t} \\
            + \frac{1}{n^2} \sum_{i=1}^n \mbb{E}\bracs*{\pars*{w\pars{x_i}(L_{x_i}g)^\intercal (L_{x_i}b_k)}^2 \,\big|\, \mcal{F}_t}
        \end{multlined} \\
        &= \frac{n\pars{n-1}}{n^2} \pars{g, b_k}^2
        + \frac{1}{n} \int w(x) (L_x g)^\intercal (L_x b_k) (L_x b_k)^\intercal (L_x g) \dx[\rho(x)] ,
    \end{align}
    where the final equality follows from~\eqref{eq:measure_equality}.
    Summing this over $k$ yields
    \begin{align}
        \mbb{E}\bracs*{\norm{\eta^n}_2^2 \,\big|\, \mcal{F}_t}
        &= \pars*{1 - \frac{1}{n}} \norm{\eta}_2^2
        + \frac{1}{n} \int w(x) (L_x g)^\intercal \pars*{\sum_{k=1}^{d}  (L_x b_k) (L_x b_k)^\intercal} (L_x g) \dx[\rho(x)] \\
        &\le \pars*{1 - \frac{1}{n}} \norm{\eta}_2^2
        + \frac{1}{n} \int w(x) \mfrak{K}(x) \norm{L_x g}^2 \dx[\rho(x)] \\
        &\le \pars*{1 - \frac{1}{n}} \lambda^*(G) \norm{P_t g}^2 + \frac{k_t}{n}  \norm{g}^2, 
    \end{align}
    which ends the proof.
\end{proof}

\begin{lemma}
\label{lem:sgd_bias_variance}
    Assume that $x_1, \ldots, x_n$ are i.i.d.\ samples from $\mu_t$ given $\mcal{F}_t$.
    Let $\mfrak{K}_t(x) := \lambda^*\pars*{\sum_{k=1}^{d} \pars{L_{x}b_k}\pars{L_xb_k}^\intercal}$ and $k_t := \norm{w \mfrak{K}_t}_{L^\infty(\rho)}$.
    Then, the operator $P_t^n$ defined in equation~\eqref{eq:non-projection}  satisfies assumption~\eqref{eq:bbv} with constants 
    \begin{align}
        c_{\mathrm{bias},1} &= \lambda_*\pars{G} ,
        & c_{\mathrm{var},1} &= \tfrac{\lambda^*(G)^2(n-1) + \lambda^*(G) k_t}{n} , \\
        c_{\mathrm{bias},2} &= 0 ,
        & c_{\mathrm{var},2} &= \tfrac{\lambda^*(G) k_t}n .
    \end{align}
    These bounds are tight.
\end{lemma}
\begin{proof}
    The two bounds $c_{\mathrm{bias},1} = \lambda_*\pars{G}$ and $c_{\mathrm{bias},2} = 0$ follow directly from Lemma~\ref{lem:quasi-projection_unbiased} and Lemma~\ref{lem:quasi-projection_norm} implies
    \begin{align}
        \mbb{E}\bracs*{\norm{P_t^n g}^2\,\big|\,\mcal{F}_t}
        &\le \Big(1 - \frac1n\Big)\lambda^*(G)^2 \norm{P_t g}^2 + \frac{\lambda^*(G)k_t}{n}\norm{g}^2 \\
        &=  \frac{\lambda^*(G)^2(n-1) + \lambda^*(G)k_t}n \norm{P_t g}^2 + \frac{\lambda^*(G)k_t}{n}\norm{(I-P_t) g}^2 .
    \end{align}
\end{proof}

\section{Using the quasi-projection yields NGD}
\label{app:quasi_projection_relevance}

Suppose that $\mathcal{M} = \{v = F\pars\theta : \theta \in \mathbb{R}^d\}$ and $\mcal{T}_t := \operatorname{span}\braces{\varphi_k  : 1 \le k \le d}$ with $\varphi_k := \partial_{k} F\pars\theta$.
Define the Gramian matrix $H_{kl} := (\varphi_k, \varphi_l)$ and suppose that $b = H^+\varphi$ is an orthonormal basis for $\mcal{T}_t$.
Then, we can write
\begin{equation}
\label{eq:quasi-projection-system}
    \eta = H^+\zeta
    \quad\text{with}\quad
    \zeta_k = (g, \varphi_k)
    \quad\qquad\text{and}\qquad\quad
    \hat\eta = H^+\hat\zeta
    \quad\text{with}\quad
    \hat\zeta_k = (g, \varphi_k)_n .
\end{equation}
The quasi-projection's coefficients $\hat\eta$ are hence an estimate of the orthogonal projection's coefficients $\eta$, where $\zeta$ is replaced by the Monte Carlo estimate $\hat\zeta$.
This demonstrates that the algorithm presented in equation~\eqref{eq:descent_scheme} can be seen as a natural gradient descent, as defined in~\cite{nurbekyan2023efficient}.

\section{Optimal sampling}
\label{app:optimal_sampling}

\begin{lemma}
\label{lem:optimal_sampling_density}
    Let $\braces{b_k}_{k=1}^{d_t}$ be an $\mcal{H}$-orthonormal basis for the $d_t$-dimensional subspace $\mcal{T}_t\subseteq \mcal{H}$.
    Recall that
    $$
        \mfrak{K}_t(x) := \lambda^*\pars*{\sum_{k=1}^{d_t} \pars{L_{x}b_k}\pars{L_xb_k}^\intercal}
    $$
    and denote by
    $$
        \mfrak{K}_{\mcal{T}_t}(x) := \sup_{v\in\mcal{T}_t} \frac{\norm{L_x v}_2^2}{\norm{v}}
    $$
    the \emph{generalised inverse Christoffel function} (or \emph{variation function}, cf.~\cite{trunschke2023convergence}). 
    It holds $\mfrak{K}_t = \mfrak{K}_{\mcal{T}_t}$. 
    Moreover, define 
    $$
        \tilde{\mfrak{K}}_t(x)
        := \operatorname{trace}\pars*{\sum_{k=1}^{d_t} \pars{L_{x}b_k}\pars{L_xb_k}^\intercal}
        = \sum_{k=1}^{d_t}  \Vert L_{x}b_k \Vert_2^2.
    $$
    It holds $\mfrak{K}_t \le \tilde{\mfrak{K}}_t$, with equality if $l=1$ in~\eqref{eq:norm}.
    Moreover, it holds that $\Vert \tilde{\mfrak{K}}_t \Vert_{L^1(\rho)} =  \int \tilde{\mfrak{K}}_t \dx[\rho] = d_t$.
\end{lemma}
\begin{proof}
    Let $B_x \in \mathbb{R}^{l\times d_t}$ be the matrix with entries $[B_x]_{jk} = (L_xb_k)_j$. 
    To show the first claim, we compute
    $$
        \mfrak{K}_{\mcal{T}_t}\pars{x}
        := \sup_{\substack{v\in\mcal{T}_t\\\norm{v}=1}} \norm{L_x v}_2^2
        = \sup_{\substack{\bm{v}\in\mbb{R}^{d_t}\\\norm{\bm{v}}_2=1}} \bm{v}^\intercal B_x^\intercal B_x \bm{v}
        = \lambda^*\pars{B_x^\intercal B_x}
        = \norm{B_x}_{2}^2
        = \norm{B_x^\intercal}_{2}^2
        = \lambda^*\pars{B_xB_x^\intercal}
    $$
    and observe that
    $$
        \mfrak{K}_t(x)
        = \lambda^*\pars*{\sum_{k=1}^{d_t} \pars{L_{x}b_k}\pars{L_xb_k}^\intercal}
        = \lambda^*\pars*{B_x B_x^\intercal} .
    $$
    The second claim $\mfrak{K}_t(x) \le \tilde{\mfrak{K}}_t(x)$ is trivial while the last claim follows from 
    \begin{align}
        \int \tilde{\mfrak{K}}_{t} \dx[\rho]
        &
        = \int \operatorname{trace}\pars{B_x^\intercal B_x} \dx[\rho(x)]
        = \operatorname{trace}\pars*{\int B_x^\intercal B_x \dx[\rho(x)]}
        = \operatorname{trace}\pars*{I_{d_t}}
        = d_t .
    \end{align}
\end{proof}

\section{Least squares projection}
\label{app:least-squares-projection}

\begin{lemma}
\label{lemma:QBP_bias}
    Let $\delta\in(0,1)$ and assume that $x_1, \ldots, x_n$ are i.i.d.\ samples from $ \mu_t$ given  $\mcal{F}_t$, conditioned to satisfy the event $\mcal{S}_\delta = \bracs{\norm{I - \hat{G}}_2\le \delta}$.
    Moreover, assume the the probability of the event $\mcal{S}_\delta$ is positive $\mbb{P}\pars{S_\delta\,|\,\mcal{F}_t} = p_{\mcal{S}_\delta} > 0$ and let $\mfrak{K}_t(x) := \lambda^*\pars*{\sum_{k=1}^{d_t} \pars{L_{x}b_k}\pars{L_xb_k}^\intercal}$ and $k_t := \norm{w \mfrak{K}_t}_{L^\infty(\rho)}$.
    Then, the operator $P_t^n$ defined in equation~\eqref{eq:least-squares-projection} satisfies assumption~\eqref{eq:bbv} with constants
    $$
        c_{\mathrm{bias},1} = 1,
        \quad
        c_{\mathrm{bias},2} = \tfrac{\sqrt{k_t}}{(1-\delta)\sqrt{p_{\mcal{S}_\delta}n}} ,
        \quad
        c_{\mathrm{var},1} = \tfrac{1}{(1-\delta)^2p_{\mcal{S}_\delta}}\tfrac{n-1+k_t}{n}
        \quad\text{and}\quad
        c_{\mathrm{var},2} = \tfrac{1}{(1-\delta)^2p_{\mcal{S}_\delta}}\tfrac{k_t}{n} .
    $$
\end{lemma}
\begin{proof}
    Conditioned on this event, the empirical Gramian is invertible and
    \begin{equation}
    \label{eq:Pn_b_estimate}
        \norm{P_t^n g}
        = \norm{\hat{G}^{-1} \hat\zeta}_{2}
        \le \norm{\hat{G}^{-1}}_2\norm{{\hat\zeta}}_{2}
        \le \frac{1}{1-\delta}\norm{\hat\zeta}_{2} .
    \end{equation}
    Since the random variable $\norm{\hat\zeta}_2^2$ is non-negative and 
    $\mbb{P}\bracs{\norm{I - \hat{G}}\le \delta} \geq p_{\mcal{S}_\delta}$ by assumption, it holds that
    \begin{equation}
        \label{eq:b_cond_ub}
        \mbb{E}\bracs*{\norm{\hat\zeta}_2^2\,\big|\, \mcal{F}_t,\mcal{S}_\delta}
        \le \frac{\mbb{E}\bracs*{\norm{\hat\zeta}_2^2\,\big|\,\mcal{F}_t}}{p_{\mcal{S}_\delta}}
        = \frac{1}{p_{\mcal{S}_\delta}}\sum_{k=1}^{d_t} \mbb{E}\bracs*{(g, b_k)_n^2\,\big|\,\mcal{F}_t} .
    \end{equation}
    The last term is the norm of the quasi-projection and is bounded by Lemma~\ref{lem:quasi-projection_norm}.
    This leads to the variance bound
    \begin{equation}
       \mbb{E}\bracs*{\norm{P_t^n g}^2\,\big|\,\mcal{F}_t,\mcal{S}_\delta}
       \le \frac{1}{\pars{1-\delta}^2p_{\mcal{S}_\delta}} \pars*{\pars{1-\tfrac{1}{n}} \norm{P_t g}^2 + \tfrac{k_t}{n}\norm{g}^2} .
    \end{equation}
    To bound the bias terms, we start by applying Jensen's inequality to the variance term above to obtain
    \begin{align}
       \mbb{E}\bracs*{\norm{P_t^n g}\,\big|\,\mcal{F}_t,\mcal{S}_\delta}
       \le \mbb{E}\bracs*{\norm{P_t^n g}^2\,\big|\,\mcal{F}_t,\mcal{S}_\delta}^{1/2}
       \le \frac{1}{\pars{1-\delta}\sqrt{p_{\mcal{S}_\delta}}} \pars*{\pars{1-\tfrac{1}{n}} \norm{P_t g}^2 + \tfrac{k_t}{n}\norm{g}^2}^{1/2} .
    \end{align}
    Now let $g^\perp := \pars{I-P_t}g$ and observe that
    \begin{equation}
        \pars{g, P_t^n g}
        = \pars{g, P_t g} + \pars{g, \pars{P_t^n - P_t}g}
        = \norm{P_t g}^2 + \pars{g, P_t^n g^\perp} ,
    \end{equation}
    where the inner product can be bounded via Cauchy--Schwarz inequality by
    \begin{equation}
        \abs{\pars{g, P_t^n g^\perp}}
        = \abs{\pars{P_t g, P_t^n g^\perp}}
        \le \norm{P_t g}\norm{P_t^n g^\perp} .
    \end{equation}
    Combining the last three equations yields
    \begin{align}
    \label{eq:Lsmooth_empirical_linear_term}
        \mbb{E}\bracs*{\pars{g, P_t^n g}\,\big|\, \mcal{F}_t,\mcal{S}_\delta}
        &\ge \norm{P_t g}^2 - \tfrac{\sqrt{k_t}}{(1-\delta)\sqrt{p_{\mcal{S}_\delta}n}} \norm{P_t g}\norm{\pars{I - P_t}g} .
    \end{align}
\end{proof}

\section{Least squares projection with determinantal point processes and volume sampling}
\label{sec:dpp-projection}

As before, let $\braces{b_k}_{k=1}^{d_t}$ be an $L^2(\rho)$-orthonormal basis of $\mcal{T}_t$ and denote by $b(x) = (b_j(x))_{j=1}^{d_t} \in \mathbb{R}^{d_t}$ the vector of all basis functions.
Let $\mu_t = w_t^{-1} \rho$ be the optimal measure for the i.i.d.\ setting, with $w_t(x)^{-1} = \frac{1}{d_t} \Vert b(x) \Vert_2^2$.
Given a set of points $\boldsymbol{x} := (x_1, \ldots, x_n) \in\mcal{X}^n$, we consider the weighted least-squares projection associated with the empirical inner product 
$$
    (u,v)_n = \frac{1}{n}\sum_{i=1}^n w_t(x_i) u(x_i) v(x_i) .
$$
The points $\boldsymbol{x} := (x_1, \ldots, x_n)$ are drawn from the volume-rescaled distribution $\gamma_n^{\mu_t}$ defined by 
$$
    \dx[\gamma_n^{\mu_t}(\boldsymbol{x})] = \frac{n^{d_t}(n-d_t)!}{n!} \det(\hat G(\boldsymbol{x})) \dx[\mu_t^{\otimes n}(\boldsymbol{x})] ,
$$
where $\hat G(\boldsymbol{x})$ is the empirical Gramian matrix $\hat G(\boldsymbol{x})_{k,l} = (b_k,b_l)_n$ and $\mu_t^{\otimes n}$ denotes the product measure on $\mathcal{X}^n$.
This distribution tends to favour a high likelihood with respect to the optimal sampling measure $\mu_t^{\otimes n}$ in the i.i.d.\ setting and a high value of the determinant of the empirical Gramian matrix.
For $n=d_t$, $\gamma_{d_t}^{\mu_t}$ corresponds to a projection determinantal point process (DPP) with distribution  
$$
    \dx[\gamma_{d_t}(\boldsymbol{x})] = \frac{1}{d_t!} \det(B(\boldsymbol{x})^T B(\boldsymbol{x})) \dx[\rho^{\otimes d_t}(\boldsymbol{x})]
$$
with $B(\boldsymbol{x})_{i,j} = b_j(x_i)$.
Whenever two selected features $b(x_i)$ and $b(x_k)$ are equal for $i \neq k$, the matrix $B(\boldsymbol{x})$ becomes singular and the density 
of $\gamma_{d_t}$ vanishes.
This introduces a repulsion between the points.
For $n>d_t$, a sample from $\gamma_{d_t}^{\mu_t}$ is composed (up to a random permutation) by $d_t$ points $x_1, \ldots, x_{d_t}$ drawn from the DPP distribution $\gamma_{d_t}$, and $n-d_t$ i.i.d.\ points $x_{d_t+1}, \ldots x_n$ drawn from the measure $\mu_t$ (see \cite[Theorem 2.7]{derezinski2022unbiased} or \cite[Theorem 4.7]{nouy2023dpp}).

Given a sample from $\gamma^{\mu_t}_n$, we have the remarkable result that the least-squares projection $P_t^n$ is an unbiased estimate of $P_t$
(see~\cite[Theorem 3.1]{derezinski2022unbiased} and~\cite[Theorem 4.13]{nouy2023dpp})
and that we have a control on the variance~\cite[Theorem~4.14]{nouy2023dpp}.
We recall these results in the following Lemma.    
\begin{lemma}
\label{lemma:DPP_bias}
    Let $\delta\in(0,1)$ and $\zeta > 0 $ and suppose that $n \ge \max\braces{2d_t + 2, 4 d_t \delta^{-2}\log(\zeta^{-1} d_t^2 n)}$ sample points $x_1, \ldots, x_n$ are drawn from $\gamma^{\mu_t}_n$ given $\mcal{F}_t$. 
    Then the weighted least-squares projection $P_t^n$ satisfies assumption~\eqref{eq:bbv} with constants 
    $$
        c_{\mathrm{bias},1} = 1,
        \qquad
        c_{\mathrm{bias},2} = 0 ,
        \qquad
        c_{\mathrm{var},1} = c_{\mathrm{var},2} + 4 (1-\delta)^{-2} 
        \qquad\text{and}\qquad
        c_{\mathrm{var},2} =  4\frac{d_t}{n}(1-\delta)^{-2} + \zeta .
    $$
    For the choice $\zeta = \tfrac{d_t}n$, we obtain $c_{\mathrm{var},2} = \frac{d_t}{n}(1+ 4(1-\delta)^{-2})$.
\end{lemma}
Note that, in contrast to the least squares projection in Section~\ref{sec:least-squares_projection}, the bias and variance bounds in Lemma~\ref{lemma:DPP_bias} hold without the need to condition on the stability event~\eqref{eq:stability}.

\section{Debiased projection}
\label{app:debiased_projection}

Let $\tilde{P}_t : \mcal{H}\to\mcal{T}_t$ be an arbitrary projection operator, such as an empirical or an oblique projection.
Given an unbiased estimate $Q_t : \mcal{H}\to\mcal{T}_t$ of the orthogonal projection $P_t$, such as a quasi-projection, we can define the debiased version of $\tilde{P}_t$ as
\begin{equation}
\label{eq:debiased_projector}
	\bar{P}_t := \tilde{P}_t + Q_t(I-\tilde{P}_t) .
\end{equation}
Denoting by $\mbb{E}_{Q_t}$ the expectation with respect to $Q_t$, it holds that
$$
	\mbb{E}_{Q_t}\bracs*{\tilde{P}_tv + Q_t(I-\tilde{P}_t)v}
    = \tilde{P}_tv + P_t(I-\tilde{P}_t)v
    = P_tv ,
$$
justifying the name ``debiased projection''.

The motivation for this definition is that the variance of estimating the projection of $(I-\tilde{P}_t)g$ is smaller than that of estimating the projection of $g$ (similar to multi-level Monte Carlo).
An intuitive application of this idea is to define $\tilde{P}_t$ as the least squares projector using all previously drawn sample points, while $Q_t$ is a quasi-projection that uses new and independent sample points.

Bounds for the bias and variance constants under the strong assumption that $\norm{\tilde{P}_t} < B$ are provided in the subsequent Lemma~\ref{lem:debiased_P_bounded}.
For the sake of analysis, we also make the simplifying assumption that $\tilde{P}_t$ is drawn independently in every step.
Then a uniform bound of the form $\norm{\tilde{P}_t} \le B$ \textbf{is not required} and the resulting bounds are provided in the subsequent Lemma~\ref{lemma:debiased_P}.

Note that using a quasi-projection for $Q_t$ yields $c_{\mathrm{var},1}(Q_t)-1 = \tfrac{k_t-1}{n} \le \tfrac{k_t}{n} = c_{\mathrm{var},2}(Q_t)$ and therefore $c_{\mathrm{var},2} \le \pars{c_{\mathrm{var},2}(\tilde{P}_t) + 1} \tfrac{k_t}{n}$.
This means that the $c_{\mathrm{var},2}$-term can still be made as small as desired.

\begin{lemma}
\label{lem:debiased_P_bounded}
    Let $Q_t : \mcal{H}\to\mcal{T}_t$ be an unbiased estimate of $P_t$ and $\tilde{P}_t : \mcal{H}\to\mcal{T}_t$ be an $\mcal{F}_t$-measurable projection with uniformly bounded norm $\norm{\tilde{P}_t} < B$.
    Suppose that $Q_t$ satisfies the variance bound
    \begin{align}
    	\mbb{E}_{Q_t}\bracs*{\norm{Q_tw}^2}
    	\le c_{\mathrm{var},1}(Q_t) \norm{Pw}^2 + c_{\mathrm{var},2}(Q_t) \norm{(I-P)w}^2 .
    \end{align}
    Then, the debiased projection $\bar{P}_t := \tilde{P}_t + Q_t(I-\tilde{P}_t)$ satisfies assumption~\eqref{eq:bbv} with constants 
    $$
        c_{\mathrm{bias},1} = 1,
        \qquad
        c_{\mathrm{bias},2} = 0 ,
        \qquad
        c_{\mathrm{var},1} = 1
        \qquad\text{and}\qquad
        c_{\mathrm{var},2} = (c_{\mathrm{var},1}(Q_t)-1)B^2 + c_{\mathrm{var},2}(Q_t). 
    $$
\end{lemma}
\begin{proof}
    We start by observing that
    \begin{align}
    	\mbb{E}_{Q_t}\bracs*{\norm{\bar{P}_tv}^2}
    	&= \mbb{E}_{Q_t}\bracs*{\norm{\tilde{P}_tv + {Q_t}(I-\tilde{P}_t)v}^2} \\
    	&= \norm{\tilde{P}_tv}^2
    	+ 2\mbb{E}_{Q_t}\bracs*{\pars{\tilde{P}_tv,  {Q_t}(I-\tilde{P}_t)v}}
    	+ \mbb{E}_{Q_t}\bracs*{\norm{{Q_t}(I-\tilde{P}_t)v}^2} \\
    	&= \norm{\tilde{P}_tv}^2
    	+ 2\pars{\tilde{P}_tv, (P_t-\tilde{P}_t)v}
    	+ \mbb{E}_{Q_t}\bracs*{\norm{{Q_t}(I-\tilde{P}_t)v}^2} \\
    	&\le
        \begin{multlined}[t]
            \norm{\tilde{P}_tv}^2
        	+ 2\pars{\tilde{P}_tv, (P_t-\tilde{P}_t)v}
        	+ c_{\mathrm{var},1}({Q_t})\norm{P_t(I-\tilde{P}_t)v}^2 \\
        	+ c_{\mathrm{var},2}({Q_t})\norm{(I-P_t)(I-\tilde{P}_t)v}^2
        \end{multlined} \\
    	&=
        \begin{multlined}[t]
            \norm{\tilde{P}_tv}^2
        	+ 2\pars{\tilde{P}_tv, (P_t-\tilde{P}_t)v}
        	+ c_{\mathrm{var},1}({Q_t})\norm{(P_t-\tilde{P}_t)v}^2 \\
        	+ c_{\mathrm{var},2}({Q_t})\norm{(I-P_t)v}^2
        \end{multlined} \\
    	&=
        \begin{multlined}[t]
            \norm{\tilde{P}_tv + (P_t - \tilde{P}_t)v}^2
        	+ (c_{\mathrm{var},1}({Q_t})-1)\norm{(P_t-\tilde{P}_t)v}^2 \\
        	+ c_{\mathrm{var},2}({Q_t})\norm{(I-P_t)v}^2
        \end{multlined} \\
    	&= \norm{P_tv}^2
    	+ (c_{\mathrm{var},1}({Q_t})-1)\norm{\tilde{P}_t(I-P_t)v}^2
    	+ c_{\mathrm{var},2}({Q_t})\norm{(I-P_t)v}^2
    	.
    \end{align}
    Using the uniform bound $\norm{\tilde{P}_t}\le B$ now yields the result
    \begin{align}
    	\mbb{E}_{Q_t}\bracs*{\norm{\bar{P}_tv}^2}
    	&\le \norm{P_tv}^2
    	+ \pars*{(c_{\mathrm{var},1}(Q_t)-1)\norm{\tilde{P}_t}^2 + c_{\mathrm{var},2}(Q_t)} \norm{(I-P_t)v}^2 .
    \end{align}
\end{proof}

\begin{lemma}
\label{lemma:debiased_P}
    Let $Q_t : \mcal{H}\to\mcal{T}_t$ be an unbiased estimate of $P_t$ and $\tilde{P}_t : \mcal{H}\to\mcal{T}_t$ be a projection.
    Suppose that $Q_t$ and $\tilde{P}_t$ satisfy the variance bounds
    \begin{align}
    	\mbb{E}_{Q_t}\bracs*{\norm{{Q_t}w}^2}
    	&\le c_{\mathrm{var},1}({Q_t}) \norm{P_tw}^2 + c_{\mathrm{var},2}({Q_t}) \norm{(I-P_t)w}^2 \\
    	\mbb{E}_{\tilde{P}_t}\bracs*{\norm{\tilde{P}_tw}^2}
    	&\le c_{\mathrm{var},1}(\tilde{P}_t) \norm{P_tw}^2 + c_{\mathrm{var},2}(\tilde{P}_t) \norm{(I-P_t)w}^2
    	.
    \end{align}
    Then, the debiased projection $\bar{P}_t := \tilde{P}_t + Q_t(I-\tilde{P}_t)$ satisfies assumption~\eqref{eq:bbv} with constants 
    $$
        c_{\mathrm{bias},1} = 1,
        \quad
        c_{\mathrm{bias},2} = 0 ,
        \quad
        c_{\mathrm{var},1} = 1
        \quad\text{and}\quad
        c_{\mathrm{var},2} = \pars*{(c_{\mathrm{var},1}(Q_t)-1) c_{\mathrm{var},2}(\tilde{P}_t) + c_{\mathrm{var},2}(Q_t)}. 
    $$
\end{lemma}
\begin{proof}
    Recalling the tower property
    $$
        \mbb{E}\bracs*{\norm{\bar{P}_tv}^2}
        = \mbb{E}_{\tilde P_t}\bracs*{\mbb{E}_{Q_t}\bracs*{\norm{\bar{P}_tv}^2}},
    $$
    we can start the proof with the initial estimate
    $$
    	\mbb{E}_{Q_t}\bracs*{\norm{\bar{P}_tv}^2}
    	\le \norm{P_tv}^2
    	+ (c_{\mathrm{var},1}({Q_t})-1)\norm{\tilde{P}_t(I-P_t)v}^2
    	+ c_{\mathrm{var},2}({Q_t})\norm{(I-P_t)v}^2
    $$
    from the proof of Lemma~\ref{lem:debiased_P_bounded}.
    Continuing, we estimate
    \begin{align}
    	\mbb{E}_{\tilde P_t}\bracs*{\norm{\bar{P}_tv}^2}
    	&= \norm{P_tv}^2
    	+ (c_{\mathrm{var},1}(Q_t)-1) \mbb{E}_{\tilde P_t}\bracs*{\norm{\tilde{P}_t(I-P_t)v}^2}
    	+ c_{\mathrm{var},2}(Q_t)\norm{(I-P_t)v}^2 \\
    	&\le \norm{P_tv}^2
    	+ \pars*{(c_{\mathrm{var},1}(Q_t)-1) c_{\mathrm{var},2}(\tilde{P}_t) + c_{\mathrm{var},2}(Q_t)} \norm{(I-P_t)v}^2 .
    \end{align}
\end{proof}

\section{The strong Polyak--\L{}ojasiewicz condition}
\label{sec:A3}

\begin{definition}
    A manifold $\mcal{M}$  is called geodesically convex if for every two points $v, w \in \mcal{M}$ there exists a unique geodesic that connects $v$ and $w$. 
\end{definition}

\begin{definition}
    Let $\mcal{M}$ be geodesically convex.
    A parameterisation of this geodesic $\gamma_{vw}: [0, T_{vw}] \to \mcal{M}$ is said to have unit speed if at every point $t\in(0,T_{vw})$ the derivative $\dot\gamma(t) \in \mathbb{T}_{\gamma(t)}\mcal{M}$ satisfies
    $$
        \|\dot\gamma(t)\|_{\mcal{H}} = 1 ,
    $$
    where we identify the bounded linear map $d_t\gamma\in \mcal{L}(\mbb{R}, \mathbb{T}_{\gamma(t)}\mcal{M})$ with $\dot\gamma(t)\in\mathbb{T}_{\gamma(t)}\mcal{M}$.
\end{definition}

\begin{definition}
    Let $\mcal{M}$ be geodesically convex and $\mcal{L}:\mcal{M}\to\mbb{R}$.
    $\mcal{L}$ is called $\lambda$-strongly geodesically convex on $\mcal{M}$, if for every two points $v,w\in\mcal{M}$ and corresponding unit-speed parameterised geodesic $\gamma_{vw}$, the function $\mcal{L}\circ\gamma_{vw} : [0, T_{vw}] \to \mbb{R}$ is $\lambda$-strongly convex in the conventional sense.
\end{definition}

\revision[0]{A sufficient condition for~\eqref{eq:mu-PL_strong} is provided in the following Lemma~\ref{A3:Lemma_SPL}, which generalises the relation between~\eqref{eq:mu-PL} and strong convexity to the case of manifolds.
It relies on the geodesic convexity of $\mcal{M}$, which generalises the concept of convexity to Riemannian manifolds, and a related strong convexity assumption on $\mcal{L}$.
In contrast to classical convexity, the assumptions of Lemma~\ref{A3:Lemma_SPL} are not easy to verify in the non-Euclidean setting and often only give local convergence guarantees around non-degenerate stationary points.}

\begin{lemma}
\label{A3:Lemma_SPL}
    Suppose that $\mcal{M} \subseteq \mcal{H}$ is a geodesically convex Riemannian submanifold of $\mcal{H}$ without boundary and let $T_{t} = \mbb{T}_{u_t}\mcal{M}$ be the tangent space of $\mcal{M}$ at $u_t$. Furthermore assume that $\mcal{L}_{\mathrm{min},\mcal{M}}\in\mathbb{R}$ is attained in $\mathcal{M}$.
    If $\mcal{L} : \mcal{M}\to\mbb{R}$ is $\lambda$-strongly geodesically convex on $\mcal{M}$, then~\eqref{eq:mu-PL_strong} is satisfied.
\end{lemma}
\begin{proof}
    To prove the assertion, let $u\in\mcal{M}$ be arbitrary and denote by $u_{\mathrm{min},\mcal{M}}$ the minimum of $\mcal{L}$ in $\mcal{M}$.
    Let $\gamma := \gamma_{uu_{\mathrm{min},\mcal{M}}}$ be the unit-speed parameterisation of the geodesic from $u$ to $u_{\mathrm{min},\mcal{M}}$ and define the function $g := \mcal{L}\circ\gamma : [0, T] \to \mbb{R}$ with $T := T_{uu_{\mathrm{min},\mcal{M}}}$.
    When $\mcal{L}$ is $\lambda$-strongly geodesically convex on $\mcal{M}$, then
    \begin{equation}
    \label{eq:g_strong_convexity}
        g(t) \ge g(s) + \dot{g}(s)(t-s) + \tfrac\lambda 2(t-s)^2
    \end{equation}
    for any $t\in[0,T]$ and $s\in(0, T)$.
    We now minimise both sides of this equation with respect to $t$.
    For the left-hand side, this immediately yields $g(T) = \mcal{L}_{\mathrm{\min},\mcal{M}}$.
    To minimise the right-hand side, we compute the first-order necessary optimality condition
    $$
        \dot{g}(s) + \lambda(t-s) = 0 .
    $$
    Solving this equation for $t$ yields $t = s - \tfrac{1}\lambda\dot{g}(s)$.
    Since minimisation preserves the inequality in equation~\eqref{eq:g_strong_convexity}, we can substitute the optimal values that we obtained for each side back into~\eqref{eq:g_strong_convexity}.
    This yields
    \begin{equation}
    \label{eq:g_strong_convexity_minimum}
        \mathcal{L}_{\mathrm{min},\mcal{M}}
        \ge g(s) - \tfrac{1}\lambda\dot{g}(s)^2 + \tfrac1{2\lambda} \dot{g}(s)^2
        = g(s) - \tfrac{1}{2\lambda}\dot{g}(s)^2 .
    \end{equation}
    Let $P_{\dot\gamma(s)}$ denote the $\mcal{H}$-orthogonal projection onto $\operatorname{span}\braces{\dot\gamma(s)}\subseteq \mbb{T}_{u}\mcal{M}$.
    Since $\gamma$ is a unit-speed parameterisation, the chain rule implies that
    $$
        \dot{g}(s)^2
        = \pars*{\mathcal{R} (d_{\gamma(s)}\mcal{L} \circ d_s\gamma)}^2
        = (\nabla\mcal{L}(\gamma(s)), \dot{\gamma}(s))_{\mcal{H}}^2
        = \|P_{\dot\gamma(s)} \nabla\mcal{L}(\gamma(s))\|_{\mcal{H}}^2
        \le \|P_{\mbb{T}_u\mcal{M}} \nabla\mcal{L}(\gamma(s))\|_{\mcal{H}}^2 ,
    $$
    where $\mathcal{R}(d_{\gamma(s)}\mcal{L} \circ d_s\gamma) \in \mbb{R}$ denotes the Riesz representative of the bounded linear map $d_{\gamma(s)}\mcal{L} \circ d_s\gamma\in\mcal{L}(\mbb{R}, \mbb{R})$.
    Substituting this bound into equation~\eqref{eq:g_strong_convexity_minimum} and taking the limit $s\to 0$ proves the claim.
\end{proof}

\section{Example of a non-trivial retraction}
\label{app:retraction}

\revision[0]{A trivial case in which assumption~\eqref{eq:C-retraction} is satisfied with $C_{\mathrm{R}} = 0$ and $\beta=0$ is when $\mcal{M}$ is a (potentially uncountable) union of linear spaces, and $\mcal{T}_t\subseteq\mcal{M}$ with $R_t g := g$ for all $v_t\in\mcal{M}$.
Three prominent examples of this setting are when
\mbox{(i)}  $\mcal{M}$ is a linear space, 
\mbox{(ii)}  $\mcal{M}$ is a set of sparse vectors, or 
\mbox{(iii)}  $\mcal{M}$ is formed by tensor networks (of arbitrary topologies).}

To understand assumption~\eqref{eq:C-retraction} \revision[0]{for non-trivial retractions}, first consider the \revision[0]{case $\beta=0$, i.e.}
\begin{equation}
    \mcal{L}\pars{R_u(u + g)} \le \mcal{L}\pars{u + g} + \frac{C_{\mathrm{R}}}{2}\norm{g}^2 .
\end{equation}
Although it is non-standard and not easy to guarantee a priori, this condition seems natural because it ensures that the error of the retraction is of higher order (i.e.\ $\mcal{O}(\norm{g}^2)$) than the progress of the linear step (which is of order $\mcal{O}(\norm{g})$).
Assumption~\eqref{eq:C-retraction} is a weaker form of the above condition, which allows for an additional additive perturbation $\beta$.

To provide an example for when assumption~\eqref{eq:C-retraction} is satisfied, we suppose that $\mcal{M}$ is a relatively compact (in $\mathcal{H}$) manifold with positive reach $R > 0$.%
\footnote{An explanation of the reach and its properties, as well as further references can be found in~\cite{trunschke2023convergence}.}
Since $\mcal{L}$ is $L$-smooth, we can assume that it is locally Lipschitz continuous, i.e.\ that for every radius $r$ there exists a constant $\mathrm{Lip}(r)$ such that for all $v,w\in\mcal{H}$ with $\norm{v-w}\le r$
\begin{equation}
    |\mcal{L}(v) - \mcal{L}(w)|
    \le \mathrm{Lip}(r) \norm{v - w} .
\end{equation}
Since $\mcal{M}$ is (relatively) compact, Weierstrass's theorem guarantees that a uniform upper bound $\overline{\mathrm{Lip}}$ for $\mathrm{Lip}_t$ can be found.
This means that $\mcal{L}$ is Lipschitz continuous on $\mcal{M}$ with Lipschitz constant $\overline{\mathrm{Lip}}$.
Now, when the step size is small enough, Proposition~14.2 in~\cite{trunschke2023convergence} ensures that we can find a retraction $R_v$ such that
$$
    \abs{
        \mcal{L}(R_v(v+sg)) - \mcal{L}(v+sg)
    }
    \le \overline{\mathrm{Lip}} \norm{R_v(v+sg) - (v+sg)}
    \le \frac{\overline{\mathrm{Lip}}}{R} s^2\norm{g}^2 .
$$

A natural way to define this retraction is the projection of $v+sg$ onto $\mcal{M}$ (which exists when $s$ is small enough).
For the model class of low-rank tensors of order $D\in\mbb{N}$, 
the exact projection is NP-hard to compute, but a quasi-optimal projection is provided by the HSVD~\cite{grasedyck2010} and satisfies $\norm{R^{\mathrm{HSVD}}_v(v+g) - (v+g)} \le \sqrt{D} \norm{R_v(v+g) - (v+g)}$.
This gives a computable retraction satisfying
$$
    \mcal{L}(R^{\mathrm{HSVD}}_v(v+g))
    \le \mcal{L}(v+g) + \frac{\overline{\mathrm{Lip}}\sqrt{D}}{2R} \norm{g}^2,
$$
where $R$ depends on the singular values of the tensor $v$. 
Since computing the retraction $R_v$ may be very costly, it is often interesting to approximate its exact value by stochastic or iterative methods with a controllable error.
This can be captured by the $\beta$ term in assumption \eqref{eq:C-retraction}, which is possible because most of our theorems do not require a deterministic sequence $(\beta_t)_{t\ge 1}$ but only $(\mbb{E}\bracs*{\beta_t \mid \mcal{F}_t})_{t\ge 1}\in\ell^1$.

\section{Bound for \texorpdfstring{$\boldsymbol\kappa$}{} in the setting of recovery}
\label{app:kappa_bound_reach}

The proof of this result relies on the subsequent basic proposition.
\begin{proposition}[\cite{Federer1959}] \label{prop:nhood_properties}
    Suppose that $\mcal{M}$ is a manifold and that $\mcal{T}_v$ is the tangent space of $\mcal{M}$ at $v\in\mcal{M}$.
    Assume moreover, that $R_v := \operatorname{rch}\pars{\mcal{M}, v} > 0$.
    Then, if $u\in\mcal{M}$ and $g = v - u$ it holds that
    $$
        \norm{g - P_vg} \le \frac{\norm{g}^2}{2R_v} .
    $$
\end{proposition}
A simple, geometric proof of this result can be found in~\cite[Theorem~7.8~(2)]{boissonnat2018geometric} and~\cite[Appendix~B]{trunschke2023convergence}.

\begin{lemma}
    Suppose that $\mcal{M}$ is a differentiable manifold and that $\mcal{T}_v$ is the tangent space of $\mcal{M}$ at $v\in\mcal{M}$.
    Assume moreover, that its reach $R_v := \operatorname{rch}\pars{\mcal{M}, v} > 0$.
    Let $u\in\mcal{M}$ and consider the risk functional $\mcal{L}(v) := \tfrac12\norm{u - v}^2$.
    Then it holds for all $v\in\mcal{M}$ with $\norm{u - v} \le r \le R_v$ that
    $$
        \frac{\norm{(I-P_v)g}}{\norm{P_v g}} \le \frac{r}{R_v} ,
    $$
    where $g := \nabla \mcal{L}\pars{v} = v - u$.
\end{lemma}
\begin{proof}
    Recall that by Proposition~\ref{prop:nhood_properties}
    \begin{equation}
        \norm{(I - P_v)g} \le \frac{\norm{g}^2}{2R_v} .
    \end{equation}
    Since $\norm{g}^2 = \norm{P_vg}^2 + \norm{(I - P_v)g}^2$ and $\norm{(I - P_v)g} \le \norm{g}\le r$, we can write
    \begin{equation}
        \norm{(I - P_v)g}
        \le \frac{\norm{g}^2}{2R_v}
        = \frac{\norm{P_vg}^2 + \norm{(I - P_v)g}^2}{2R_v}
        \le \frac{\norm{P_vg}^2}{2R_v} + \frac{\norm{(I - P_v)g}r}{2R_v} 
    \end{equation}
    and, consequently,
    \begin{equation}
        \norm{(I-P_v)g}
        \le \frac{\norm{P_vg}^2}{2R_v - r}
        \le \frac{\norm{P_vg}^2}{R_v} .
    \end{equation}
    This yields the bound $\frac{\norm{(I-P_v)g}}{\norm{P_vg}} \le \frac{\norm{P_vg}}{R_v} \le \frac{r}{R_v}$ and concludes the proof.
\end{proof}

\include{content/050_stepsize}

\section{Shallow neural networks}
\label{app:shallow_nns}

\subsection{Addition}
It is well-known that the sum of two shallow neural networks $\Phi_{\theta}$ and $\Phi_{\theta'}$ of width $w$ can be represented as another shallow neural network of width $2w$.
To see this, observe that
\begin{align}
    \Phi_{(A_1, b_1, A_0, b_0)}(x) + \Phi_{(A_1', b_1', A_0', b_0')}(x)
    &= A_1 \sigma(A_0x + b_0) + b_1 + A_1' \sigma(A_0'x + b_0') + b_1' \\
    &= \underbrace{\begin{bmatrix} A_1 & A_1' \end{bmatrix}}_{\eqqcolon A_1''}
    \sigma\Bigg(
        \underbrace{\begin{bmatrix}
            A_0 \\ A_0'
        \end{bmatrix}}_{\eqqcolon A_0''} x
        + 
        \underbrace{\begin{bmatrix}
            b_0 \\ b_0'
        \end{bmatrix}}_{\eqqcolon b_0''}
    \Bigg)
    + \underbrace{(b_1 + b_1')}_{\eqqcolon b_1''} \\
    &= \Phi_{(A_1'', b_1'', A_0'', b_0'')} .
\end{align}
Defining the binary operation $\boxplus$ accordingly, we can thus write $\Phi_\theta + \Phi_{\theta'} = \Phi_{\theta \boxplus \theta'}$.

\subsection{Inner products}

Denote by $A_{k*}$ the $k$\textsuperscript{th} row of the matrix $A$ and observe that the $L^2(\rho)$-inner product of two shallow neural networks is given by
\begin{equation}
    \pars{\Phi_\theta, \Phi_{\theta'}}
    =
    \begin{multlined}[t]
        \sum_{k=1}^m \sum_{l=1}^{m'} A_{1,1k} A'_{1,1l} \pars*{\sigma(A_{0,k*}x + b_{0,k}), \sigma(A_{0,l*}'x + b_{0,l}')} \\
        + \sum_{k=1}^m  A_{1,1k} b'_{1} \pars*{\sigma(\vphantom{A_0'}A_{0,k*}x + b_{0,k}), 1} \\
        + \sum_{l=1}^{m'} b_1 A'_{1,1l} \pars*{\sigma(A_{0,l*}'x + b_{0,l}'), 1} .
    \end{multlined}
\end{equation}
To compute $\pars{\Phi_\theta, \Phi_{\theta'}}$, it hence suffices to compute inner products of the form
$$
    \pars*{\sigma(\alpha^\intercal x + \beta), \sigma(\alpha'^\intercal x + \beta')}
    \qquad\text{and}\qquad
    \pars*{\sigma(\vphantom{\alpha'}\alpha^\intercal x + \beta), 1}
$$
for given $\alpha,\alpha'\in\mbb{R}^n$ and $\beta,\beta'\in\mbb{R}$.

Since the function $\sigma\pars{\alpha^\intercal x + \beta}$ varies only in the direction of $\alpha$ and is constant in its orthogonal complement, the inner product $\pars*{\sigma(\vphantom{\alpha'}\alpha^\intercal x + \beta), 1}$ can be recast as a one-dimensional integration problem after a suitable transform of coordinates.
Similarly, the product $\sigma(\alpha^\intercal x + \beta) \sigma(\alpha'^\intercal x + \beta')$ only varies in the two-dimensional subspace $\operatorname{span}(\alpha, \alpha')$ and can thereby be recast as a two-dimensional integration problem.

This allows the computation of norms and distances between shallow neural networks as well as the computation of Gramians of bases of spaces of neural networks.
We can then use these Gramians to find orthonormal bases, compute least squares projections and draw optimal samples.

\paragraph{One-dimensional integrals}
If $\rho$ is the uniform measure on a convex polytope $\Sigma$ and $\sigma$ is the logistic sigmoid $\sigma(x) := (1 + \exp(-x))^{-1}$, explicit formulae for the one-dimensional integrals $\pars*{\sigma(\alpha^\intercal x + \beta), 1}$ are derived in~\cite{Lloyd2020}.
Although we believe these formulae can be extended to the two-dimensional integrals, we will consider a simpler example in the following.

Let $\rho$ denote the probability density function of the standard Gaussian measure on $\mbb{R}^d$ for $d\in\mbb{N}$ chosen appropriately according to context.
Note that for any $\alpha\in\mbb{R}^n$ and $\beta\in\mbb{R}$, we can find an orthogonal matrix $Q\in\mbb{R}^{n\times n}$, such that $Qe_1 \propto \alpha$.
Therefore, the change-of-variables formula then implies
\begin{align}
    \int_{\mathbb{R}^n} \sigma(\alpha^\intercal y + \beta) \rho(y) \dx[y] 
    &= \int_{\mathbb{R}^n} \sigma(\alpha^\intercal Qx + \beta) \rho(Qx) \dx \\
    &= \int_{\mathbb{R}} \sigma(\norm{\alpha}_2 x_1 + \beta) \pars*{\int_{\mbb{R}^{n-1}} \rho(Qx) \dx[(x_2, \ldots, x_n)]} \dx[x_1] \\
    &= \int_{\mathbb{R}} \sigma(\norm{\alpha}_2 x_1 + \beta) \rho(x_1) \dx[x_1] .
\end{align}
This integral is computable by classical one-dimensional quadrature rules such as Gau\ss--Hermite quadrature.

\paragraph{Two-dimensional integrals}

Let again $\rho$ denote the probability density function of the standard Gaussian measure on $\mbb{R}^d$ for any $d\in\mbb{N}$ and let $\alpha,\alpha'\in\mbb{R}^n$ and $\beta,\beta'\in\mbb{R}$.
We can assume that $\alpha$ and $\alpha'$ are linearly independent since if this is not the case, we are in the setting of one-dimensional integrals.
Under this assumption, we can find an orthogonal matrix $Q\in\mbb{R}^{n\times n}$, such that $Qe_1 \propto \alpha$ and $Qe_2 \propto \alpha' - \frac{(\alpha', \alpha)}{\norm{\alpha}_2^2}\alpha$.
Defining $\gamma_1 := \alpha'^\intercal Qe_1$ and $\gamma_2 := \alpha'^\intercal Qe_2$, we can again employ the change-of-variables formula to obtain
\begin{align}
    \int_{\mathbb{R}^n} \sigma(\alpha^\intercal y + \beta) \sigma(\alpha'^\intercal y + \beta') \rho(y) \dx[y]
    &= \int_{\mathbb{R}^n} \sigma(\alpha^\intercal Qx + \beta) \sigma(\alpha'^\intercal Qx + \beta') \rho(Qx) \dx[x] \\
    &= \begin{multlined}[t]
        \int_{\mathbb{R}^2} \sigma(\norm{\alpha}_2 x_1 + \beta) \sigma(\gamma_1 x_1 + \gamma_2 x_2 + \beta') \\
        \cdot\pars*{\int_{\mbb{R}^{n-2}} \rho(Qx) \dx[(x_3,\ldots,x_n)]} \dx[(x_1, x_2)]
    \end{multlined} \\
    &= \int_{\mathbb{R}^2} \sigma(\norm{\alpha}_2 x_1 + \beta) \sigma(\gamma_1 x_1 + \gamma_2 x_2 + \beta') \rho(x_1, x_2) \dx[(x_1, x_2)] .
\end{align}
This integral can be computed by classical two-dimensional Gau\ss--Hermite quadrature.

\subsection{Local linearisation}

Analogously to the definition of the tangent space of embedded manifolds, one can define the local linearisation of $\mcal{M}$ around the network $\Phi_\theta$ as
$$
    \mcal{T}_{\Phi_{\theta}} := \operatorname{span}\braces{\partial_{\theta_j} \Phi_\theta \,:\, j=1,\ldots, d} ,
$$
where $\theta\in \mbb{R}^{1\times m} \times \mathbb{R} \times \mathbb{R}^{m\times n} \times \mathbb{R}^m$ is identified with its vectorisation $\operatorname{vec}(\theta)\in\mathbb{R}^{d}$ with $d = {m(n+2) + 1}$.
Since the space $\mcal{T}_{\Phi_\theta}$ is spanned by the partial derivatives
\begin{align}
    \partial_{A_{1,1j}} \Phi_{(A_1, b_1, A_0, b_0)}
    &= \sigma(A_0 x + b_0)_j, \\
    \partial_{b_{1}} \Phi_{(A_1, b_1, A_0, b_0)}
    &= 1, \\
    \partial_{A_{0,ij}} \Phi_{(A_1, b_1, A_0, b_0)}
    &= A_{1,1i} \sigma'(A_0 x + b_0)_i x_j, \\
    \partial_{b_{0,i}} \Phi_{(A_1, b_1, A_0, b_0)}
    &= A_{1,1i} \sigma'(A_0 x + b_0)_i ,
\end{align}
it is at most $d$-dimensional and contains $\Phi_\theta$.
Note, however, that the important derivative $\partial_{A_{0,ij}}\Phi_{(A_1,b_1,A_0,b_0)}$ can not be expressed directly as a neural network.
This is relevant because we want to compute the Gramian according to the theory presented in the previous section.
(Note that $\partial_{b_{0,i}} \Phi_{(A_1, b_1, A_0, b_0)}$ is not a problem because it still exhibits the ridge-structure that is necessary for efficient integration.)

There are two ways to bypass this problem.
\begin{enumerate}
    \item In the Gaussian setting, we can use the Gaussian integration by parts formula (also known as Isserlis' theorem or Wick's probability theorem) to get rid of the additional $x_j$-factor.
    \item We can approximate the function $\partial_{A_{0,ij}} \Phi_{(A_1, b_1, A_0, b_0)}$ by a shallow neural network.
    A straightforward way to do this is to replace the differential with a finite difference approximation with step size $h > 0$.
\end{enumerate}
For the sake of simplicity, we choose to opt for the second option.
We define the elementary matrices $(E_{ij})_{kl} := \delta_{ik}\delta_{jl}$ and propose to approximate
\begin{align}
    \partial_{A_{0,ij}} \Phi_{(A_1, b_1, A_0, b_0)}
    &\approx \tfrac{A_{1,1i}}{h}\left[\sigma((A_0 + E_{ij}h) x + b_0)_i - \sigma(A_0 x + b_0)_i\right] \\
    \partial_{b_{0,i}} \Phi_{(A_1, b_1, A_0, b_0)}
    &\approx \tfrac{A_{1,1i}}{h} \left[\sigma(A_0 x + (b_0 + e_ih))_i - \sigma(A_0 x + b_0)_i\right] .
\end{align}
The corresponding linearisation space can then be written as
\begin{align}
    \mcal{T}_{\Phi_{\theta}}^h
    :=\ &\operatorname{span}\braces*{\sigma\pars{A_0x + b_0}_j \,:\, j=1,\ldots,m} \\
    +\  &\operatorname{span}\braces*{1} \\
    +\  &\operatorname{span}\braces*{\sigma\pars{(A_0 + E_{ij}h)x + b_0}_i \,:\, i=1,\ldots,n,\ j=1,\ldots,m} \\
    +\  &\operatorname{span}\braces*{\sigma\pars{A_0x + (b_0 + e_jh)}_i \,:\, j=1,\ldots,m} .
\end{align}

\subsection{Optimal sampling}

Let $\braces{\phi_{j}}_{j=1,\ldots,d}$ be a generating system of the space $\mcal{T}_t$.
Given the Gramian $G$ of this generating system, we can compute an orthonormal basis $\braces{\psi_k}_{k=1,\ldots,r}$ with $r\le d$ and
$$
    \psi_k := \sum_{j=1}^d c_{kj} \phi_j .
$$
Recall that the optimal sampling density of the space $\operatorname{span}\braces{\psi_k}_{k=1,\ldots,r}$ is given by
$$
    p(x) := \frac1r \sum_{k=1}^r p_k(x)
    \qquad\text{with}\qquad
    p_k(x) := \psi_k^2(x) \rho(x) .
$$
To draw samples from this distribution, we can employ a combination of mixture sampling and sequential conditional sampling.
First, we draw an index $k\in\bracs{r}$ uniformly at random.
Then, we draw a sample from $p_k$ by decomposing
$$
    p_k(x) = p_k(x_n \mid x_{n-1},\ldots, x_1) \cdots p_k(x_2 \mid x_1) p_k(x_1) .
$$
To compute the marginals, observe that
\begin{align}
    p_k(x_l \mid x_{l-1}, \ldots, x_1)
    &= \int \psi_k^2(x)\rho(x) \dx[(x_n, \ldots, x_{l+1})] \\
    &= \int \pars*{\sum_{j=1}^d c_{kj} \phi_j(x)}^2 \rho(x) \dx[(x_n, \ldots, x_{l+1})] \\
    &= \sum_{i,j=1}^d c_{ki}c_{kj} \int \phi_i(x) \phi_j(x) \rho(x) \dx[(x_n, \ldots, x_{l+1})]
\end{align}
is a sum of (parametric in $x_l,\ldots,x_1$) two-dimensional integrals, as discussed above.
This can be implemented in an efficient, parallelised fashion, taking arrays of sample points $x_{l-1}, \ldots, x_1 \in \mbb{R}^N$ and of evaluation points $x_l\in\mbb{R}^D$ and returning the matrix of probabilities $p_k(x_l\mid x_{l-1}, \ldots, x_1)\in\mbb{R}^{D\times N}$.
This allows us to approximate the (one-dimensional) cumulative distribution function of $p_k(x_l\mid x_{l-1}, \ldots, x_1)$, which can be used for inverse transform sampling of $x_l$.

\subsection{Retraction error bounds}

\begin{revisione}[0]
Consider the minimisation of the least squares loss
$$
    \mathcal{L}\pars{v} := \tfrac12\norm{u-v}^2
$$
for some $u\in\mcal{H}$.
To define the retraction, recall that $u_{t+1} = R_{t}(\bar{u}_{t+1})$ and $\bar{u}_{t+1} = u_t - s_tP_{t}^n g_t$ with $g_t = \nabla\mcal{L}(u_t)$.
Moreover, we can express $u_t = \Phi_{\theta_t}$ and $P_t^n g_t = \sum_{k=1}^d (\vartheta_t)_k (\partial_{(\theta_t)_k} \Phi_{\theta_t})$.
Therefore, by Taylor's theorem,
$$
    \Phi_{\theta_t} + s_t P_t^ng_t = \Phi_{\theta_t + s_t\vartheta_t} + \mathcal{O}(s_t^2)
$$
for sufficiently small $s_t$.
We thus define the retraction operator
$$
    R_t(u_t + s_tP_t^ng_t) := \Phi_{\theta_t + s_t\vartheta_t} .
$$
To ensure assumption~\eqref{eq:C-retraction}, we employ the local Lipschitz continuity of $\mathcal{L}$.
To do this, we define for every radius $r>0$ the constant $\mathrm{Lip}_t(r)$ such that for all $v,w\in\mcal{H}$ with $\norm{v-u_t}\le r$ and $\norm{w-u_t}\le r$,
\begin{equation}
\label{eq:lipr}
    |\mcal{L}(v) - \mcal{L}(w)|
    \le \mathrm{Lip}_t(r) \norm{v - w} .
\end{equation}
A bound for $\mathrm{Lip}_t(r)$ is provided in the subsequent lemma.

\begin{lemma}
\label{lem:lipr}
   For any $r>0$ it holds that $\mathrm{Lip}_t(r) \le \norm{\nabla\mcal{L}(u_t)} + Lr$.
\end{lemma}
\begin{proof}
    Suppose that $v, w\in\mcal{H}$ satisfy $\norm{v-u_t}\le r$ and $\norm{w-u_t}\le r$ and let $f:[0,1]\to\mbb{R}$ be defined by $f(t) = \mcal{L}(\gamma(t))$ with $\gamma(t) = tv + (1-t)w$.
    The mean value theorem guarantees that there exists $\eta\in(0,1)$ such that 
    $$
        \abs{\mcal{L}(v) - \mcal{L}(w)}
        = \abs{f(1) - f(0)}
        = f'(\eta)
        = \pars{\nabla\mcal{L}(\gamma(\eta)), v - w}
        \le \norm{\nabla\mcal{L}(\gamma(\eta))}\norm{v - w} .
    $$
    Since $\mcal{L}$ is $L$-smooth, the reverse triangle inequality implies $\norm{\nabla\mcal{L}(\gamma(\eta))} \le \norm{\nabla\mcal{L}(u_t)} + L\norm{u_t - \gamma(\eta)} \le \norm{\nabla\mcal{L}(u_t)} + Lr$.
    This implies
    \begin{align}
        \abs{\mcal{L}(v) - \mcal{L}(w)}
        &\le \pars{\norm{\nabla\mcal{L}(u_t)} + Lr} \norm{v - w} .
    \end{align}
\end{proof}

\noindent
Inserting $v = u_{t+1}$ and $w = \bar{u}_{t+1}$ into equation~\eqref{eq:lipr} and using Lemma~\ref{lem:lipr} with $\norm{\nabla\mcal{L}(v)} = \sqrt{2\mcal{L}(v)}$, we can write
\begin{align}
    |\mathcal{L}(u_{t+1}) - \mathcal{L}(\bar{u}_{t+1})|
    &\le \mathrm{Lip}_t(\max\braces{\norm{u_{t+1} - u_{t}}, \norm{\bar{u}_{t+1} - u_{t}}}) \norm{u_{t+1} - \bar{u}_{t+1}} \\
    &= (\sqrt{2\mathcal{L}(u_{t})} + \max\braces{\norm{u_{t+1} - u_{t}}, \norm{\bar{u}_{t+1} - u_{t}}}) \norm{u_{t+1} - \bar{u}_{t+1}} \\
    &\le (\sqrt{2\mathcal{L}(u_{t})} + \norm{u_{t+1} - \bar{u}_{t+1}} + s_t\norm{P_t^ng_t}) \norm{u_{t+1} - \bar{u}_{t+1}} .
\end{align}
Since $\mcal{L}(u_t) = \tfrac12 \norm{u - u_t}^2 \approx \tfrac12 \norm{u - u_t}_n^2$, we define the estimate
$$
    \mathcal{L}(u_t) \approx \lambda_{t} := \tfrac12\lambda_{t-1} + \tfrac14 \norm{u - u_t}_n^2 ,
$$
where the initial value $\lambda_0$ is estimated once for the initial guess.
Using this estimate,
\begin{equation}
    |\mathcal{L}(u_{t+1}) - \mathcal{L}(\bar{u}_{t+1})|
    \le (\sqrt{2\lambda_t} + \varepsilon(s_t) + s_t\norm{P_t^ng_t}) \varepsilon(s_t) ,
\end{equation}
with $\varepsilon(s_t) := \norm{u_{t+1} - \bar{u}_{t+1}}$ 
the retraction error for step size $s_t$.
Assumption~\eqref{eq:C-retraction} is thus satisfied with $C_{\mathrm{R}} = 0$ and 
$$
    \beta_t(s_t) := \mathrm{Lip}_t(\sqrt{2\lambda_t} + \varepsilon(s_t) + s_t\norm{P_t^ng_t}) \varepsilon(s_t) .
$$
\end{revisione}

\end{document}